\documentclass[reqno]{amsart}
\usepackage{setspace,tikz,xcolor,mathrsfs,listings,multicol}
\usepackage{amssymb}
\usepackage{rotating}
\usepackage[vcentermath]{youngtab}
\usepackage{ytableau}
\usepackage{enumerate}
\usepackage[all,cmtip]{xy}
\usetikzlibrary{arrows,matrix,positioning}
\usepackage{soul}
\tikzset{>=stealth',
  head/.style = {fill = white, text=black},
  plaque/.style = {draw, rectangle, minimum size = 10mm}, 
  pil/.style={->,thick},
  junct/.style = {draw,circle,inner sep=0.5pt,outer sep=0pt, fill=black}
  }

\usepackage[colorlinks=true, pdfstartview=FitV, linkcolor=blue, citecolor=blue, urlcolor=blue]{hyperref}

\newcommand{\g}{\mathfrak{g}}
\newcommand{\G}{\mathfrak{G}}
\newcommand{\xx}{\mathbf{x}}
\newcommand{\HH}{\mathcal{H}}
\newcommand{\inner}[2]{\left\langle #1, #2 \right\rangle}
\newcommand{\iso}{\cong}
\newcommand{\bbb}{\mathsf{b}}
\newcommand{\poly}{\mathrm{Poly}}
\newcommand{\sgn}{\operatorname{sgn}}

\newcommand{\fsl}{\mathfrak{sl}}
\newcommand{\sym}{S}
\newcommand{\imapsto}[1]{\; \overset{b_{#1}}{\longmapsto} \;}  

\DeclareMathOperator{\Supp}{Supp} 
\DeclareMathOperator{\wt}{wt} 
\DeclareMathOperator{\excess}{ex} 
\DeclareMathOperator{\ssyt}{SST} 
\DeclareMathOperator{\svssyt}{SV} 
\DeclareMathOperator{\EYD}{EYD} 
\DeclareMathOperator{\GL}{GL} 
\DeclareMathOperator{\Gr}{Gr} 
\DeclareMathOperator{\rect}{rect} 
\DeclareMathOperator{\Inf}{Inf} 

\newcommand{\ZZ}{\mathbb{Z}}

\newcommand{\CC}{\mathbb{C}}

\newcommand{\bon}{\overline{1}}
\newcommand{\btw}{\overline{2}}
\newcommand{\bth}{\overline{3}}

\newcommand{\bplus}{{\color{blue}+}}
\newcommand{\bminus}{{\color{blue}-}}

\newcommand{\cbone}{{\color{darkred}\mathbf{1}}}
\newcommand{\cbtwo}{{\color{darkred}\mathbf{2}}}
\newcommand{\cbthr}{{\color{darkred}\mathbf{3}}}

\usepackage{listings}
\lstdefinelanguage{Sage}[]{Python}
{morekeywords={False,sage,True},sensitive=true}
\definecolor{dblackcolor}{rgb}{0.0,0.0,0.0}
\definecolor{dbluecolor}{rgb}{0.01,0.02,0.7}
\definecolor{dgreencolor}{rgb}{0.2,0.4,0.0}
\definecolor{dgraycolor}{rgb}{0.30,0.3,0.30}

\usepackage{xparse}

\makeatletter
\protected\def\specialmergetwolists{%
  \begingroup
  \@ifstar{\def\cnta{1}\@specialmergetwolists}
    {\def\cnta{0}\@specialmergetwolists}%
}
\def\@specialmergetwolists#1#2#3#4{%
  \def\tempa##1##2{%
    \edef##2{%
      \ifnum\cnta=\@ne\else\expandafter\@firstoftwo\fi
      \unexpanded\expandafter{##1}%
    }%
  }%
  \tempa{#2}\tempb\tempa{#3}\tempa
  \def\cnta{0}\def#4{}%
  \foreach \x in \tempb{%
    \xdef\cnta{\the\numexpr\cnta+1}%
    \gdef\cntb{0}%
    \foreach \y in \tempa{%
      \xdef\cntb{\the\numexpr\cntb+1}%
      \ifnum\cntb=\cnta\relax
        \xdef#4{#4\ifx#4\empty\else,\fi\x#1\y}%
        \breakforeach
      \fi
    }%
  }%
  \endgroup
}
\makeatother

\usepackage{xparse}
\DeclareDocumentCommand\rpp{ m m g }{
	\foreach \x [count=\s from 1] in {#1}{
	        {\ifnum\s=1
	                \draw (0,-\s)--(\x,-\s);
	                \fi}
	   \draw (0,-\s-1) to (\x,-\s-1);
	   \foreach \y in {0, ..., \x} {\draw (\y,-\s)--(\y,-\s-1);}
	}
	\specialmergetwolists{/}{#1}{#2}\ziplist
	\foreach \x/\y [count=\yi from 1] in \ziplist{
	    \node[anchor=west,font=\small] at (\x,-\yi - .5) {$\y$};
	}
	\IfValueT {#3}
	{\foreach \z [count=\zi from 1] in {#3} {\node[anchor=east,font=\small] at (0,-\zi - .5) {$\z$};}}
	{}
}

\definecolor{darkred}{rgb}{0.7,0,0} 
\newcommand{\defn}[1]{{\color{darkred}\emph{#1}}} 

\usepackage[colorinlistoftodos]{todonotes}

\theoremstyle{plain}
\newtheorem{thm}{Theorem}[section]
\newtheorem{lemma}[thm]{Lemma}
\newtheorem{conj}[thm]{Conjecture}
\newtheorem{prop}[thm]{Proposition}
\newtheorem{cor}[thm]{Corollary}
\theoremstyle{definition}
\newtheorem{dfn}[thm]{Definition}
\newtheorem{examplex}[thm]{Example}
\newenvironment{ex}
  {\pushQED{\qed}\examplex}
  {\popQED\endexamplex}
\newtheorem{remark}[thm]{Remark}
\newtheorem{prob}[thm]{Open Problem}
\numberwithin{equation}{section}

\begin{document}
\title[Symmetric Grothendieck crystals]{Crystal structures for symmetric Grothendieck polynomials}

\author[C.~Monical]{Cara Monical}
\address[C.~Monical]{Department of Mathematics, University of Illinois at Urbana-Champaign, Urbana, IL 61801, USA}
\email{caramonicalmath@gmail.com}
\urladdr{https://faculty.math.illinois.edu/~cmonica2/}

\author[O.~Pechenik]{Oliver Pechenik}
\address[O.~Pechenik]{Department of Mathematics, University of Michigan, Ann Arbor, MI 48109, USA}
\email{pechenik@umich.edu}
\urladdr{http://www-personal.umich.edu/~pechenik/}

\author[T.~Scrimshaw]{Travis Scrimshaw}
\address[T.~Scrimshaw]{School of Mathematics and Physics, The University of Queensland, St.\ Lucia, QLD 4072, Australia}
\email{tcscrims@gmail.com}
\urladdr{https://people.smp.uq.edu.au/TravisScrimshaw/}


\thanks{OP was partially supported by a National Science Foundation Graduate Research Fellowship and a National Science Foundation Mathematical Sciences Postdoctoral Research
Fellowship (\#1703696).
TS was partially supported by the National Science Foundation RTG grant NSF/DMS-1148634 and the Australian Research Council DP170102648.}

\begin{abstract}
The symmetric Grothendieck polynomials representing Schubert classes in the $K$-theory of Grassmannians are generating functions for semistandard set-valued tableaux. We construct a type $A_n$ crystal structure on these tableaux. This crystal yields a new combinatorial formula for decomposing symmetric Grothendieck polynomials into Schur polynomials. For single-columns and single-rows, we give a new combinatorial interpretation of Lascoux polynomials (K-analogs of Demazure characters) by constructing a K-theoretic analog of crystals with an appropriate analog of a Demazure crystal. We relate our crystal structure to combinatorial models using excited Young diagrams, Gelfand--Tsetlin patterns via the $5$-vertex model, and biwords via Hecke insertion to compute symmetric Grothendieck polynomials.
\end{abstract}

\maketitle

\section{Introduction}
\label{sec:introduction}

The set of $k$-dimensional linear subspaces in $\CC^n$ is known as the Grassmannian $X = \Gr_k(\CC^n)$.  Grassmannians are naturally smooth projective varieties and have been well studied from numerous viewpoints; for background, see, \textit{e.g.},~\cite{Fulton,Gillespie,Kleiman72,Pechenik:thesis,Schubert79} and references therein. One approach to studying the Grassmannian is through its cohomology ring, where a basis of cohomology classes appears as the Poincar\'e duals to the Schubert varieties $X_{\lambda}$ that decompose $X$ into a cell complex. By corresponding the Schubert classes with those Schur polynomials $s_{\lambda}$ whose defining partition $\lambda$ fits inside a $k \times (n-k)$ rectangle, the cohomology $H^*(X)$ is isomorphic to a projection of the algebra of symmetric functions. That is to say, the Schubert structure coefficients of $H^*(X)$ are either Littlewood--Richardson coefficients or $0$. This identification permits the application of many combinatorial tools to the study of Grassmannian Schubert calculus.

Modern Schubert calculus strives for a richer understanding of the Grassmannian through use of generalized cohomology theories, such as K-theory. In the K-theory ring $K(X)$ of algebraic vector bundles over $X$, there is a canonical basis of Schubert classes given by the structure sheaves of the Schubert varieties. As in the usual cohomology, these Schubert classes may be represented by certain polynomials; in this case the symmetric Grothendieck polynomials $\G_{\lambda}$. The combinatorics of K-theoretic Grassmannian Schubert calculus have also been well-studied by numerous authors using diverse tools; \textit{e.g.},~\cite{Buch02,BKSTY08,BS16,GMPPRST16,FK94,IIM17,IS14,Lenart00,LMS16,MS13,MS14,PP16,PS18,PY:genomic,PY18,TY09,Vakil06,WZJ16,Yel17} and references therein.

Remarkably, the product $\G_{\mu} \G_{\lambda}$ of symmetric Grothendieck polynomials  (even in an infinite number of variables) can be written in a unique way as a \emph{finite} sum of other symmetric Grothendieck polynomials $\G_{\nu}$.
There are many known positive combinatorial rules to compute the Schubert structure coefficients $C_{\lambda\mu}^{\nu}$, many of which are natural K-theory analogs of rules for Littlewood--Richardson coefficients. 
In particular, $C_{\lambda\mu}^{\nu}$ is the number of set-valued tableaux of weight $\nu$ with skew shape given by the shapes $\lambda$ and $\mu$ touching at their corners such that a particular reading word is Yamanouchi~\cite{Buch02}. Other such rules have been given in \cite{BS16,LMS16,PP16,PY:genomic,TY09,Vakil06}.

One important interpretation of Schur functions is as characters of the special linear Lie algebra $\fsl_n$.
To study the representation theory of a Kac--Moody Lie algebra $\g$ using combinatorics, M.~Kashiwara introduced the $q \to 0$ limit of representations of the quantum group $U_q(\g)$ in the seminal papers~\cite{K90, K91}. As $q$ corresponds to absolute temperature of certain physical systems and $q \to 0$ corresponds to taking temperature to absolute zero, he coined these bases crystal bases. For $\g = \fsl_n$, this construction gave a natural (algebraic) interpretation to many well-known combinatorial constructions. This includes that coplactic operators are crystal operators (see, \textit{e.g.},~\cite[Ch.~8]{BS17}), Sch\"utzenberger's evacuation involution~\cite{Sch72} is the $q \to 0$ limit of the Lusztig involution on the quantum group~\cite{Lenart07}, and that tableau promotion reflects a Dynkin diagram automorphism for the corresponding affine type $\widehat{\fsl}_n$~\cite{Shimozono02}.

Our aim is to apply M.~Kashiwara's crystal theory to the study of symmetric Grothendieck polynomials.\footnote{We will work in the slight extra generality of the Fomin--Kirillov $\beta$-deformations of symmetric Grothendieck polynomials~\cite{FK94}; geometrically, this corresponds to moving from ordinary K-theory into connective K-theory~\cite{Hudson}.} By work of A.~Buch \cite{Buch02}, $\G_\lambda$ is a generating function for the set $\svssyt^n(\lambda)$ of semistandard set-valued tableaux of shape $\lambda$ with entries at most $n$.
Our main result shows that this combinatorial set carries a $U_q(\fsl_n)$-crystal structure in the sense that it is isomorphic to a direct sum of the crystals $B(\mu)$ of irreducible highest weight representations. 
This implies we can write a symmetric Grothendieck polynomial $\G_{\lambda}$ as a positive sum of Schur functions $s_{\mu}$, where the multiplicities $M_{\lambda}^{\mu}$ are given by counting highest weight (\textit{i.e.} Yamanouchi) semistandard set-valued tableaux. 
C.~Lenart~\cite[Thm.~2.2]{Lenart00}  has given a different combinatorial formula for the multiplicities $M_{\lambda}^{\mu}$ in terms of certain flagged increasing tableaux. Using our crystal  structure and the ``uncrowding'' operation of~\cite{Buch02} (see also~\cite{RTY18}), we further obtain a new proof of Lenart's formula.

The goal of this project was to construct a K-theoretic analog of (combinatorial) crystal theory. Indeed, we believe there are additional K-theory crystal operators such that when added to our $U_q(\fsl_n)$-crystal structure on $\svssyt^n(\lambda)$, we obtain a connected component such that the characters of the K-theory analog of a Demazure crystal are equal to the so-called Lascoux polynomials~\cite{Kirillov:notes,Lascoux01,Monical16,MPS18,RY15}, the K-theory analog of the well-studied key polynomials (or Demazure characters for type $A_n$)~\cite{AS18:Demazure,AS18:Kohnert,Demazure74,Kohnert91,LS90,Mason09,RS95}. Towards this goal, when $\lambda$ is a single row or column, we construct such a K-theory Demazure crystal yielding a new combinatorial interpretation of the associated Lascoux polynomials.

Furthermore, we also expect a suitable notion of a tensor product such that the connected components are uniquely determined by what we call minimal highest weight and so that the multiplicity of $\svssyt^n(\nu)$ in $\svssyt^n(\mu) \otimes \svssyt^n(\lambda)$ gives $C_{\lambda\mu}^{\nu}$. We expect this structure to connect with the column insertion given in~\cite{Buch02} and to provide a K-theory analog of jeu de taquin (K-jdt) on semistandard set-valued tableaux. We make some progress in this direction by introducing a K-jdt for semistandard set-valued tableaux and showing that it is a $U_q(\fsl_n)$-crystal isomorphism. We note that such a tensor product rule would suggest a K-theory analog of the quantum group $U_q(\fsl_n)$ (or would be a consequence of such an algebra), where the characters of certain irreducible representations are symmetric Grothendieck polynomials, and would lead to an algebraic interpretation of the Lascoux polynomials.

Now, we discuss how our results connect with some other related work and how these connections may be useful in the effort to construct K-crystal structures.
Another set of combinatorial objects, known as excited Young diagrams, used to compute symmetric Grothendieck polynomials were independently introduced by T.~Ikeda--H.~Naruse~\cite{IN09} and V.~Kreiman~\cite{Kreiman05} for the usual ($T$-equivariant, where $T$ is a maximal torus) cohomology classes of Schubert varieties and then extended to the K-theory setting in~\cite{GK15,IN13} (see also, \cite{KY04}).
Excited Young diagrams are convenient for computing localizations of the equivariant Schubert classes to torus fixed-points (on set-valued tableaux, this is obtained by flagging conditions) and without localizations appeared in~\cite{KMY09} under guise of reduced pipe dreams.
We show that our crystal structure translated to excited Young diagrams under the natural bijection is a restriction of the elementary excitations of~\cite{IN09} (the Type~$1$ excitations of~\cite{GK15}) when the crystal operator acts in the usual way on semistandard tableaux. We require a modified elementary excitation in order to obtain the full crystal structure on semistandard set-valued tableaux; however, this does not change the resulting set, even when imposing the localization.

Another approach to construct $\G_{\lambda}$ is by using the $5$-vertex model~\cite{GK17,MS13,MS14}.
Configurations of the $5$-vertex model with certain boundary conditions are in natural bijection with Gelfand--Tsetlin (GT) patterns.
This gives an expression for $\G_{\lambda}$ as a sum over GT patterns of a product of single variable Grothendieck polynomials (corresponding to the rows of the GT pattern).
Using the single variable symmetric Grothendieck polynomial, we add a marking to the GT pattern, and thus can write a $\G_{\lambda}$ as a sum over these marked GT patterns.
On the other hand, since GT patterns are naturally in bijection with semistandard tableaux, we can impose a natural crystal structure on the configurations of the $5$-vertex model, as was described explicitly in~\cite{EV17}.
This gives a ``coarse'' crystal structure of $B(\lambda)$ that can compute $\G_{\lambda}$ in analogy to the Tokuyama formulas for Whittaker functions (see, \textit{e.g.},~\cite{BBF11,BBCFG12}).

Let $\G_w$ denote the stable Grothendieck corresponding to a permutation $w$, a K-analog of the Stanley symmetric function $F_w$~\cite{Stanley84} (so $\G_w$ at $\beta = 0$ is $F_w$). One may write $\G_w$ as a sum over decreasing factorizations of words equivalent to $w$ in the $\beta$-deformed $0$-Hecke monoid~\cite{BKSTY08, Lam06}, where at $\beta = 0$ we obtain the nilCoxeter algebra and can express $F_w$ as decreasing factorizations of \emph{reduced} words of $w$. In~\cite{BKSTY08}, Hecke insertion was introduced to relate these decreasing factorizations with semistandard set-valued tableaux via a generalization of Edelman--Greene insertion~\cite{EG87} (for the general bi-word case, see, \textit{e.g.},~\cite[Thm.~4.11]{MS16}).
In~\cite{MS16}, (cyclically) decreasing factorizations of reduced words to compute $F_w$ were given a crystal structure that corresponds to the natural tableaux crystal structure under Edelman--Greene insertion.
To lift to $\G_w$ under Hecke insertion, we again we have to make a modification, but unlike in~\cite{MS16}, we provide an example exhibiting a change that involves more than two decreasing factors. A direct crystal structure on decreasing factorizations would yield a new combinatorial description of the decomposition given by~\cite{FG98} of $\G_w$ into Schur functions as a consequence of our main theorem.

This paper is organized as follows.
In Section~\ref{sec:background}, we recall the necessary background.
In Section~\ref{sec:crystal_set_valued}, we prove our main result: a $U_q(\fsl_n)$-crystal structure on semistandard set-valued tableaux, as well as a new proof of Lenart's Schur expansion formula.
Section~\ref{sec:diagrams} relates our crystal structure to excited Young diagrams and (marked) Gelfand--Tsetlin patterns.
In Section~\ref{sec:hecke}, we compute some examples of our crystal structure on biwords via Hecke insertion.
In Section~\ref{sec:K_jdt}, we define a jeu de taquin on semistandard set-valued tableaux.
In Section~\ref{sec:K_theory}, we construct a K-theory analog of a crystal structure on single-row and single-column semistandard set-valued tableaux and discuss extensions to the general case.

\subsection*{Acknowledgements}

The authors would like to thank Anders Buch, Zachary Hamaker, Rebecca Patrias, Kohei Motegi, Hiraku Nakajima, Vic Reiner, Kazumitsu Sakai, Alexander Yong, and Paul Zinn-Justin for useful discussions.
The authors would like to thank Darij Grinberg, Takeshi Ikeda, Cristian Lenart, Pavlo Pylyavskyy, and Alexander Yong for comments on an earlier draft of this manuscript.
The authors would like to thank Jianping Pan, Wencin Poh, and Anne Schilling for useful discussions and pointing out some typos and errors on an earlier draft of this manuscript.
TS would like to thank the University of Illinois at Urbana--Champaign for its hospitality in December 2015, where this work began.
Part of this work was done while OP was at Rutgers University and TS was at the University of Minnesota.
This work benefited from computations using \textsc{SageMath}~\cite{sage, combinat}.

\section{Background}
\label{sec:background}

In this section, we give background on crystals using semistandard tableaux, Grothendieck polynomials, and the associated combinatorics.

\subsection{Crystals}

Let $\fsl_n$ be the type $A_{n-1}$ Lie algebra (the Lie algebra of traceless $n \times n$ matrices over $\CC$) with index set $I = \{1, 2, \dotsc, n-1\}$, simple roots $\{\alpha_i \mid i \in I\}$, simple coroots $\{h_i \mid i \in I\}$, and fundamental weights $\{\Lambda_i \mid i \in I\}$.
Let $Q = \operatorname{span}_{\ZZ} \{ \alpha_i \mid i \in I \}$ be the root lattice, $Q^{\vee} = \operatorname{span}_{\ZZ} \{ h_i \mid i \in I\}$ be the coroot lattice, and $P = \operatorname{span}_{\ZZ} \{\Lambda_i \mid i \in I\}$ be the weight lattice. We denote the set of dominant weights by $P^+$.
Let $C = (C_{ij})_{i,j \in I}$ be the Cartan matrix, that is the matrix of the isomorphism $Q \to P$ given by
\[
C = \begin{bmatrix}
2 & -1 & & & & \\
-1 & 2 & -1 & & & \\
& \ddots & \ddots & \ddots & &\\
& & -1 & 2 & -1 \\
& & & -1 & 2
\end{bmatrix}
\]
and the canonical pairing $\inner{\cdot}{\cdot} \colon P \times Q^{\vee} \to \ZZ$ given by $\inner{\alpha_i}{h_j} = C_{ij}$. 
Let $U_q(\fsl_n)$ denote the corresponding Drinfeld--Jimbo quantum group.
Let $\sym_n$ denote the symmetric group on $\{1, \dotsc, n\}$ with simple transpositions $\{s_i \mid 1 \leq i < n \}$, where $s_i$ interchanges $i$ and $i+1$.

An \defn{abstract $U_q(\fsl_n)$-crystal} is a nonempty set $B$ together with maps
\begin{align*}
e_i, f_i & \colon B \to B\sqcup\{0\},
\\
\wt & \colon B \to P, 
\end{align*}
which satisfy, for all $i \in I$, the properties
\begin{enumerate}
\item $f_i b = b'$ if and only if $b = e_i b'$ for all $b,b' \in B$;
\item $\wt(f_i b) = \wt(b) - \alpha_i$ for all $b \in B$ such that $f_i b \neq 0$;
\item $\varphi_i(b) = \varepsilon_i(b) + \inner{h_i}{\wt(b)}$;
\end{enumerate}
where $\varepsilon_i, \varphi_i \colon B \to \ZZ_{\geq 0}$ are defined by
\begin{align*}
\varepsilon_i(b) & = \max \{k \mid e_i^k b \neq 0\},
\\ \varphi_i(b) & = \max \{k \mid f_i^k b \neq 0\}.
\end{align*}
\begin{remark}
We define our abstract $U_q(\fsl_n)$-crystals with the additional condition that they are regular (sometimes called seminormal) compared to the standard general axioms. For more details on the general case, we refer the reader to \cite{BS17,K91}.
\end{remark}

The maps $e_i$ and $f_i$, for $i \in I$, are called the \defn{crystal operators} or \defn{Kashiwara operators}.
We can express an entire $i$-string through an element $b \in B$ diagrammatically by
\[
e_i^{\max}b \overset{i}{\longrightarrow}
\cdots \overset{i}{\longrightarrow}
e_i^2 b \overset{i}{\longrightarrow}
e_ib \overset{i}{\longrightarrow}
b \overset{i}{\longrightarrow}
f_ib \overset{i}{\longrightarrow}
f_i^2 b \overset{i}{\longrightarrow}
\cdots \overset{i}{\longrightarrow}
f_i^{\max}b,
\]
where $e_i^{\max} b := e_i^{\varepsilon_i(b)} b$ and $f_i^{\max} b := f_i^{\varphi_i(b)} b$.

An element $u \in B$ is called \defn{highest weight} if $e_i u = 0$ for all $i \in I$. An abstract $U_q(\fsl_n)$-crystal is called a \defn{highest weight crystal} if there exists a highest weight element $u \in B$ and, for each $b \in B$, a finite sequence $(a_1, a_2, \dotsc, a_{\ell})$ such that $b = f_{a_1} f_{a_2} \cdots f_{a_{\ell}} u$.

Let $B_1$ and $B_2$ be two abstract $U_q(\fsl_n)$-crystals.  A \defn{crystal morphism} $\psi \colon B_1 \to B_2$ is a map $B_1\sqcup\{0\} \to B_2 \sqcup \{0\}$ such that
\begin{enumerate}
\item $\psi(0) = 0$;
\item if $b \in B_1$ and $\psi(b) \in B_2$, then $\wt(\psi(b)) = \wt(b)$, $\varepsilon_i(\psi(b)) = \varepsilon_i(b)$, and $\varphi_i(\psi(b)) = \varphi_i(b)$;
\item for $b \in B_1$, we have $\psi(e_i b) = e_i \psi(b)$ provided $\psi(e_ib) \neq 0$ and $e_i\psi(b) \neq 0$;
\item for $b\in B_1$, we have $\psi(f_i b) = f_i \psi(b)$ provided $\psi(f_ib) \neq 0$ and $f_i\psi(b) \neq 0$.
\end{enumerate}
A crystal morphism $\psi$ is called \defn{strict} if $\psi$ commutes with $e_i$ and $f_i$ for all $i \in I$.  Moreover, a crystal morphism $\psi \colon B_1 \to B_2$ is called an \defn{embedding} or \defn{isomorphism} if the induced map $B_1 \sqcup\{0\} \to B_2 \sqcup \{0\}$ is injective or bijective, respectively. If there is an isomorphism between $B_1$ and $B_2$, then we say they are (crystal) isomorphic and write $B_1 \iso B_2$.

We say an abstract $U_q(\fsl_n)$-crystal is simply a \defn{$U_q(\fsl_n)$-crystal} if it is crystal isomorphic to the crystal basis of a $U_q(\fsl_n)$-module. In particular, the highest weight $U_q(\fsl_n)$-module $V(\lambda)$ for $\lambda \in P^+$ has a crystal basis~\cite{K90,K91}. We denote the corresponding $U_q(\fsl_n)$-crystal by $B(\lambda)$ and its highest weight element by $u_{\lambda}$.

The tensor product $B_2 \otimes B_1$ is the crystal whose set is the Cartesian product $B_2\times B_1$ and the crystal structure given by
\begin{align*}
e_i(b_2 \otimes b_1) &= \begin{cases}
e_i b_2 \otimes b_1 & \text{if } \varepsilon_i(b_2) > \varphi_i(b_1), \\
b_2 \otimes e_i b_1 & \text{if } \varepsilon_i(b_2) \le \varphi_i(b_1),
\end{cases} \\
f_i(b_2 \otimes b_1) &= \begin{cases}
f_i b_2 \otimes b_1 & \text{if } \varepsilon_i(b_2) \ge \varphi_i(b_1), \\
b_2 \otimes f_i b_1 & \text{if } \varepsilon_i(b_2) < \varphi_i(b_1),
\end{cases} \\ 
\wt(b_2 \otimes b_1) &= \wt(b_2) + \wt(b_1).
\end{align*}

\begin{remark}
Our convention for tensor products follows~\cite{BS17}, which is opposite the convention given by Kashiwara in~\cite{K91}.
\end{remark}

Consider $U_q(\fsl_n)$-crystals $B_1, \dotsc, B_t$. The action of the crystal operators on the tensor product $B = B_t \otimes \cdots \otimes B_2 \otimes B_1$ can be computed by the \defn{signature rule}.  Let $b = b_t \otimes \cdots \otimes b_2 \otimes b_1 \in B$, and for $i \in I$, we write
\[
\underbrace{\bplus\cdots\bplus}_{\varphi_i(b_t)}\
\underbrace{\bminus\cdots\bminus}_{\varepsilon_i(b_t)}\ 
\cdots\ 
\underbrace{\bplus\cdots\bplus}_{\varphi_i(b_1)}\
\underbrace{\bminus\cdots\bminus}_{\varepsilon_i(b_1)}\ .
\]
Then by successively deleting any $(\bminus,\bplus)$-pairs (in that order) in the above sequence, we obtain a sequence
\[
\sgn_i(b) :=
\underbrace{\bplus\cdots\bplus}_{\varphi_i(b)}\ 
\underbrace{\bminus\cdots\bminus}_{\varepsilon_i(b)}
\ .
\]
Suppose $1 \leq j_\bminus, j_\bplus \leq t$ are such that $b_{j_\bminus}$ contributes the leftmost $\bminus$ in $\sgn_i(b)$ and $b_{j_\bplus}$ contributes the rightmost $\bplus$ in $\sgn_i(b)$.  Then 
\begin{align*}
e_i b &= b_t \otimes \cdots \otimes b_{j_\bminus+1} \otimes e_ib_{j_\bminus} \otimes b_{j_\bminus-1} \otimes \cdots \otimes b_1, \\
f_i b &= b_t \otimes \cdots \otimes b_{j_\bplus+1} \otimes f_ib_{j_\bplus} \otimes b_{j_\bplus-1} \otimes \cdots \otimes b_1.
\end{align*}

\subsection{Crystal structure on semistandard tableaux}

Let $\lambda$ be a dominant integral weight. We systematically identify $\lambda$ with a partition by taking the sequence of indices of its constituent fundamental weights and conjugating the sequence. We further identify partitions with their Young diagrams so that $\Lambda_i$ corresponds to a column of height $i$.
A \defn{semistandard tableau of shape $\lambda$} is a filling of the boxes of $\lambda$ by positive integers, so that 
\begin{itemize}
\item each box contains exactly one number,
\item the entries of each row weakly increase from left to right, and
\item the entries of each column strictly increase from top to bottom.
\end{itemize}
We will use English convention for our Young diagrams and tableaux.
For example,
\[
\ytableausetup{boxsize=1.5em}
\ytableaushort{1113,233,5}
\]
is a semistandard tableau of shape $\lambda = (4,3,1) = \Lambda_1 + 2 \Lambda_2 + \Lambda_3$.
Let $\ssyt^n(\lambda)$ denote the set of all semistandard tableaux of shape $\lambda$ with all entries at most $n$.

Next, we give a simplified version of the usual crystal structure for $\fsl_n$ on $\ssyt^n(\lambda)$. Note that this structure agrees with considering the embedding of $T$ (of shape $\lambda$) into $B(\Lambda_1)^{\otimes |\lambda|}$ by taking either the reverse Far-Eastern reading word (bottom-to-top and left-to-right) or reverse Middle-Eastern reading word (left-to-right and bottom-to-top).\footnote{We will also refer to this simplified rule as the \emph{signature rule} as it precisely corresponds to the signature rule for tensor products if we realize the tableau as a tensor product of single boxes. Note that all three rules produce exactly the same crystal structure.}
Furthermore, this is the crystal basis associated to the highest weight representation $V(\lambda)$ of $U_q(\fsl_n)$.

The crystal operator $f_i$ acts on $T \in \ssyt^n(\lambda)$ as follows: Write $\bplus$ above each column of $T$ containing $i$ but not $i+1$, and write $\bminus$ above each column containing $i+1$ but not $i$. Now cancel signs in ordered pairs $\bminus \bplus$. If every $\bplus$ thereby cancels, then $f_i(T) = 0$. Otherwise $f_i T$ is given by replacing the $i$ corresponding to the rightmost uncanceled $\bplus$ with an $i+1$; note that in this case, $f_i T \in \ssyt^n(\lambda)$. 
The action of $e_i$ is similar: After recording signs and canceling in pairs as above, if every $\bplus$ has been canceled, then $e_i T = 0$. Otherwise $e_i T$ is given by replacing the $i+1$ corresponding to the leftmost uncanceled $\bplus$ with an $i$.
The \defn{weight} $\wt(T)$ of $T \in \ssyt^n(\lambda)$ is the \defn{weak composition} $(a_1, \dotsc, a_n) \in \ZZ_{\geq 0}^n$, where $a_i$ records the number of $i$'s in the tableau $T$.
Note that the highest weight element $u_\lambda$ is the semistandard tableau given by filling the $i$-th row of $\lambda$ from the top with $i$'s.

\begin{ex}
The highest weight tableau of shape $\lambda = (4,2,1) = 2\Lambda_1 + \Lambda_2 + \Lambda_3$ is
\[ \pushQED{\qed}
\ytableaushort{1111,22,3}. \qedhere \popQED
\] \let\qed\relax
\end{ex}

\begin{ex}\label{ex:ordinary_crystal}
For type $A_2$, the crystal graph for $\lambda = (2,1) = \Lambda_1 + \Lambda_2$ is
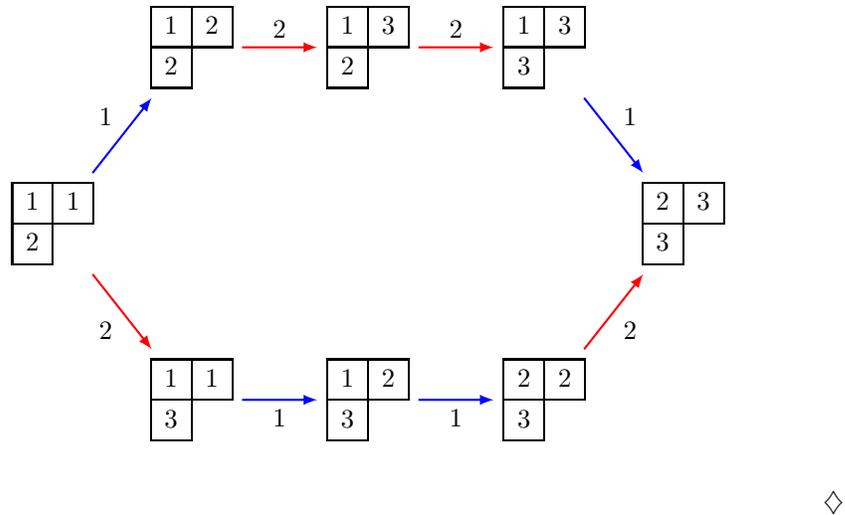
\[\begin{tikzpicture}[>=latex]
\node (A) {\ytableaushort{11,2}};
\node[above right= 1 and 0.5 of A] (B) {\ytableaushort{12,2}};
\node[below right=1 and 0.5 of A] (C) {\ytableaushort{11,3}};
\node[right=1 of B] (D) {\ytableaushort{13,2}};
\node[right=1 of C] (E) {\ytableaushort{12,3}};
\node[right=1 of D] (F) {\ytableaushort{13,3}};
\node[right=1 of E] (G) {\ytableaushort{22,3}};
\node[below right=1 and 0.5 of F] (H) {\ytableaushort{23,3}};
\path (A) edge[pil,color=blue] node[above left,black]{$1$} (B);
\path (C) edge[pil,color=blue] node[below,black]{$1$} (E);
\path (E) edge[pil,color=blue] node[below,black]{$1$} (G);
\path (F) edge[pil,color=blue] node[above right,black]{$1$} (H);
\path (A) edge[pil,color=red] node[below left,black]{$2$} (C);
\path (B) edge[pil,color=red] node[above,black]{$2$} (D);
\path (D) edge[pil,color=red] node[above,black]{$2$} (F);
\path (G) edge[pil,color=red] node[below right,black]{$2$} (H);
\end{tikzpicture}\]
\end{ex}

Let $w_0 \in \sym_n$ be the reverse permutation $[n, n-1, \cdots, 2, 1]$ (also known as the longest element). The \defn{Lusztig involution} is an involution on a highest weight crystal $B(\lambda)$ defined by sending the highest weight element $u_{\lambda} \in B(\lambda)$ to the lowest weight element $u_{w_0\lambda} \in B(\lambda)$ and extending as a crystal isomorphism by
\begin{align*}
e_i(T^*) & \mapsto (f_{n+1-i} T)^*,
\\ f_i(T^*) & \mapsto (e_{n+1-i} T)^*,
\\ \wt(T^*) & = w_0 \wt(T).
\end{align*}
We extend this involution to a direct sum of crystals by acting independently on each summand.
On a semistandard tableau $T$, the image $T^*$ of the Lusztig involution from the quantum group is equal to the \defn{Sch\"utzenberger involution} (or \defn{evacuation}) of $T$~\cite{Lenart07}.
We can alternatively compute the Sch\"utzenberger involution using the \defn{Bender--Knuth involutions} $t_i$. To obtain the action of $t_i$ on $T$, for each row, interchange the number of $i$'s without an $i+1$ directly below with the number of $i+1$'s without an $i$ directly above (see, \textit{e.g.},~\cite{ECII} for more details).
Indeed by~\cite{BPS16,CGP16,Stanley09}, we can compute the Sch\"utzenberger involution as
\begin{equation}
\label{eq:evac_BK}
T^* = t_1 (t_2 t_1) \cdots (t_{n-1} t_{n-2} \cdots t_2 t_1) T.
\end{equation}

\subsection{Schur and Grothendieck functions}

Let $\xx = (x_1, x_2, \dotsc, )$ be a countable vector of indeterminants.
Importantly, the generating function for semistandard tableaux with a fixed shape is a Schur function. More precisely, the \defn{Schur function} $s_\lambda$ may be defined by
\[
s_\lambda = \sum_{T \in \ssyt^\infty(\lambda)} \xx^{\wt(T)}, 
\]
where $\xx^{a} = x_1^{a_1} \cdots x_k^{a_k}$ for a weak composition $a = (a_1, \dots, a_k)$. Note that if we restrict ourselves to semistandard tableaux with maximum entry $n$ and fixed shape $\lambda$, we obtain the Schur polynomial $s_{\lambda}(x_1, \ldots, x_n )$, which is hence the character of $V(\lambda)$. Moreover, these are the characters of the irreducible polynomial representations of $\GL_n(\CC)$.

Another important appearance of the Schur functions is in the geometry of Grassmannians. Let $X = \Gr_k(\mathbb{C}^n)$ be the \defn{Grassmannian} parametrizing $k$-dimensional linear subspaces of $\mathbb{C}^n$. The defining action of $\GL_n(\mathbb{C})$ on $\mathbb{C}^n$ passes to an action on $X$, which we may restrict to an action by the Borel subgroup $B$ of lower triangular invertible matrices. The $B$-orbits are affine cells called the \defn{Schubert cells} of $X$. Their closures are the \defn{Schubert varieties} $X_\lambda$ and are naturally indexed by certain partitions $\lambda$. In the cohomology $H^\star(X)$ of $X$, the classes $\sigma_\lambda$ of the Schubert varieties are a basis. The role of the Schur polynomials is as polynomial representatives for these cohomology classes. That is, the product of Schur polynomials may be uniquely expressed as a sum of Schur polynomials
\[
s_\lambda \cdot s_\mu = \sum_\nu c_{\lambda, \mu}^\nu s_\nu,
\]
where $c_{\lambda, \mu}^\nu$ also satisfies
\[
\sigma_\lambda \cdot \sigma_\mu = \sum_\nu c_{\lambda, \mu}^\nu \sigma_\nu
\]
in $H^\star(X)$ such that $\nu$ is contained in a $k \times n$ rectangle.

One might desire analogous polynomial representatives for Schubert classes in richer cohomology theories. The symmetric Grothendieck polynomials, described below, play exactly this role with respect to the (connective) K-theory of the Grassmannian. A \defn{semistandard set-valued tableau of shape $\lambda$} is a filling $T$ of the boxes of $\lambda$ by finite sets of positive integers so that
\begin{itemize}
\item each box contains at least one integer,
\item every tableau obtained by deleting all but one integer from each box is a semistandard tableau (in the sense of the previous section).
\end{itemize}
Note that for a set $A$ to the left of a set $B$ in the same row, we have $\max A \leq \min B$, and for $C$ below $A$ in the same column, we have $\max A < \min C$. We will abuse notation by saying an integer $a \in T$ if there exists a box with a set $A \in T$ such that $a \in A$.

Let $\svssyt^n(\lambda)$ denote the set of all semistandard set-valued tableaux of shape $\lambda$ with all entries at most $n$. We also need the \defn{excess} of $T \in \svssyt^n(\lambda)$, defined as
\[
\excess(T) = \sum_{A \in T} \big( |A| - 1 \big).
\]
In other words, the excess of $T$ is the number of extra integers in $T$ compared to a usual semistandard tableau of the same shape. We also have $\excess(T) = |\wt(T)| - |\lambda|$, where $\lambda$ is the shape of $T$.

Following A.~Buch's formulation~\cite{Buch02}, for a partition $\lambda$, we define the \defn{symmetric Grothendieck function} by 
\[
\G_\lambda(\xx; \beta) = \sum_{T \in \svssyt^\infty(\lambda)} \beta^{\excess(T)} \xx^{\wt(T)}.
\]
The \defn{symmetric Grothendieck polynomials} are the finite-variable truncations
\[
\G_\lambda(x_1, \ldots, x_n; \beta) = \sum_{T \in \svssyt^n(\lambda)} \beta^{\excess(T)} \xx^{\wt(T)}.
\]

\begin{remark}
Symmetric Grothendieck functions are often referred to as \defn{stable Grothendieck polynomials} as they are the stable limits as $n \to \infty$ of the original Grothendieck polynomials due to A.~Lascoux--M.-P.~Sch\"{u}tzenberger~\cite{LS82,LS83}. This terminology is somewhat unfortunate, as the stable limits are not polynomials, but rather power series. 

Moreover, strictly speaking, Grothendieck polynomials are usually defined by taking $\beta = -1$. The extension to general $\beta$ is due to S.~Fomin--A.~Kirillov~\cite{FK94}; geometrically, this extension corresponds to enriching ordinary K-theory to connective K-theory~\cite{Hudson}. 
The symmetric $\beta$-Grothendieck polynomials are the main object of this paper; for concision, when it will not cause confusion, we will simply refer to them as Grothendieck polynomials.
\end{remark}

For a Schubert variety $X_\lambda$ with a fixed Bott--Samelson resolution $Y_\lambda$, one obtains a class $[X_\lambda ]$ in the connective K-theory ring of $X$ by pushing forward the Bott--Samelson class, and together these classes form a basis. (By \cite{BE90}, these classes are independent of the choice of Bott--Samelson resolutions.) Thereby define structure coefficients 
\[
[ X_\lambda] \cdot [ X_\mu ] = \sum_\nu C_{\lambda, \mu}^\nu [ X_\nu ].
\]
Then the definition of Grothendieck functions is such that
\[
\G_\lambda \cdot \G_\mu = \sum_\nu C_{\lambda, \mu}^\nu \G_\nu,
\]
for $n, k$ sufficiently large.
When we take $\beta = 0$, we obtain $\G_{\lambda} = s_{\lambda}$, and so the Grothendieck functions are therefore \emph{K-theoretic analogs of Schur functions}. Similarly, the $\beta = -1$ specializations represent the classes of the structure sheaves of the Schubert varieties in the ordinary K-theory of $X$.

There is also a Pieri rule for Grothendieck polynomials.

\begin{thm}[{\cite{Lenart00}}]
Let $\lambda$ be a partition and $\ell \geq 1$. Then, we have
\[
\G_{1^{\ell}} \cdot \G_{\lambda} = \sum_{\nu} \beta^{|\nu / \lambda| - \ell} \binom{c(\nu/\lambda) - 1}{|\nu/\lambda| - \ell} \G_{\nu},
\]
where the sum is over all partitions $\nu \supset \lambda$ such that $\ell(\nu) \leq n$, $\nu/\lambda$ is a vertical strip and $c(\nu/\lambda)$ is the number of non-empty columns in $\nu/\lambda$.
\end{thm}

In particular, we have
\begin{equation}
\label{eq:grothendieck_pieri_rule}
\G_1 \cdot \G_{\lambda} = \sum_{\nu} \beta^{|\nu/\lambda| - 1} \G_{\nu},
\end{equation}
where we sum over all $\nu$ formed from $\lambda$ by adding a box to addable corners in all possible ways with $\ell(\nu) \leq n$.

\subsection{Key and Lascoux polynomials}

Let $\poly = \mathbb{Z}[x_1, x_2, \dots, x_n]$ denote the ring of polynomials in $n$ variables. We will be interested in several families of polynomials, defined by the action of certain operators on $\poly$. Let the symmetric group $\sym_n$ act on $\poly$ by permuting variables. Recall that the length of an permutation $w$ is the length of the smallest word of $w$ in terms of simple transpositions $s_i$, and we say an expression (in terms of simple transpositions) is reduced if its length is equal to the length of $w$.

For $1 \leq i < n$, the \defn{Newton (divided difference) operator} $\partial_i$ acts on $\poly$ by 
\[ \partial_i f = \frac{f - s_i f}{x_i - x_{i+1}} ,\]
for $f \in \poly$. 
Then the \defn{Demazure operator} $\pi_i$ acts by
\[
\pi_i f = \partial_i(x_i \cdot f) = \frac{x_i \cdot f - x_{i+1} \cdot s_i f}{x_i - x_{i+1}},
\]
while the \defn{Demazure--Lascoux operator} $\varpi_i$ acts by
\begin{align*}
\varpi_i f 
= \pi_i\bigl( (1  + \beta x_{i+1}) \cdot f \bigr) = \partial_i\bigl( (x_i + \beta x_i x_{i+1}) \cdot f \bigr)
= \pi_i f + \beta \pi_i (x_{i+1} \cdot f).
\end{align*}
One may check that all three of these sets of operators satisfy the braid relations:
\begin{align*}
\partial_i \partial_j &= \partial_j \partial_i \hspace{1.5cm} \text{for $|i-j| > 1$}, \\
\partial_i \partial_{i+1} \partial_i &= \partial_{i+1} \partial_i \partial_{i+1},
\end{align*} 
and analogously for $\pi_i$ and $\varpi_i$. Hence for any permutation $w \in \sym_n$, one may unambiguously define $\partial_w := \partial_{s_1} \partial_{s_2} \cdots \partial_{s_k}$, where $s_1 s_2 \cdots s_k$ is any reduced expression for $w$, and similarly for $\pi_w$ and $\varpi_w$.

The \defn{key polynomials} $\{ \kappa_a \}$ (also known as \defn{Demazure characters} or \defn{standard bases}) are a family of polynomials (indexed by weak compositions) introduced by M.~Demazure~\cite{Demazure74}. Important early work on these polynomials includes~\cite{LS90,RS95}. They are defined by 
\[
\kappa_{a} = \pi_{w(a)} \mathbf{x}^{{\sf sort}(a)},
\]
where ${\sf sort}(a)$ is the unique partition obtained by permuting the $a_i$'s into weakly decreasing order and $w(a) \in \sym_n$ is the minimal length permutation that permutes $a$ to ${\sf sort}(a)$.
Moreover, Demazure characters are the characters of Demazure modules~\cite{Demazure74}, certain $U_q^+(\g)$-modules that are known to admit crystal bases~\cite{K93,L95-3} and are closely related to Schubert classes for the cohomology of the flag variety. Indeed, for a fixed reduced expression $w(a) = s_{i_1} \dotsm s_{i_{\ell}}$ we have
\[
\kappa_{a}(\xx) = \sum_{b \in B_w(\lambda)} \wt(b),
\]
where $\lambda = {\sf sort}(a)$ and the sum is over the \defn{Demazure crystal}
\[
B_w(\lambda) := \{ b \in B(\lambda) \mid e_{i_{\ell}}^{\max} \cdots e_{i_1}^{\max} b = u_{\lambda} \}.
\]

More mysterious are the polynomials
\[
L_a(\xx; \beta) = \varpi_{w(a)} \xx^{{\sf sort}(a)}.
\]
These polynomials have been studied recently by C.~Ross--A.~Yong~\cite{RY15}, A.~Kirillov~\cite{Kirillov:notes}, C.~Monical~\cite{Monical16}, and C.~Monical--O.~Pechenik--D.~Searles \cite{MPS18}. We follow~\cite{Monical16} in referring to the $L_a$'s as \defn{Lascoux polynomials}, in honor of A.~Lascoux who essentially introduced them in~\cite{Lascoux01}. Note that the specialization at $\beta = 0$ is  $L_a(\xx; 0) = \kappa_{a}(\xx)$. Algebraically, the relation between the Lascoux and key polynomials is precisely analogous to that between the Grothendieck and Schur functions in the sense that
\[
s_{\lambda}(\xx) = \pi_{w_0} \xx^{\lambda},
\hspace{80pt}
\G_{\lambda}(\xx; \beta) = \varpi_{w_0} \xx^{\lambda}.
\]
Hence we can think of the Lascoux polynomials as K-theoretic analogs of the key polynomials, although no such geometric or representation-theoretic interpretation is currently known. We think of the results of Section~\ref{sec:K_theory} as potential progress towards such a representation-theoretic interpretation.

\section{Crystal structure of semistandard set-valued tableaux}
\label{sec:crystal_set_valued}

We define an abstract $U_q(\fsl_n)$-crystal structure on $\svssyt^n(\lambda)$.

\begin{dfn}
The action of $f_i$ on $\svssyt^n(\lambda)$ is defined exactly as for the usual semistandard tableaux unless $i \in \bbb^{\rightarrow}$, where $\bbb^{\rightarrow}$ is the box immediately to the right of the box $\bbb$ that corresponds to the rightmost uncanceled $\bplus$. In this case, $f_i(T)$ is given by removing $i$ from $\bbb^{\rightarrow}$ and adding $i + 1$ to $\bbb$. Recall that if $i \notin \bbb^{\rightarrow}$, then we simply replace $i$ with $i+1$ in $\bbb$.

The action of $e_i$ is the reverse: We define the action of $e_i$ exactly as for the usual semistandard tableaux unless $i+1 \in \bbb^{\leftarrow}$, where $\bbb^{\leftarrow}$ is the box immediately to the left of the box $\bbb$ that corresponds to the leftmost uncanceled $\bminus$. In this case, $e_i(T)$ is given by removing $i+1$ from $\bbb^{\leftarrow}$ and adding $i$ to $\bbb$. Recall that if $i+1 \notin \bbb^{\leftarrow}$, then we simply replace $i+1$ with $i$ in $\bbb$.
\end{dfn}

\begin{lemma}
\label{lem:crystal_well_defined}
The action of $e_i$ and $f_i$ is well-defined. In particular, if $T \in \svssyt^n(\lambda)$ and $i < n$, then $f_i(T) = 0$ or $f_i(T) \in \svssyt^n(\lambda)$ and similarly for $e_i$.
\end{lemma}

\begin{proof}
We consider the action of $f_i$.
We may assume that not every $\bplus$ cancels. Hence, there is a rightmost uncanceled $\bplus$, corresponding to box $\bbb$. Since $\bbb$ corresponds to an uncanceled $\bplus$, we have $i \in \bbb$ and $i+1 \notin \bbb$. 

Suppose $i \notin \bbb^\rightarrow$. Then we replace $i \in \bbb$ by $i+1 \in \bbb$ to form $f_i(T)$. We must check that the result is a semistandard set-valued tableau. Since $T \in \svssyt^n(\lambda)$, $i$ is strictly greater than all integers above $\bbb$ in its column and strictly less than all integers below $\bbb$ in its column. Since $i+1$ does not appear in $\bbb$'s column in $T$, it follows that $i+1$ is strictly greater than all integers above $\bbb$ in its column and strictly less than all integers below $\bbb$ in its column. Similarly, it is clear that $i+1$ is strictly greater than all integers appearing left of $\bbb$ in its row. Since by assumption $i \notin \bbb^\rightarrow$, every integer appearing right of $\bbb$ in its row is strictly greater than $i$, and thus weakly greater than $i+1$. Thus $f_i(T) \in \svssyt^n(\lambda)$.

Otherwise $i \in \bbb^\rightarrow$.  Since $\bbb$ corresponds to the rightmost uncanceled $\bplus$, there cannot be another $\bplus$ over the column of $\bbb^\rightarrow$. Hence $i+1$ appears in that column. By semistandardness, it must appear weakly below $\bbb^\rightarrow$. If it appears strictly below $\bbb^\rightarrow$, then $\bbb^\downarrow$ is nonempty and every integer in $\bbb^\downarrow$ is strictly greater than $i$ and weakly less than $i+1$. Hence $i+1 \in \bbb^\downarrow$, contradicting that there is a $\bplus$ over its column. Thus $i+1 \in \bbb^\rightarrow$.

Now we form $f_i(T)$ be deleting $i$ from $\bbb^\rightarrow$ and adding $i+1$ to $\bbb$. Since $i+1 \in \bbb^\rightarrow$, this leaves $\bbb^\rightarrow$ nonempty. It remains to check that $\bbb$ satisfies the semistandardness conditions in $f_i(T)$. Column strictness is argued exactly as in the previous case. It is clear that $i+1$ is strictly greater than all integers appearing left of $\bbb$ in its row. Since $T$ has $i+1 \in \bbb^\rightarrow$ and $f_i(T)$ has $i \notin \bbb^\rightarrow$, it follows that $i+1$ is weakly less than all integers appearing right of $\bbb$ in its row in $f_i(T)$. Thus $f_i(T) \in \svssyt^n(\lambda)$.

The proof that $e_i$ is well-defined is exactly dual to the proof for $f_i$.
\end{proof}

It is clear that $f_i T = T'$ if and only if $T = e_i T'$ for all $T,T' \in \svssyt^n(\lambda)$ by the signature rule.
We note that for $f_i$, if $i \in \bbb^{\rightarrow}$, we must also have $i+1 \in \bbb^{\rightarrow}$. Likewise, for $e_i$, if $i+1 \in \bbb^{\leftarrow}$, we must also have $i \in \bbb^{\leftarrow}$.

As for semistandard tableaux, define $\wt(T) = \sum_{i \in I} c_i \epsilon_i$, where $c_i$ is the number of boxes $A \in T$ such that $i \in A$. Thus, it is clear that the weight changes by $-\alpha_i$ when applying $f_i$. From the signature rule, we have that $\varphi_i(T) = \varepsilon_i(T) + \langle h_i, \wt(T) \rangle$. Hence, we obtain the following.

\begin{prop}
\label{prop:set_valued_crystal}
The set $\svssyt^n(\lambda)$ is an abstract $U_q(\fsl_n)$-crystal. \qed
\end{prop}

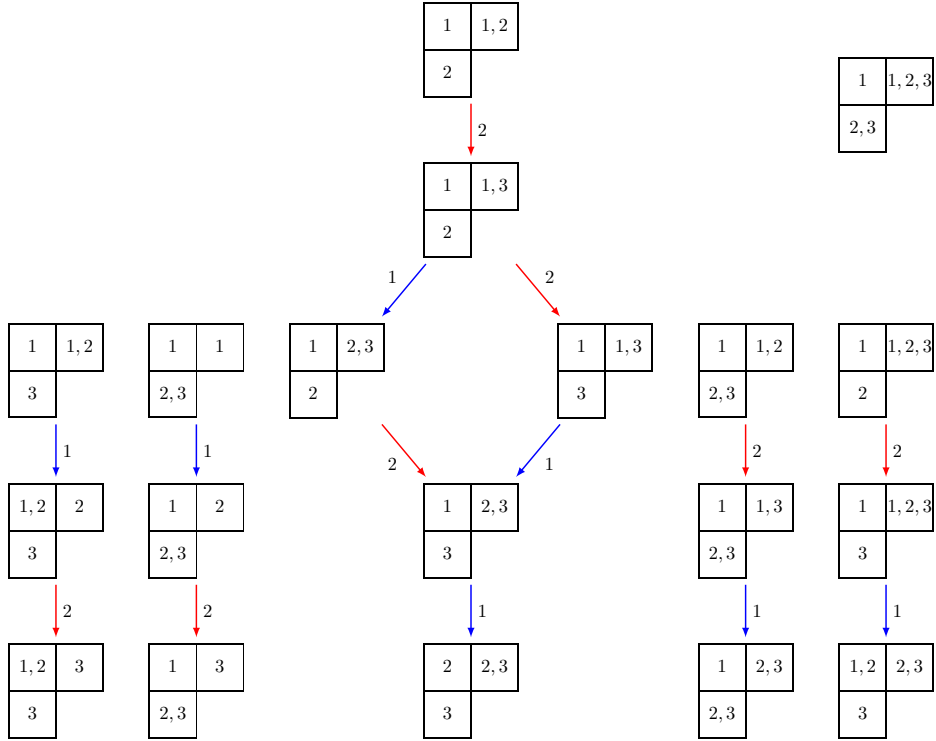
\begin{figure}
\[
\ytableausetup{boxsize=2.5em}
\scalebox{0.7}{
\begin{tikzpicture}[>=latex]
\node (A) {\ytableaushort{1{1,2},3}};
\node[below = 1 of A] (B) {\ytableaushort{{1,2}2,3}};
\node[below = 1 of B] (C) {\ytableaushort{{1,2}3,3}};
\path (A) edge[pil,color=blue] node[right,black]{$1$} (B);
\path (B) edge[pil,color=red] node[right,black]{$2$} (C);
\end{tikzpicture}
\hspace{4mm}
\begin{tikzpicture}[>=latex]
\node (A) {\ytableaushort{11,{2,3}}};
\node[below = 1 of A] (B) {\ytableaushort{12,{2,3}}};
\node[below = 1 of B] (C) {\ytableaushort{13,{2,3}}};
\path (A) edge[pil,color=blue] node[right,black]{$1$} (B);
\path (B) edge[pil,color=red] node[right,black]{$2$} (C);
\end{tikzpicture}
\hspace{4mm}
\begin{tikzpicture}[>=latex]
\node (A) {\ytableaushort{1{1,2},2}};
\node[below = 1 of A] (B) {\ytableaushort{1{1,3},2}};
\node[below left = 1 and 0.5 of B] (C) {\ytableaushort{1{2,3},2}};
\node[below right = 1 and 0.5 of B] (D) {\ytableaushort{1{1,3},3}};
\node[below right = 1 and 0.5 of C] (E) {\ytableaushort{1{2,3},3}};
\node[below = 1 of E] (F) {\ytableaushort{2{2,3},3}};
\path (A) edge[pil,color=red] node[right,black]{$2$} (B);
\path (B) edge[pil,color=red] node[above right,black]{$2$} (D);
\path (C) edge[pil,color=red] node[below left,black]{$2$} (E);
\path (B) edge[pil,color=blue] node[above left,black]{$1$} (C);
\path (D) edge[pil,color=blue] node[below right,black]{$1$} (E);
\path (E) edge[pil,color=blue] node[right,black]{$1$} (F);
\end{tikzpicture}
\hspace{4mm}
\begin{tikzpicture}[>=latex]
\node (A) {\ytableaushort{1{1,2},{2,3}}};
\node[below = 1 of A] (B) {\ytableaushort{1{1,3},{2,3}}};
\node[below = 1 of B] (C) {\ytableaushort{1{2,3},{2,3}}};
\path (A) edge[pil,color=red] node[right,black]{$2$} (B);
\path (B) edge[pil,color=blue] node[right,black]{$1$} (C);
\end{tikzpicture}
\hspace{4mm}
\begin{tikzpicture}[>=latex]
\node (A) {\ytableaushort{1{1,2,3},2}};
\node[below = 1 of A] (B) {\ytableaushort{1{1,2,3},3}};
\node[below = 1 of B] (C) {\ytableaushort{{1,2}{2,3},3}};
\node[above = 3 of A] (extra) {\ytableaushort{1{1,2,3},{2,3}}};
\path (A) edge[pil,color=red] node[right,black]{$2$} (B);
\path (B) edge[pil,color=blue] node[right,black]{$1$} (C);
\end{tikzpicture}
}
\]
\caption{The $U_q(\fsl_3)$-crystal structure on $\svssyt^3(\lambda) \setminus \ssyt^3(\lambda)$, with $\lambda = (2,1)$. (For comparison, the $U_q(\fsl_3)$-crystal structure on $\ssyt^3(\lambda)$ appears in Example~\ref{ex:ordinary_crystal}.)}
\label{fig:svt_crystal_ex}
\end{figure}

\begin{ex}\label{ex:set-valued_crystal}
For $\lambda = (2,1)$, the crystal $\svssyt^3(\lambda)$ has seven components. One component is illustrated in Example~\ref{ex:ordinary_crystal}; the other six components are shown in Figure~\ref{fig:svt_crystal_ex}.
\end{ex}

\begin{prop}
\label{prop:single_row}
Fix some $s \in \ZZ_{>0}$, $k \in \{1, \dotsc, n\}$, and $\mu = \Lambda_k + (s-1) \Lambda_1$.
We have $\svssyt^n_{\mu}(s\Lambda_1) \iso B(\mu)$, where $\svssyt^n_{\mu}(s\Lambda_1) \subseteq \svssyt^n(s\Lambda_1)$ is the closure under $e_i$ and $f_i$, for all $i \in I$, of the $1 \times s$ tableau
\[
U_{\mu} = 
\begin{array}{|c|c|c|c|c|}
\hline
1 & 1 & \cdots & 1 & 1, \ldots, k
\\\hline
\end{array}\ .
\]
\end{prop}

\begin{proof}
First consider the case $s=1$. In this case, $\mu = \Lambda_k$.
It is straightforward to see that the map $\Phi_k : \svssyt^n_{\Lambda_k}(\Lambda_1) \to B(\Lambda_k)$ given by
\[
\boxed{x_1 < \cdots < x_k} \longmapsto
\begin{array}{|c|} \hline x_1 \\\hline \vdots \\\hline x_k \\\hline \end{array}
\]
is a crystal isomorphism, as desired.

Now suppose $s > 1$. By the semistandardness conditions, it is straightforward to see that $U_{\mu}$ is the unique element of $\svssyt^n(s\Lambda_1)$ with weight $\mu$. Note that $U_\mu$ is highest weight and that there is a unique $B(\mu) \subseteq B\bigl( (s-1)\Lambda_1 \bigr) \otimes B(\Lambda_k)$.


We define a morphism $\Phi_\mu \colon \svssyt^n_{\mu}(s\Lambda_1) \to B(\mu)$ as follows. Let $T \in \svssyt^n_{\mu}(s\Lambda_1)$ be
\begin{equation}
\label{eq:T_defn}
T =
\begin{array}{|c|c|c|c|}
\hline
x_{1,1} < \cdots < x_{1,\ell_1} & x_{2,1} < \cdots < x_{2,\ell_2} & \cdots &  x_{s,1} < \cdots < x_{s,\ell_s} 
\\\hline
\end{array}\ .
\end{equation}
Then we have
\[
\Phi_\mu(T) =
\begin{array}{|c|c|c|c|}
\hline
x_{1,\ell_1} & x_{2,\ell_2} & \cdots & x_{s-1,\ell_{s-1}}
\\\hline
\end{array}
\otimes
\begin{array}{|c|} \hline x_{1,1} \\\hline \vdots \\\hline x_{1,\ell_1-1} \\\hline x_{2,1} \\\hline \vdots \\\hline x_{s-1,\ell_{s-1}-1} \\\hline x_{s,1} \\\hline \vdots \\\hline x_{s,\ell_s} \\\hline \end{array}\ .
\]
We note that $\Phi^{-1}(r \otimes c)$ can be considered as the unique insertion of $c$ into $r$ to obtain a semistandard set-valued tableau. More explicitly, for
\[
S = \begin{array}{|c|c|c|c|}
\hline
r_1 & r_2 & \cdots & r_{s-1}
\\\hline
\end{array}
\otimes
\begin{array}{|c|} \hline c_1 \\\hline c_2 \\\hline \vdots  \\\hline c_k \\\hline \end{array}\ ,
\]
we define $\Phi_{\mu}^{-1}$ by
\[
\Phi_{\mu}^{-1}(S) =
\begin{array}{|c|c|c|c|c|}
\hline
c_1 \cdots c_{i_1} r_1 & c_{i_1+1} \cdots c_{i_2} r_2 & \cdots & c_{i_{s-2} + 1} \cdots c_{i_{s-1}} r_{s-1} & c_{i_{s-1}+1} \cdots c_k
\\\hline
\end{array}\ ,
\]
where $c_{i_j} < r_j \leq c_{i_j+1}$ for all $1 \leq j < s$.
Note that any pairing of $\bminus \bplus$ in $T$ must come from two entries in the $j$-th cell. Hence, either this pair is $c_k c_{k+1}$, for some $c_{i_{j-1}} \leq k < c_{i_j}$ or $c_{i_j} r_j$. However, this clearly corresponds to a canceling $\bminus \bplus$ under $\Phi_{\mu}$, and hence $\Phi_{\mu}$ is a crystal isomorphism.
\end{proof}

\begin{ex}
\label{ex:single_row_bijection}
We consider the crystal isomorphism $\Phi_{311}$ for $\fsl_3$ given by Proposition~\ref{prop:single_row}, where we have marked the entries from the column in bold:
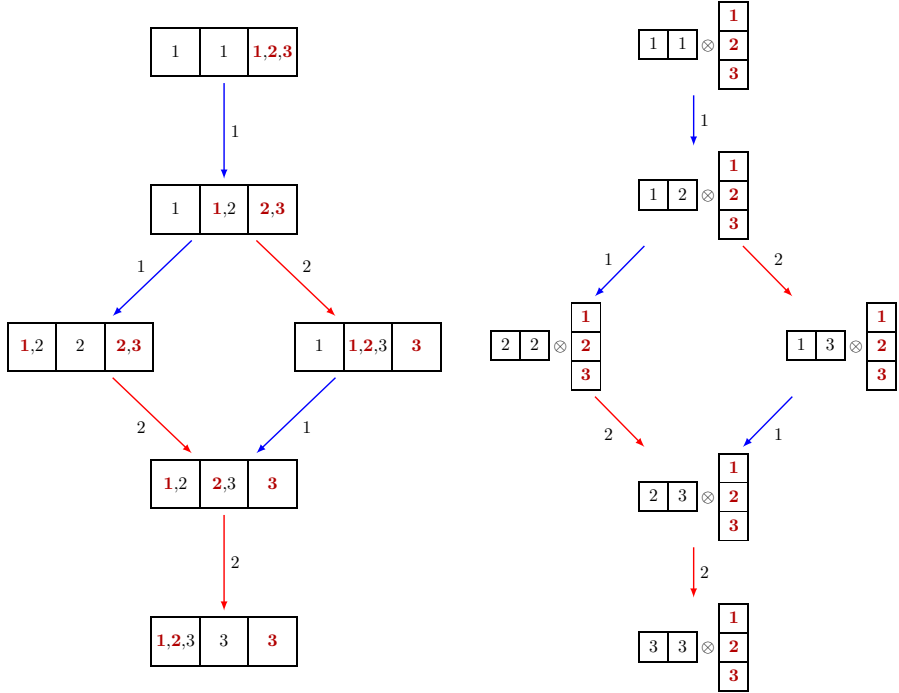
\[
\ytableausetup{boxsize=2.7em}
\scalebox{0.67}{
\begin{tikzpicture}[baseline=4,>=latex]
\node (A) {\ytableaushort{11{\cbone,\!\cbtwo,\!\cbthr}}};
\node[below = 1.9 of A] (B) {\ytableaushort{1{\cbone,\!2}{\cbtwo,\!\cbthr}}};
\node[below left = 1.5 and -0.3 of B] (C) {\ytableaushort{{\cbone,\!2}{2}{\cbtwo,\!\cbthr}}};
\node[below right = 1.5 and -0.3 of B] (D) {\ytableaushort{1{\cbone,\!\cbtwo,\!3}\cbthr}};
\node[below right = 1.5 and -0.3 of C] (E) {\ytableaushort{{\cbone,\!2}{\cbtwo,\!3}\cbthr}};
\node[below = 1.9 of E] (F) {\ytableaushort{{\cbone,\!\cbtwo,\!3}3\cbthr}};
\path (A) edge[pil,color=blue] node[right,black]{$1$} (B);
\path (B) edge[pil,color=red] node[above right,black]{$2$} (D);
\path (C) edge[pil,color=red] node[below left,black]{$2$} (E);
\path (B) edge[pil,color=blue] node[above left,black]{$1$} (C);
\path (D) edge[pil,color=blue] node[below right,black]{$1$} (E);
\path (E) edge[pil,color=red] node[right,black]{$2$} (F);
\end{tikzpicture}
\hspace{4mm}
\ytableausetup{boxsize=1.6em}
\begin{tikzpicture}[baseline=0,>=latex]
\node (A) {$\raisebox{-5pt}{\ytableaushort{11}} \otimes \raisebox{11pt}{\ytableaushort{\cbone,\cbtwo,\cbthr}}$};
\node[below = 1 of A] (B) {$\raisebox{-5pt}{\ytableaushort{12}} \otimes \raisebox{11pt}{\ytableaushort{\cbone,\cbtwo,\cbthr}}$};
\node[below left = 1 and 0.5 of B] (C) {$\raisebox{-5pt}{\ytableaushort{22}} \otimes \raisebox{11pt}{\ytableaushort{\cbone,\cbtwo,\cbthr}}$};
\node[below right = 1 and 0.5 of B] (D) {$\raisebox{-5pt}{\ytableaushort{13}} \otimes \raisebox{11pt}{\ytableaushort{\cbone,\cbtwo,\cbthr}}$};
\node[below right = 1 and 0.5 of C] (E) {$\raisebox{-5pt}{\ytableaushort{23}} \otimes \raisebox{11pt}{\ytableaushort{\cbone,\cbtwo,\cbthr}}$};
\node[below = 1 of E] (F) {$\raisebox{-5pt}{\ytableaushort{33}} \otimes \raisebox{11pt}{\ytableaushort{\cbone,\cbtwo,\cbthr}}$};
\path (A) edge[pil,color=blue] node[right,black]{$1$} (B);
\path (B) edge[pil,color=red] node[above right,black]{$2$} (D);
\path (C) edge[pil,color=red] node[below left,black]{$2$} (E);
\path (B) edge[pil,color=blue] node[above left,black]{$1$} (C);
\path (D) edge[pil,color=blue] node[below right,black]{$1$} (E);
\path (E) edge[pil,color=red] node[right,black]{$2$} (F);
\end{tikzpicture}
}
\]
\end{ex}

\begin{remark}
\label{remark:hook_shape}
We can also construct a bijection $\Psi_{\mu}$ to the hook shape $B(\mu)$ following~\cite[Sec.~6]{Buch02}. Specifically, given a semistandard set-valued tableau $T$, we take the minimal element of each entry of $T$ and place the remaining entries down the column. Specifically, keeping the same notation as in Equation~\eqref{eq:T_defn}, we have
\[
T \longmapsto
\begin{array}{|c|c|c|c|}
\hline
x_{1,1} & x_{2,1} & \cdots & x_{s,1}
\\\hline
x_{1,2}
\\\cline{1-1}
\vdots
\\\cline{1-1}
x_{1,\ell_1}
\\\cline{1-1}
\vdots
\\\cline{1-1}
x_{s,2}
\\\cline{1-1}
\vdots
\\\cline{1-1}
x_{s,\ell_2}
\\\cline{1-1}
\end{array}
\]
and that $\Psi_{\mu}$ is an isomorphism is similar to showing $\Phi_{\mu}$, the isomorphism given by Proposition~\ref{prop:single_row}, is an isomorphism.
We note that the Robinson--Schensted--Knuth (RSK) insertion $r \leftarrow c$ of the element $r \otimes c = \Phi_{\mu}(T)$ gives an explicit crystal isomorphism with the elements of $B(\mu)$ (see, \textit{e.g.},~\cite{BS17,LLT02}); in particular, the result is precisely $\Psi_{\mu}(T)$. For more on RSK insertion, we refer the reader to, \textit{e.g.},~\cite{ECII}.
\end{remark}

We also show that the $U_q(\fsl_n)$-crystal structure can be equally given by both the Far-Eastern and Middle-Eastern reading words as for $\ssyt^n(\lambda)$. Indeed, let $\svssyt_F^n(\lambda)$ (resp.~$\svssyt_M^n(\lambda)$) denote the semistandard set-valued tableaux, where the box we act on is given by the signature rule using the Far-Eastern (resp.\ Middle-Eastern) reading word of the entries.

\begin{prop}
\label{prop:reading_word_equality}
We have
\[
\svssyt^n(\lambda) = \svssyt_F^n(\lambda) = \svssyt_M^n(\lambda)
\]
as $U_q(\fsl_n)$-crystals.
\end{prop}

\begin{proof}
It is clear that $\svssyt^n(\lambda) = \svssyt_F^n(\lambda)$ by the column strictness. That $\svssyt^n_F(\lambda) = \svssyt_M^n(\lambda)$ follows from semistandardness: for any canceling pair, the entry contributing the $\bminus$ must be to the lower-left of the entry giving the $\bplus$.
\end{proof}

If we consider reading each box as a column, then highest weight elements in $\svssyt^n(\lambda)$ are characterized as the tableaux whose reading word is a highest weight element in $B(\Lambda_1)^{\otimes \lvert \mu \rvert}$, or a \defn{Yamanouchi} word. Hence, we call such highest weight elements in $\svssyt^n(\lambda)$ \defn{Yamanouchi set-valued tableaux}.

\begin{thm}
\label{thm:decomp_of_sv_crystal}
Let $\svssyt^n(T)$ denote the crystal closure of a Yamanouchi set-valued tableau $T$ of  weight $\mu$. Then 
\[
\svssyt^n(T) \iso B(\mu).
\]
Moreover, we have
\[
\svssyt^n(\lambda) \iso \bigoplus_{\mu} B(\mu)^{\oplus M_{\lambda}^{\mu}},
\]
where $M_{\lambda}^{\mu}$ denotes the number of Yamanouchi set-valued tableaux of shape $\lambda$ and weight $\mu$.
\end{thm}

\begin{proof}
Let $\lambda = (\lambda_1, \dotsc, \lambda_{\ell})$ as a partition.
We construct a map
\[
\Psi \colon \svssyt^n(\lambda) \to \svssyt^n(\lambda_1 \Lambda_1) \otimes \svssyt^n(\lambda_2 \Lambda_1) \otimes \dotsm \otimes \svssyt^n(\lambda_{\ell} \Lambda_1)
\]
by reading a tableau row-by-row, and $\Psi$ is a strict crystal embedding by Proposition~\ref{prop:reading_word_equality}, since $\Psi$ is compatible with the Middle-Eastern reading word. By Proposition~\ref{prop:single_row}, there exists an isomorphism
\[
\Phi_s \colon \svssyt^n(s \Lambda_1) \to \ \bigoplus_{k=1}^n B\bigl( (s-1) \Lambda_1 + \Lambda_k \bigr).
\]
Therefore, we construct a strict crystal embedding
\[
(\Phi_{\lambda_1} \oplus \dotsm \oplus \Phi_{\lambda_{\ell}}) \circ \Psi \colon \svssyt^n(\lambda) \to \bigotimes_{j=1}^{\ell} \left( \bigoplus_{k=1}^n B\bigl( (\lambda_j - 1) \Lambda_1 + \Lambda_k \bigr) \right),
\]
and since $\svssyt^n(\lambda)$ is a regular crystal, we have $\svssyt^n(T) \iso B(\mu)$ for $T \in \svssyt^n(\lambda)$ of weight $\mu$.

The second claim follows from the first, since each component of the crystal $\svssyt^n(\lambda)$ contains a unique Yamanouchi tableau as its highest weight element.
\end{proof}

\begin{ex}
The reader may check that the components illustrated in Example~\ref{ex:set-valued_crystal} are respectively crystal isomorphic to the crystal structures on $\ssyt^3(2,2)$, $\ssyt^3(2,1,1)$, $\ssyt^3(2,2,1)$ and $\ssyt^3(2,2,2)$.
\end{ex}

As a corollary, we obtain the following combinatorial formula for writing a symmetric Grothendieck polynomial in the basis of Schur polynomials.

\begin{cor}\label{cor:schur_expansion_of_Groth}
We have
\[
\G_{\lambda} = \sum_{\mu} \beta^{|\mu| - |\lambda|} M_\lambda^\mu s_{\mu},
\]
where $M^\lambda_\mu$ denotes the number of Yamanouchi set-valued tableaux of shape $\lambda$ and weight $\mu$.
\end{cor}

\begin{proof}
This follows from Theorem~\ref{thm:decomp_of_sv_crystal}, noting that the character for the whole crystal $\svssyt^n(\lambda)$ is the Grothendieck polynomial $\G_{\lambda}$, while the character for a component that is isomorphic to $B(\mu)$ is the Schur polynomial $s_\mu$. Note also that the excess is precisely $|\mu| - |\lambda|$.
\end{proof}

C.~Lenart~\cite[Theorem~2.2]{Lenart00} gives a different combinatorial formula for the coefficients $M_\lambda^\mu$ of Corollary~\ref{cor:schur_expansion_of_Groth}. An \defn{increasing tableau}  of shape $\mu / \lambda$ is a filling of the skew Young diagram $\mu / \lambda$ by positive integers that is strictly increasing from left to right across rows and strictly increasing from top to bottom down columns. We say a tableau is \defn{flagged} if for each $i$, every label in row $i$ is at most $i-1$ (counting the top row as row $1$). The formula of~\cite[Theorem~2.2]{Lenart00} is that $M_\lambda^\mu$ counts the number of flagged increasing tableaux of shape $\mu / \lambda$.

It follows then from Corollary~\ref{cor:schur_expansion_of_Groth} that flagged increasing tableaux of shape $\mu / \lambda$ are equinumerous with Yamanouchi set-valued tableaux of shape $\lambda$ and weight $\mu$. Although these objects are superficially very different, we now show that the \defn{uncrowding} bijection given in~\cite[Sec.~6]{Buch02} (see also \cite{RTY18}), recalled below, yields the desired crystal isomorphism. 

Let  $\mathcal{F}_{\mu / \lambda}$ denote the set of flagged increasing tableaux of shape $\mu / \lambda$. Following~\cite{Buch02}, define a bijection $\psi \colon \svssyt^n(\lambda) \to B(\mu) \times \mathcal{F}_{\mu / \lambda}$ recursively as follows. Fix some $T \in \svssyt^n(\lambda)$. Let $R$ be the top row of $T$ and $T'$ be the remaining rows. Suppose $\psi(T') = (S', F')$. Consider $R \mapsto r \otimes c$ given by Proposition~\ref{prop:single_row} (or a single hook shape, which equals the RSK insertion $r \leftarrow c$, by Remark~\ref{remark:hook_shape}), we perform RSK insertion $S = S' \leftarrow r \leftarrow c$ (we disregard the recording tableau). Let $F$ be the flagged increasing tableau of shape $\mu / \lambda$ obtained by copying the $i$-th row of $F'$ to the $(i+1)$-th row of $F$; if $F$ needs an additional box in this row, we put an $i$ in this box. The result is $\psi(T) = (S, F)$.

\begin{thm}
\label{thm:flagged_tableaux_isomorphism}
We have
\[
\svssyt^n(\lambda) \iso \bigoplus_{\mu} B(\mu)^{\oplus F^{\lambda}_{\mu}},
\]
where $F^{\lambda}_{\mu}$ denotes the number of flagged increasing tableaux of shape $\mu / \lambda$.
\end{thm}

\begin{proof}
We give $B(\mu) \times \mathcal{F}_{\mu / \lambda}$ a crystal structure by the usual crystal operators acting on $B(\mu)$.
We have that $\psi$ is a crystal isomorphism since RSK insertion and Proposition~\ref{prop:single_row} are crystal isomorphisms.
\end{proof}

\begin{ex}
We have
\begin{align*}
\ytableausetup{boxsize=1.5em}
T & = \ytableaushort{{1,\!2}3{3,\!4},{3,\!5}{5,\!6}}\,,
&
\psi(T) & = \left(
  \raisebox{30pt}{$\ytableaushort{133,24,35,5,6}$}\,, \,
  \raisebox{30pt}{$\ytableaushort{\cdot\cdot\cdot,\cdot\cdot,12,2,4}$}
\right),
\allowdisplaybreaks \\
f_3 T & = \ytableaushort{{1,\!2}{3,\!4}4,{3,\!5}{5,\!6}}\,,
&
\psi(f_3 T) = f_3 \psi(T) & = \left(
  \raisebox{30pt}{$\ytableaushort{134,24,35,5,6}$}\,, \,
  \raisebox{30pt}{$\ytableaushort{\cdot\cdot\cdot,\cdot\cdot,12,2,4}$}
\right).
\end{align*}
\end{ex}

We show in the following example that the K-Bender--Knuth moves of~\cite{IS14} do not commute with the uncrowding bijection $\psi$. First, let us recall the \defn{K-Bender--Knuth} moves $K_i$ of~\cite{IS14}. A box $B$ is \defn{free} if it either $i \in B$ such that there is not an $i+1$ in the box below it or $i+1 \in B$ such that there is not an $i+1$ in the box above it. Note that if $i,i+1 \in B$, then $B$ is free. We define $K_i T$ row-by-row, reversing the free boxes in a row and interchanging $i$ and $i+1$ in each free box.

\begin{ex}
We have
\[
\ytableausetup{boxsize=1.5em}
K_2\!\left( \ytableaushort{1{1,\!2},3} \right) = \ytableaushort{1{1,\!3},2}\,,
\qquad
\psi\!\left( \ytableaushort{1{1,\!2},3} \right) = \left(\raisebox{12pt}{$\ytableaushort{11,2,3}$}\,, \; \raisebox{12pt}{$\ytableaushort{\cdot\cdot,\cdot,2}$} \right) = (S, F),
\]
but note that $t_2 S = S$. Indeed, we have
\pushQED{\qed}
\[
\psi\!\left( K_2\!\left( \ytableaushort{1{1,\!2},3} \right) \right) = \left( \ytableaushort{11,23}\,, \; \ytableaushort{\cdot\cdot,\cdot1} \right). \qedhere \popQED
\] \let\qed\relax
\end{ex}

There is a natural definition of \defn{K-evacuation} on $\svssyt^n(\lambda)$, given by applying K-Bender--Knuth moves in the analog of Equation~\eqref{eq:evac_BK} 
\[
K_1 (K_2 K_1) \cdots (K_{n-1} K_{n-2} \cdots K_2 K_1) T.
\]
For formal reasons (as in \cite{Stanley09}), K-evacuation is necessarily an involution.
As the next example shows, K-evacuation does not generally coincide with the Lusztig involution on $\svssyt^n(\lambda)$.

\begin{ex}
Consider $\svssyt^3(\lambda)$ for $\lambda = (2,1)$. Then we have
\begin{gather*}
K_1 K_2 K_1 \! \left( \ytableaushort{1{2,\!3},2} \right) = K_1 K_2 \! \left( \ytableaushort{1{1,\!3},2} \right)
= K_1 \! \left( \ytableaushort{1{1,\!2},3} \right)
= \ytableaushort{{1,\!2}2,3}\,,
\\
\left( \ytableaushort{1{2,\!3},2} \right)^* \! = \left( f_1 f_2 \, \ytableaushort{1{1,\!2},2} \right)^*
\! = e_2 e_1 \! \left( \ytableaushort{1{1,\!2},2} \right)^*
= e_2 e_1 \, \ytableaushort{2{2,\!3},3}
= \ytableaushort{1{2,\!3},2}\,.
\end{gather*}
\end{ex}

\section{Other combinatorial models}
\label{sec:diagrams}

In this section, we describe some other combinatorial models that can be used to construct Grothendieck polynomials. We also discuss crystal structures on these models and how they relate to the crystal structure on semistandard set-valued tableaux.

\subsection{Excited Young diagrams}

Let $D = \{1, \dotsc, n\} \times \ZZ_{>0}$, which we think of as a Young diagram of length $n$ with each row having an infinite number of boxes.\footnote{It is sufficient to consider $D$ as being an $n \times (\lambda_1 + n)$ rectangle. This is also effectively removing the flag condition on set-valued tableaux (other than a max-entry of $n$).} Note that we can identity partitions of length at most $n$ with subsets of boxes in $D$ by their Young diagrams. Next, we define \defn{elementary excitations} following~\cite{GK15}:
\[
\text{Type $1$:} \quad
\begin{tikzpicture}[xscale=0.5,yscale=-0.5,baseline=-17]
\fill[blue!30] (0,0) rectangle (1,1);
\draw[step=1] (0,0) grid (2,2);
\end{tikzpicture}
\longmapsto
\begin{tikzpicture}[xscale=0.5,yscale=-0.5,baseline=-17]
\fill[blue!30] (1,1) rectangle (2,2);
\draw[step=1] (0,0) grid (2,2);
\end{tikzpicture}\,,
\hspace{30pt}
\text{Type $2$:} \quad
\begin{tikzpicture}[xscale=0.5,yscale=-0.5,baseline=-17]
\fill[blue!30] (0,0) rectangle (1,1);
\draw[step=1] (0,0) grid (2,2);
\end{tikzpicture}
\longmapsto
\begin{tikzpicture}[xscale=0.5,yscale=-0.5,baseline=-17]
\fill[blue!30] (0,0) rectangle (1,1);
\fill[blue!30] (1,1) rectangle (2,2);
\draw[step=1] (0,0) grid (2,2);
\end{tikzpicture}\,.
\]
(Elementary excitations are closely related to the ``diagram marching moves'' of \cite{KY04}.)
We say the inverse of an elementary excitation is an \defn{elementary emission}.
An \defn{excited Young diagram} of shape $\lambda$ is a set of boxes in $D$ that can be obtained from $\lambda$ by a sequence of elementary excitations.
Let $\EYD^n(\lambda)$ denote the set of excited Young diagrams of shape $\lambda$.

Next, we recall the bijection $\Theta \colon \svssyt^n(\lambda) \to \EYD^n(\lambda)$ from~\cite{GK15}. Let $T(i,j)$ denote the box of $T \in \svssyt^n(\lambda)$ in the box in row $i$ and column $j$. Define
\[
\Theta(T) = \{ (i, i + c - r) \mid i \in T(r,c) \text{ for all } r \leq \ell(\lambda), c \leq \lambda_r \}.
\]
An immediate corollary is that we can compute the symmetric Grothendieck polynomial $\G_\lambda$ as a generating function for excited Young diagrams of shape $\lambda$.
Furthermore, note that the usual crystal operators on semistandard tableaux correspond to Type~$1$ excitations under $\Theta$.
Hence, the notion of ``reduced'' in~\cite{GK15} precisely identifies the excited Young diagrams that correspond to elements of the unique component $B(\lambda) \subseteq \svssyt^n(\lambda)$, which are also the set-valued tableaux of excess $0$.

To obtain the action of the crystal operator on semistandard set-valued tableaux under $\Theta$, we require the an additional type of elementary excitation:
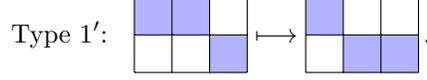
\[
\text{Type $1'$:} \quad
\begin{tikzpicture}[xscale=0.5,yscale=-0.5,baseline=-17]
\fill[blue!30] (0,0) rectangle (2,1);
\fill[blue!30] (2,1) rectangle (3,2);
\draw[step=1] (0,0) grid (3,2);
\end{tikzpicture}
\longmapsto
\begin{tikzpicture}[xscale=0.5,yscale=-0.5,baseline=-17]
\fill[blue!30] (0,0) rectangle (1,1);
\fill[blue!30] (1,1) rectangle (3,2);
\draw[step=1] (0,0) grid (3,2);
\end{tikzpicture}\,.
\]
Thus we can make $\Theta$ into a crystal isomorphism by constructing a signature rule for $e_i$/$f_i$ by reading the $i$-th row of the excited Young diagram from left-to-right and adding a $\bplus$ (resp.~$\bminus$) if we can perform a Type~$1$ or Type~$1'$ elementary excitation (resp.\ emission). The proof is straightforward and left for the interested reader.

\begin{prop}
\label{prop:EYD_crystal}
With the crystal structure on $\EYD^n(\lambda)$ given above, the map $\Theta \colon \svssyt^n(\lambda) \to \EYD^n(\lambda)$ is a crystal isomorphism.
\end{prop}
 
\begin{remark}
\label{rem:crystal_excitations}
Type~$2$ excitations correspond to adding an extra entry to a box to $T$ such that the resulting semistandard set-valued tableaux is in $\svssyt^n(\lambda)$.
In addition, the crystal operators are a (generally proper) subset of all possible Type~$1$ and Type~$1'$ excitations and emissions due to the bracketing rule.
\end{remark}

Note that Type~$1'$ moves are required in order to construct the full crystal structure and correspond to changing the sizes of the sets within boxes. However, by~\cite[Lemma~4.17]{GK15}, we do not require Type~$1'$ moves in order to construct every element in $\svssyt^n(\lambda)$; in fact, the proof provides an algorithm to compute a sequence of Type~$1$ and Type~$2$ elementary emissions to the ground state, the excited Young diagram corresponding to $\lambda$ (equivalently having no elementary emissions).

\subsection{Gelfand-Tsetlin patterns}

We recall a Gelfand--Tsetlin pattern description of semistandard set-valued tableaux. Define a \defn{horizontal strip} to be a skew partition that does not contain a vertical domino. Recall that a \defn{Gelfand--Tsetlin (GT) pattern} with top row $\lambda$ is a sequence of partitions $\Lambda = \left( \lambda^{(j)} \right)_{j=0}^n$ such that $\lambda^{(0)} = \emptyset$, $\lambda^{(n)} = \lambda$, and $\lambda^{(j)} / \lambda^{(j-1)}$ is a horizontal strip.\footnote{To see that this is equivalent to the usual interlacing condition, consider the bijection with semistandard Young tableaux and note that the boxes labeled $i$ must form a horizontal strip.} The weight of a GT pattern is $\wt(\Lambda) = \left( \lvert \lambda^{(j)} \rvert - \lvert \lambda^{(j-1)} \rvert \right)_{j=1}^n$.

\begin{dfn}
A \defn{marked Gelfand--Tsetlin (GT) pattern} is a GT pattern $\Lambda$ together with a set $M$ of entries that are ``marked,'' where the entry $(i, j)$, for $1 \leq i < \ell(\lambda^{(j)})$ and $2 \leq j \leq n$, is allowed to be marked if and only if $\lambda^{(j)}_{i+1} < \lambda^{(j-1)}_i$. The weight of a marked GT pattern $(\Lambda, M)$ is
\[
\wt(\Lambda, M) = \left( \bigl\lvert \lambda^{(j)} \bigr\rvert - \bigl\lvert \lambda^{(j-1)} \bigr\rvert + \bigl\lvert M^{(j)} \bigr\rvert \right)_{j=1}^n = \wt(\Lambda) + \left( \bigl\lvert M^{(j)} \bigr\rvert \right)_{j=1}^n,
\]
where $M^{(j)} = \{ i \mid (i, j) \in M \}$.
\end{dfn}

\begin{ex}
\label{ex:marked_GT}
The following is a marked GT pattern with top row $\lambda = (8, 7, 3, 1)$:
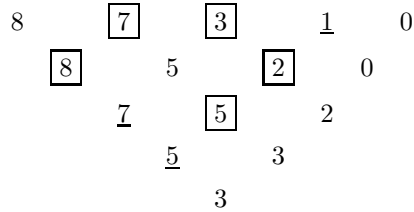
\[
\begin{array}{ccccccccc}
8 && \boxed{7} && \boxed{3} && \underline{1} && 0 \\[4pt]
& \boxed{8} && 5 && \boxed{2} && 0 \\[4pt]
&& \underline{7} && \boxed{5} && 2 \\[4pt]
&&& \underline{5} && 3 \\[4pt]
&&&& 3
\end{array}
\]
Here, we have depicted the marked entries by boxing them and we have underlined those entries that are not allowed to be marked. Note that we can never mark the rightmost entry in any row.
\end{ex}

We show that marked Gelfand--Tsetlin patterns give a combinatorial interpretation of~\cite[Cor.~3.6]{MS14} (which is equivalent to~\cite[Thm.~3.5]{MS14}).
First, we recall~\cite[Prop.~3.4]{MS14}, which states that if $\lambda / \mu$ is a horizontal strip, then
\begin{equation}
\label{eq:single_var_G}
\G_{\lambda / \mu}(x; \beta) = x^{\lvert \lambda \rvert - \lvert \mu \rvert} \prod_{i=1}^{\ell} \bigl(1 + \beta x (1 - \delta_{\lambda_{i+1}, \mu_i}) \bigr),
\end{equation}
where $\mu$ has length $\ell$, and is $0$ otherwise.

\begin{prop}
\label{prop:GT_version}
We have
\[
\G_{\lambda}(x_1, \dotsc, x_n; \beta) = \sum_{\Lambda} \prod_{j=1}^n \G_{\lambda^{(j)}/\lambda^{(j-1)}}(x_j; \beta) = \sum_{(\Lambda, M)} x^{\wt(\Lambda, M)}
\]
where we sum over all (marked) Gelfand--Tsetlin patterns $\Lambda = \left( \lambda^{(j)} \right)_{j=1}^n$ with top row $\lambda$.
\end{prop}

\begin{proof}
First, we construct a weight-preserving bijection between marked GT patterns of top row $\lambda$ and $\svssyt(\lambda)$, which implies
\[
\G_{\lambda}(x_1, \dotsc, x_n; \beta) = \sum_{(\Lambda, M)} x^{\wt(\Lambda, M)},
\]
where we sum over all marked GT patterns of shape $\lambda$.
Fix a marked GT pattern $(\Lambda, M)$. We construct an element in $\svssyt(\lambda)$ recursively.
Suppose we have added all entries $1, \dotsc, j-1$, which results in $T_{j-1}$.
We consider the horizontal strip $\lambda^{(j)} / \lambda^{(j-1)}$ as being filled with $j$, which we add to $T_{j-1}$. Additionally, for each marked entry $(i,j)$, we add $j$ to the rightmost entry of $i$-th row of $T_{j-1}$. The result is $T_j$. We note that this is the unique place we can insert $j$ to the $i$-th row. Furthermore, we note that if $\lambda^{(j)}_{i+1} = \lambda^{(j-1)}_i$ (the only other option by the interlacing condition), then we cannot insert $j$ into the $i$-th row. Hence, the process described above is reversible and gives the desired bijection.

Next, fix a GT pattern $\Lambda$. It is straightforward to see for $\mu = \lambda^{(j)}$ and $\nu = \lambda^{(j-1)}$
\[
\G_{\mu / \nu}(x_j; \beta) = x_j^{\lvert \mu \rvert - \lvert \nu \rvert} \sum_{M^{(j)}} (\beta x_j)^{\lvert M^{(j)} \rvert},
\]
where we sum over all valid markings $M^{(j)}$ of row $j$ of $\Lambda$. Note that the valid markings of each row of $\Lambda$ can be chosen independently. Thus, we have
\[
\sum_{(\Lambda, M)} x^{\wt(\Lambda, M)} = \sum_{\Lambda} \prod_{j=1}^n \left(  x_j^{\lvert \mu \rvert - \lvert \nu \rvert} \sum_{M^{(j)}} (\beta x_j)^{\lvert M^{(j)} \rvert} \right) =  \sum_{\Lambda} \prod_{j=1}^n \G_{\lambda^{(j)}/\lambda^{(j-1)}}(x_j; \beta),
\]
and the claim follows.
\end{proof}

\begin{ex}
Consider the marked GT pattern from Example~\ref{ex:marked_GT}, then the corresponding semistandard set-valued tableaux under the bijection from the proof of Proposition~\ref{prop:GT_version} is
\[
\ytableausetup{boxsize=2.2em}
\ytableaushort{111223{3,\!4}4,22{2,\!3}3{3,\!5}55,3{3,\!4,\!5}5,5}
\]
Note that the positions in row $j$ that cannot be marked in the GT pattern correspond to the boxes of the tableau where we cannot add a $j$ and remain semistandard.
\end{ex}

Given this interpretation and Equation~\eqref{eq:single_var_G}, we can compute a Grothendieck polynomial in terms of marked GT patterns in similar fashion to the Tokuyama formula for Whittaker functions (see, \textit{e.g.},~\cite{BBF11,BBCFG12}):
\[
\G_{\lambda}(x_1, \dotsc, x_n; \beta) = \sum_{\Lambda} x^{\wt(\Lambda)} \prod_{j=1}^n (1 + \beta x_j)^{m_j(\Lambda)},
\]
where we sum over all GT patterns with top row $\lambda$ and where $m_j(\Lambda)$ denotes the number of markable entries in row $j$ of $\Lambda$.
We can connect this with the $5$-vertex model approach to Grothendieck polynomials of, \textit{e.g.},~\cite{GK17,MS13,MS14}. Indeed, configurations of the $5$-vertex model with natural boundary conditions that depend on $\lambda$ are in natural bijection with GT patterns in exactly the same way as the $6$-vertex model. This is simply a translation of~\cite[Cor.~3.6]{MS14}.

Note that GT patterns, and hence $5$-vertex configurations, have a natural crystal structure coming from the bijection with semistandard tableaux. On $5$-vertex configurations, this was explicitly described in~\cite{EV17}. This crystal structure is a ``coarse'' version of the crystals on semistandard set-valued tableaux obtained by grouping together multiple terms.

\section{Hecke insertion}
\label{sec:hecke}

We recall Hecke insertion from~\cite{BKSTY08} and mention how it can be used to describe a crystal structure on stable Grothendieck polynomials for any\footnote{Recall that partitions are in bijection with Grassmannian permutations.} permutation in analogy to~\cite{MS16} with affine Stanley symmetric functions (which include stable Schubert polynomials).

We require the \defn{$0$-Hecke monoid} $\HH_0(n)$, the monoid of all finite words in the alphabet $[n]$ subject to the relations
\begin{align*}
pp & \equiv p & & \text{for all } p, \\
pqp & \equiv qpq & & \text{for all } p,q, \\
pq & \equiv qp & & \text{if } \lvert p - q \rvert > 1.
\end{align*}
Let $a$ be a finite word in $[n-1]$ and let $w(a)$ be the natural projection of $a$ into $\HH_0(n)$. Note that the reduced words for any $h \in \HH_0(n)$ correspond to the reduced words of some corresponding permutation $w \in \sym_n$.
Recall that an increasing tableau is a semistandard tableau that is strictly increasing along each row and column.

Consider a two-line array:
\[
\left[
\begin{array}{cccccccccccc}
1 & \cdots & 1 & 1 & 2 & \cdots & 2 & 2 & \cdots & m & \cdots & m \\
a_{1\ell_1} & \cdots & a_{12} & a_{11} & a_{2\ell_2} & \cdots & a_{22} & a_{21} & \cdots & a_{m\ell_m} & \cdots & a_{m1}
\end{array}
\right],
\]
where $1 \leq a_{k1} < a_{k2} < \cdots < a_{k\ell_k} < n$ for all $1 \leq k \leq m$ (with possibly $\ell_k = 0$). Note that this is a reformulation of the notion of a compatible pair of words from~\cite{BKSTY08}. Start with $(P_0, Q_0)$ being the empty increasing tableau and semistandard set-valued tableau, respectively. We recursively construct $(P_{i+1}, Q_{i+1})$ from $(P_i, Q_i)$ as follows. Suppose the $i$-th column in the array is $[k_i, a_i]^T$. We insert $a_i$ into $P_i$ column-by-column by the following algorithm. Suppose we are trying to insert the letter $x$ into column $C$.
\begin{itemize}
\item If $x \geq y$ for all $y \in C$, then: If we can append $x$ to the current column and obtain an increasing tableau, the result is $P_{i+1}$, and form $Q_{i+1}$ by adding a box with $k_i$ to $Q_i$ in the corresponding position. Otherwise set $P_{i+1} = P_i$, and form $Q_{i+1}$ by inserting $k_i$ into the corner of $Q_i$ whose row contains the bottom entry of $C$. In either case, terminate.
\item Otherwise: Let $y = \min \{ y' \in C \mid y' > x \}$. Replace $y$ with $x$ if the result is still an increasing tableau. In either case, proceed by inserting $y$ into the next column.
\end{itemize}

\begin{remark}
The row version of Hecke insertion considered in, \textit{e.g.},~\cite{PP16,sage,TY11} produces a recording tableau $Q$ that is conjugate to a set-valued tableau.
\end{remark}

The Hecke insertion process is reversible by moving right-to-left for the largest integer in $Q$ (see~\cite[Sec.~3.2]{BKSTY08}).

\begin{thm}[{\cite[Thm.~4]{BKSTY08}}]
Two-line arrays $(\mathbf{k}, \mathbf{a})$ are in bijection under Hecke insertion with pairs $(P, Q)$ such that $P$ is an increasing tableau and $Q$ is a semistandard set-valued tableau of the same shape as $P$.
\end{thm}

\begin{remark}
\label{rem:Hecke_to_EG}
We note that valid Hecke two-line arrays are reversed versions of valid Edelman--Greene two-line arrays from, \textit{e.g.},~\cite{MS16} (equivalently the usual RSK two-line arrays). In other words, to recover Edelman--Greene insertion $\iota_{EG}$ from Hecke insertion $\iota_H$, we replace the two-line array $[\mathbf{k}, \mathbf{a}]^T$ of length $\ell$ with the two-line array $[\mathbf{k}^r, \mathbf{a}^r]^T$, where $k_{\ell+1-i}^r = m + 1 - k_i$ and $a_{\ell+1-i}^r = a_i$. However, for $\iota_H(\mathbf{k}, \mathbf{a}) = (P, Q)$, we have $\iota_{EG}(\mathbf{k}^r, \mathbf{a}^r) = (P, Q^*)$, where $Q^*$ is the image of $Q$ under the Sch\"utzenberger involution. This follows from~\cite[Cor.~7.21, 7.22]{EG87} and that the Hecke insertion algorithm is done column-wise as opposed to row-wise for Edelman--Greene insertion.
\end{remark}

\begin{ex}
We have
\begin{align*}
\ytableausetup{boxsize=1.3em}
\left( \raisebox{18pt}{$\ytableaushort{134,24,3,4}$}\,, \raisebox{18pt}{$\ytableaushort{123,23,3,4}$} \right)
& \xleftarrow[\hspace{30pt}]{EG}
\begin{bmatrix}
1 & 2 & 2 & 3 & 3 & 3 & 4 \\
4 & 3 & 4 & 2 & 3 & 4 & 1
\end{bmatrix},
\\
\left( \raisebox{18pt}{$\ytableaushort{134,24,3,4}$}\,, \raisebox{18pt}{$\ytableaushort{122,23,3,4}$} \right)
& \xleftarrow[\hspace{30pt}]{H}
\begin{bmatrix}
1 & 2 & 2 & 2 & 3 & 3 & 4 \\
1 & 4 & 3 & 2 & 4 & 3 & 4
\end{bmatrix},
\end{align*}
and note that the evacuation of the recording tableau under Hecke insertion agrees with the recording tableau under Edelman--Greene insertion.
\end{ex}

Let $w \in \sym_n$. Following~\cite{FK94}, the \defn{stable Grothendieck polynomial} (or \defn{K-Stanley symmetric function}) for $w$ is defined as
\begin{equation}\label{eq:Kstanley}
\G_w(\xx) := \sum_{(\mathbf{k},\mathbf{a})} \beta^{\ell(\mathbf{a}) - \ell(w)} \xx^{\mathbf{k}},
\end{equation}
where we sum over all two-line arrays $[\mathbf{k}, \mathbf{a}]^T$ such that $\mathbf{a} \equiv w$. Note that the stable Grothendieck polynomial is actually a power series, not a polynomial, and that the $\beta=0$ specialization is the corresponding Stanley symmetric function $F_w$~\cite{Stanley84}. Equation~\eqref{eq:Kstanley} agrees with the definition given in~\cite{Lam06,LSS10} (when restricted to $\sym_n$) in terms of  elements of the $\beta$-deformed $0$-Hecke algebra (\textit{i.e.}, where $T_i^2 = \beta T_i$ for any standard generator $T_i$). Indeed, we have a bijection between two-line arrays and decreasing factorizations given by
\[
\begin{bmatrix} \mathbf{k} \\ \mathbf{a} \end{bmatrix} \longleftrightarrow (a_{1\ell_1} \cdots a_{11}) \cdots (a_{m\ell_m} \cdots a_{m1}),
\]
where an $a$ represents the simple transposition $s_a$ in the decreasing factorization. Note, that this means if we want to consider this as Hecke insertion of $\mathbf{a}$, we need to insert from left to right. So we write this as $(P, Q) \xleftarrow{H} \mathbf{a}$.

Hence, the crystal structure on semistandard set-valued tableaux from Section~\ref{sec:crystal_set_valued} can be used to determine the decomposition of a stable Grothendieck polynomial into Schur functions via Hecke insertion. This likely naturally recovers the rule given by~\cite{FG98}: If $\G_w = \sum_{\lambda} \beta^{\lvert \lambda \rvert - \ell(w)} g_{w\lambda} s_{\lambda}$, then
\[
g_{w\lambda} = \lvert \{ T \in \ssyt^n(\lambda') \mid w_C(T) \equiv w \} \rvert,
\]
where $w_C(T)$ is the column reading word of $T$.

When the product of the cyclically decreasing factors is a reduced expression, we can define a crystal structure on the decreasing factorizations by using a Lusztig dual version of the crystal operators of~\cite{MS16}. This is necessary from the fact that the Sch\"utzenberger involution is the Lusztig involution~\cite{Lenart07} as per Remark~\ref{rem:Hecke_to_EG}.

First, we consider two decreasing factors $u,v$ (considered as words in $[n]$).
Define the \defn{pairing} of $u$ and $v$ as follows:
Pair the smallest $b \in v$ with the largest $a < b$ in $u$, and if no such $a$ exists, then $b$ is unpaired. We proceed in increasing order on letters in $v$, ignoring previously paired letters in $u$.

Define $f_i$ on a decreasing factorization $w^1 \cdots w^m$ by considering the factors $w^i w^{i+1}$.
We move the smallest unpaired letter in $w^i$ to $w^{i+1}$, decreasing it as necessary so that $w^{i+1}$ is a decreasing word.

As explained in~\cite[\S3.3]{MS16}, the above crystal structure corresponds to applying the usual braid relations, which under EG insertion corresponds to changing an $i$ to an $i+1$ within a box. However, it is not the case in general that this crystal structure corresponds to the $f_i$ action on semistandard set-valued tableau (see Example~\ref{ex:non_local_hecke} below) as there is extra bracketing that can come from the set-valued entry. Despite this, anytime an $i$ changes to an $i+1$ within a box, it should correspond to doing a sequence of braid moves given that Hecke insertion is a generalization of EG insertion. Furthermore, this suggests that one approach to obtaining a K-theory analog would be to have additional bracketing so that the crystal operators come from the equivalence
\begin{equation}
\label{eq:hecke_equivalence}
ppq \equiv pq \equiv pqq.
\end{equation}

\begin{ex}
Consider type $A_1$. The following is an example of the crystal operator on Hecke words:
\[
\ytableausetup{boxsize=1.4em} 
\begin{tikzpicture}
\node (t1) at (-3.5,0) {$\left( \raisebox{-4pt}{$\ytableaushort{12}\,, \ytableaushort{1{1,\!2}}$} \right)$};
\node (t2) at (-3.5,-2) {$\left( \raisebox{-4pt}{$\ytableaushort{12}\,, \ytableaushort{{1,\!2}2}$} \right)$};
\node (w1) at (3,0) {$(2 \, 1)(1)$};
\node (w2) at (3,-2) {$(2)(2 \, 1)$};
\node (b1) at (0,0) {$\begin{bmatrix} 1 & 1 & 2 \\ 2 & 1 & 1 \end{bmatrix}$};
\node (b2) at (0,-2) {$\begin{bmatrix} 1 & 2 & 2 \\ 2 & 2 & 1 \end{bmatrix}$};
\draw[->] (b1) -- node[midway,above] {$H$} (t1);
\draw[->] (b2) -- node[midway,above] {$H$} (t2);
\draw[<->] (b1) -- (w1);
\draw[<->] (b2) -- (w2);
\draw[->] (t1) -- node[midway,right] {$f_1$} (t2);
\draw[->] (w1) -- node[midway,right] {$f_1$} (w2);
\draw[->] (b1) -- node[midway,right] {$f_1$} (b2);
\end{tikzpicture}
\]
Note the application of Equation~\eqref{eq:hecke_equivalence} to obtain the new decreasing factorization.
Furthermore, we have the trivial representation given by the element
\[
\ytableausetup{boxsize=1.4em}
\left(
\ytableaushort{12,2}\,, \ytableaushort{1{1,\!2},2}
\right)
\xleftarrow[\hspace{20pt}]{H}
\begin{bmatrix} 1 & 1 & 2 & 2 \\ 2 & 1 & 2 & 1 \end{bmatrix}
\longleftrightarrow
(2 \, 1) (2 \, 1).
\]
So every letter should be paired by some procedure, and similarly for
\[
\ytableausetup{boxsize=1.4em}
\left(
\ytableaushort{123,23}\,, \ytableaushort{11{1,\!2},22}
\right)
\xleftarrow[\hspace{20pt}]{H}
\begin{bmatrix} 1 & 1 & 1 & 2 & 2 & 2 \\ 3 & 2 & 1 & 3 & 2 & 1 \end{bmatrix}
\longleftrightarrow
(3 \, 2 \, 1) (3 \, 2 \, 1).
\]
However, such a pairing procedure is not simply based on whether the resulting word is reduced or not:
\begin{align*}
\ytableausetup{boxsize=1.4em}
\left(
\ytableaushort{1278,278}\,, \ytableaushort{111{1,\!2},222}
\right)
& \xleftarrow[\hspace{20pt}]{H}
\begin{bmatrix} 1 & 1 & 1 & 1 & 2 & 2 & 2 & 2 \\ 8 & 7 & 2 & 1 & 8 & 7 & 2 & 1 \end{bmatrix}
\allowdisplaybreaks \\
\ytableausetup{boxsize=1.4em}
\left(
\ytableaushort{12678,278}\,, \ytableaushort{11112,222}
\right)
& \xleftarrow[\hspace{20pt}]{H}
\begin{bmatrix} 1 & 1 & 1 & 1 & 2 & 2 & 2 & 2 \\ 8 & 7 & 2 & 1 & 8 & 6 & 2 & 1 \end{bmatrix}
\allowdisplaybreaks \\
\ytableausetup{boxsize=1.4em}
\left(
\ytableaushort{12378,278}\,, \ytableaushort{11112,222}
\right)
& \xleftarrow[\hspace{20pt}]{H}
\begin{bmatrix} 1 & 1 & 1 & 1 & 2 & 2 & 2 & 2 \\ 8 & 7 & 3 & 1 & 8 & 7 & 2 & 1 \end{bmatrix}
\end{align*}
\end{ex}

Supposing such a bracketing could be given, the natural approach would be to consider a pair of decreasing words such that the product is not a reduced expression in the Hecke monoid.
%
%
%
Yet, unlike in~\cite{MS16}, any definition of the crystal operators on decreasing factorizations cannot be described by reducing to a pair of decreasing words. We observe this in the following examples.

\begin{ex}
\label{ex:non_local_hecke}
We thank Jianping Pan for this example. Consider the following:
\[
\ytableausetup{boxsize=1.4em} 
\begin{tikzpicture}
\node (t1) at (-5,0) {$\left( \ytableaushort{123,23}\,, \ytableaushort{11{2,\!3},22} \right)$};
\node (t2) at (-5,-2) {$\left( \ytableaushort{123,23}\,, \ytableaushort{11{2,\!3},23} \right)$};
\node (w1) at (4,0) {$(2 \, 1)(3 \, 2 \, 1)(1)$};
\node (w2) at (4,-2) {$(3 \, 2)(3 \, 1)(2 \, 1)$};
\node (b1) at (0,0) {$\begin{bmatrix} 1 & 1 & 2 & 2 & 2 & 3 \\ 2 & 1 & 3 & 2 & 1 & 1 \end{bmatrix}$};
\node (b2) at (0,-2) {$\begin{bmatrix} 1 & 1 & 2 & 2 & 3 & 3 \\ 3 & 2 & 3 & 1 & 2 & 1 \end{bmatrix}$};
\draw[->] (b1) -- node[midway,above] {$H$} (t1);
\draw[->] (b2) -- node[midway,above] {$H$} (t2);
\draw[<->] (b1) -- (w1);
\draw[<->] (b2) -- (w2);
\draw[->] (t1) -- node[midway,right] {$f_2$} (t2);
\draw[->] (w1) -- node[midway,right] {$f_2$} (w2);
\draw[->] (b1) -- node[midway,right] {$f_2$} (b2);
\end{tikzpicture}
\]
We have the sequence of braid moves
$
213211 \equiv 231211 \equiv 232121 \equiv 323121
$
which corresponds to doing a partial dual-equivalence move $f_1 f_2 e_1$ (see~\cite{Assaf08}) when removing the final letter from the word (which is a reduced word).
\end{ex}

\begin{ex}
Consider the following $2$-string:
\begin{align*}
\ytableausetup{boxsize=1.4em}
\left(
\ytableaushort{12367,234}\,, \ytableaushort{1112{2,\!3},22{2,\!3}}
\right)
& \xleftarrow[\hspace{20pt}]{H}
\begin{bmatrix} 1 & 1 & 1 & 2& 2 & 2 & 2 & 2 & 3 & 3 \\ 3 & 2 & 1 & 7 & 6 & 4 & 3 & 1 & 3 & 2 \end{bmatrix}
\allowdisplaybreaks \\
\ytableausetup{boxsize=1.4em}
\left(
\ytableaushort{12367,234}\,, \ytableaushort{111{2,\!3}3,22{2,\!3}}
\right)
& \xleftarrow[\hspace{20pt}]{H}
\begin{bmatrix} 1 & 1 & 1 & 2& 2 & 2 & 2 & 3 & 3 & 3 \\ 3 & 2 & 1 & 7 & 6 & 4 & 1 & 4 & 3 & 2 \end{bmatrix}
\allowdisplaybreaks \\
\ytableausetup{boxsize=1.4em}
\left(
\ytableaushort{12367,234}\,, \ytableaushort{111{2,\!3}3,2{2,\!3}3}
\right)
& \xleftarrow[\hspace{20pt}]{H}
\begin{bmatrix} 1 & 1 & 1 & 2& 2 & 2 & 3 & 3 & 3 & 3 \\ 7 & 3 & 2 & 7 & 6 & 1 & 6 & 4 & 3 & 2 \end{bmatrix}
\allowdisplaybreaks \\
\ytableausetup{boxsize=1.4em}
\left(
\ytableaushort{12367,234}\,, \ytableaushort{111{2,\!3}3,{2,\!3}33}
\right)
& \xleftarrow[\hspace{20pt}]{H}
\begin{bmatrix} 1 & 1 & 1 & 2& 2 & 3 & 3 & 3 & 3 & 3 \\ 7 & 3 & 2 & 7 & 1 & 7 & 6 & 4 & 3 & 2 \end{bmatrix}
\end{align*}
In particular, note that the second application of $f_2$ resulted in a change to the first decreasing factorization, in addition to the second and third.
\end{ex}

\begin{prob}\label{prob:Hecke_crystal}
Determine an explicit $U_q(\fsl_n)$-crystal structure on decreasing factorizations such that Hecke insertion is a crystal isomorphism with the crystal structure on semistandard set-valued tableaux.
\end{prob}

Such a $U_q(\fsl_n)$-crystal structure would give a new interpretation of the coefficients $\g_{w\lambda}$ in the formula $\G_w(\xx) = \sum_{\lambda} \beta^{\lvert \lambda \rvert - \ell(w)} g_{w\lambda} s_{\lambda}(\xx)$ given in~\cite{FG98}.
Furthermore, based on experimental evidence and Remark~\ref{rem:Hecke_to_EG}, we have the following conjecture.

\begin{conj}
The Lusztig dual of the crystal structure given in~\cite{MS16} under Hecke insertion corresponds to changing $i \leftrightarrow n+1-i$ in every box and reversing each row.
\end{conj}

\section{K-jeu de taquin}
\label{sec:K_jdt}

We propose a K-theoretic analog of jeu de taquin (K-jdt) for semistandard set-valued skew tableaux. Recall that a skew tableau is a filling of a skew partition $\lambda / \mu$, where $\mu$ is contained within $\lambda$.

First, we recall the classical \defn{jeu de taquin (jdt)} for semistandard tableaux, using a somewhat nonstandard definition. For a nice proof that our definition is equivalent to the standard one, see~\cite[Sec.~3]{CGP16}. 

Let $T \in \ssyt^n(\lambda/\mu)$ be a tableau that we wish to rectify to some $R \in \ssyt^n(\nu)$ for some partition $\nu$. First, for any $U \in \ssyt^m(\mu)$, let $\overline{U}$ be the tableau where we replace $k \mapsto \overline{k}$. Next, we construct a \defn{layered} tableau of shape $\lambda$ as
\[
U \sqcup T :=
\begin{tikzpicture}[baseline=-30]
\draw (0,0) -- (2,0) -- (0,-2) -- cycle;
\draw (1,0) -- (0,-1);
\draw (0.3,-0.3) node {$\overline{U}$};
\draw (0.7,-0.7) node {$T$};
\end{tikzpicture}.
\]
Note that $U \sqcup T$ is a semistandard tableau of shape $\lambda$ in the totally ordered alphabet
\[
\bon < \btw < \cdots \overline{m} < 1 < 2 < \cdots < n.
\]
We define operators $b_i$ on tableaux whose entries are in a totally ordered alphabet $\mathcal{A}$.  Let $j$ be the letter of $\mathcal{A}$ immediately greater than $i$. Then, $b_i$ acts by first applying the Bender--Knuth operator $t_i$ to the labels $i$ and $j$, and then switching all instances of $i$ and $j$. Note that the result of $b_i$ has the same content as the original layered tableau and is a semistandard tableau of shape $\lambda$ in a modified alphabet where we have swapped the order of $i$ and $j$. 
By analogy with~\cite{TY09}, we refer to the automorphism
\begin{equation}\label{eq:infusion}
\Inf := b_{\bon}^{\circ n} \circ \cdots \circ b_{\overline{m-1}}^{\circ n} \circ b_{\overline{m}}^{\circ n}
\end{equation} 
on layered semistandard tableaux as \defn{infusion}.
We remark that this is an example of a tableau switching algorithm~\cite{BSS96,Haiman92}.
Observe that $\Inf(L)$ is a semistandard tableau of shape $\lambda$ in the totally ordered alphabet
\[
1 < 2 < \cdots < n < \bon < \btw < \cdots < \overline{m},
\]
where the barred letters are now greater than the unbarred letters.
The \defn{rectification} $\rect_U(T)$ of $T$ (with respect to the rectification order $U$) is the semistandard tableau obtained by restricting $\Inf(U \sqcup T)$ to the unbarred alphabet.  
Intuitively, we have pushed $T$ and $U$ through each other and $\rect_U(T)$ is the result of this pushing on $T$. A key feature of rectification is that the result is independent of the choice of $U$, a property known as confluence.


Rectification gives a well-known combinatorial rule for the Littlewood--Richardson coefficients $c_{\lambda\mu}^{\nu}$. 

\begin{thm}\label{thm:lr}
The Littlewood--Richardson coefficient, $c_{\lambda\mu}^{\nu}$ counts ordered pairs of semistandard tableaux $T \in \ssyt^n(\lambda)$ and $S \in \ssyt^n(\mu)$ such that $u_{\nu} \in \ssyt^n(\nu)$ is the rectification of the skew tableau
\[
T \ast S :=
\begin{tikzpicture}[baseline=0]
\draw (0,0) -- (0,1) -- (1,1) -- cycle;
\draw (0,0) -- (-1,0) -- (-1,-1) -- cycle;
\draw (0.3,0.7) node {$S$};
\draw (-0.7,-0.3) node {$T$};
\end{tikzpicture}
\]
of shape $\lambda \ast \mu$.
\end{thm}

Since the application of $b_{\overline{i}}^n$ during rectification preserves the semistandardness of the unbarred subtableau, it is easy to see that rectification can be considered as a crystal isomorphism (see, \textit{e.g.}~\cite[Thm.~3.3.1]{vanLeeuwen01}). Moreover, this implies that ${\rm rect}(T \ast S) = u_\nu$ if and only if $T \ast S$ is a highest weight element of weight $\nu$.\footnote{Highest weight skew tableaux are also called \defn{Yamanouchi}, \defn{ballot}, or \defn{lattice}.}
Additionally, recall that the target $u_{\nu}$ can be replaced by any other fixed element of $\ssyt^n(\nu)$.

In the K-theory setting, there are several jdt analogs for various notions of tableaux. H.~Thomas--A.~Yong gave a K-jdt on increasing tableaux~\cite{TY09}.  A closely related K-jdt rule was given on genomic tableaux (which can be considered as semistandard analogs of increasing tableaux) by O.~Pechenik--A.~Yong~\cite{PY:genomic}. However, unlike for classical jdt, K-jdt for increasing and genomic tableaux do not generally have a confluence property (see, \textit{e.g.},~\cite{BS16, GMPPRST16, TY09}). To obtain K-analogs of Littlewood--Richardson rules using K-jdt for increasing or genomic tableaux, one needs either to rectify with respect to a carefully chosen rectification order or to choose the target tableau carefully. (In fact, in both cases an analog of $u_\nu$ is a valid choice of target.) Indeed, the equivariant rules of \cite{PY:KT,TY18} require making particular choices of rectification order and target tableau.

However, increasing and genomic tableaux are ``dual'' to semistandard set-valued tableaux with respect to  Hecke insertion.  There are no known bijections between genomic tableaux and semistandard set-valued tableaux, except in the case of the tableaux  that compute the structure coefficients $C_{\lambda\mu}^{\nu}$~\cite[Sec.~5.2]{PY:genomic}. The K-jdt rules for analogs of Littlewood--Richardson coefficients require rectifying skew tableaux of shape $\nu / \lambda$ instead of shape $\lambda \ast \mu$. There is no known rule that directly extends Theorem~\ref{thm:lr} to K-theory.

Toward such a rule, we wish to construct a K-jdt for semistandard set-valued tableaux. 
We propose that one should rectify set-valued tableaux by the algorithm described above for the usual semistandard tableaux, replacing everywhere the Bender--Knuth operators $t_i$ with the K-Bender--Knuth operators $K_i$ of T.~Ikeda--T.~Shimazaki~\cite{IS14} (see the end of Section~\ref{sec:crystal_set_valued} for the definition of $K_i$). That is, to rectify $T \in \svssyt^n(\lambda / \mu)$, we choose $U \in \svssyt^m(\mu)$ and form the layered semistandard set-valued tableau $U \sqcup T$. Then, modifying Equation~\eqref{eq:infusion}  to use the K-Bender--Knuth operator $K_i$ in place of $b_i$, we get $\rect_U(T)$ by applying the infusion automorphism $\Inf$ to $U \sqcup T$ and restricting to the unbarred alphabet.

\begin{ex}
Consider the tableaux
\[
\ytableausetup{boxsize=1.5em}
T = \ytableaushort{ {}{}{2}{2}, {}{1,\!2}, {1}}\,,
\qquad\qquad
\overline{U} = \ytableaushort{ {\bon} {\bon,\! \bth}, {\btw,\!\bth} }\,.
\]
We have
\begin{align*}
\Inf(U \sqcup T) = \ytableaushort{ {\bon} {\bon,\!\bth}{2}{2}, {\btw,\!\bth}{1,\!2}, {1}}
& \imapsto{\bth}
\ytableaushort{ {\bon}{\bon,\!1}{2}{2}, {\btw,\!1} {\bth, \!2}, {\bth}}
\imapsto{\bth}
\ytableaushort{ {\bon} {\bon,\! 1} {2} {2}, {\btw,\! 1}{2,\!\bth}, {\bth}}
\\ & \imapsto{\btw}
\ytableaushort{ {\bon}{\bon,\!1}{2}{2}, {1,\! \btw}{2,\!\bth}, {\bth}}
\imapsto{\btw}
\ytableaushort{ {\bon}{\bon, \!1}{2}{2}, {1,\!2}{\btw,\! \bth}, {\bth}}
\\ & \imapsto{\bon}
\ytableaushort{ {1} {1, \! \bon} {2} {2}, {\bon,\! 2} {\btw,\! \bth}, {\bth}}
\imapsto{\bon}
\ytableaushort{ {1}{1,\!2}{2}{\bon}, {2,\!\bon}{\btw,\!\bth}, {\bth}}
\end{align*}
Hence, $\rect_U(T) = \ytableaushort{1 {1,\!2} 2,2}$\,.
\end{ex}

Since K-jdt for increasing and genomic tableaux are not generally confluent, it is not surprising that our proposed K-jdt for set-valued tableaux is not generally confluent either. As for the other jdt rules for equivariant cohomology and (equivariant) K-theory, we suspect that an extension of Theorem~\ref{thm:lr} into K-theory would require using special rectification orders and perhaps also special rectification targets.

\begin{ex}
Consider
\[
\ytableausetup{boxsize=1.5em}
T = \ytableaushort{ {}, {1,\!2} }\,,
\qquad\qquad
\overline{U} = \raisebox{-4pt}{$\ytableaushort{{\bon}}$}\,,
\qquad\qquad
\overline{V} = \raisebox{-4pt}{$\ytableaushort{{\bon,\!\btw}}$}\,.
\]
Then, we have
\begin{align*}
\rect_U(T) & = \ytableaushort{\bon,{1,\!2}} \imapsto{\btw} \ytableaushort{\bon,{1,\!2}} \imapsto{\btw} \ytableaushort{\bon,{1,\!2}} \imapsto{\bon} \ytableaushort{1,{\bon,\!2}}
\imapsto{\bon} \ytableaushort{1,{2,\!\bon}} \; \longmapsto \; \ytableaushort{1,2}\,,
\\
\rect_V(T) & = \ytableaushort{{\bon,\!\btw},{1,\!2}} \imapsto{\btw} \ytableaushort{{\bon,\!1},{\btw,\!2}}
\imapsto{\btw} \ytableaushort{{\bon,\!1},{2,\!\btw}} \imapsto{\bon} \ytableaushort{{1,\!\bon},{2,\!\btw}}  \imapsto{\bon} \ytableaushort{{1,\!2},{\bon,\!\btw}}  \; \longmapsto \; \raisebox{-4pt}{$\ytableaushort{{1,\! 2}}$}\,.
\end{align*}
\end{ex}

However, we have the following.

\begin{prop}
K-jdt with respect to any fixed rectification order is a $U_q(\fsl_n)$-crystal isomorphism.
\end{prop}

\begin{proof}
Note that during every application of $b_{\overline{k}}$, the $\overline{k} \leftrightarrow i$ if and only if $\overline{k}$ and $i$ are not free.
Therefore, after every application of $b_{\overline{k}}$, the resulting set-valued tableau remains semistandard (of skew shape) and the (reduced) $j$-signature remains unchanged. Hence, the claim follows from the signature rule.
\end{proof}

Ideally, we would like to obtain the Grothendieck product expansion $\G_\lambda \cdot \G_\mu = \sum_\nu C_{\lambda,\mu}^\nu \G_{\nu}$, analogously to Theorem~\ref{thm:lr}, by taking ordered pairs of tableaux $T \in \svssyt(\lambda), S \in \svssyt(\mu)$, rectifying with respect to  some rectification orders, and obtaining exactly the tableau $R \in \svssyt(\nu)$ with the correct multiplicity. In examples, this in fact appears to be possible using ad hoc rectification orders; we do not know a uniform choice of rectification order that works in general.

In the special case $\lambda = \mu = \Lambda_1$, we have identified a candidate for such a uniform rectification order. For $S,T \in \svssyt(\Lambda_1)$, if $\min S < \min T$, let $k$ be the number of $i \in S$ with $i < \min T$, and if $\min T \leq \min S$, let $k$ be the number of $j \in T$ with $j \leq \min S$. Then, we believe it suffices to use the rectification order $U = \raisebox{-4pt}{$\ytableaushort{X}$}$, where $X = \{ \overline{1}, \dotsc, \overline{k} \}$.
More directly, the rectification with respect to this $U$ is given by
\begin{equation}
\label{eq:K_jdt}
\ytableausetup{boxsize=1.4em}
\ytableaushort{{ }A,B} = \begin{cases}
\ytableaushort{{A'}{A_-},B} & \text{if } \min A < \min B, \\[15pt]
\ytableaushort{{B'}A,{B_-}} & \text{if } \min B \leq \min A,
\end{cases}
\end{equation}
where we construct $A'$ (resp.~$B'$) by taking as many values from $A$ such that the result is a semistandard set-valued tableau and $C_- := C \setminus C'$ for $C = A,B$. For an example, see Table~\ref{table:K_jdt_ex}; we note that K-jdt on $T \ast S$ for this example agrees with what we call \defn{Buch insertion} $S \xleftarrow{B} T$ in honor of A.~Buch, who introduced it in~\cite[Def.~4.1]{Buch02}.

\begin{table}
\[
\ytableausetup{boxsize=1.5em}
\begin{array}{c|ccccccccc}
 & 1 & 2 & 3 & 1,2 & 1,3 & 2,3 & 1,2,3
\\ \hline
\\[-7pt] 
1 & \ytableaushort{11} & \ytableaushort{12} & \ytableaushort{13} & \ytableaushort{1{1,\!2}} & \ytableaushort{1{1,\!3}} & \ytableaushort{1{2,\!3}} & \ytableaushort{1X}
\\[2pt]
2 & \ytableaushort{1,2} & \ytableaushort{22} & \ytableaushort{23} & \ytableaushort{12,2} & \ytableaushort{13,2} & \ytableaushort{2{2,\!3}} & \ytableaushort{1{2,\!3},2}
\\[15pt]
3 & \ytableaushort{1,3} & \ytableaushort{2,3} & \ytableaushort{33} & \ytableaushort{{1,\!2},3} & \ytableaushort{13,3} & \ytableaushort{23,3} & \ytableaushort{{1,\!2}3,3}
\\[15pt]
1,2 & \ytableaushort{11,2} & \ytableaushort{{1,\!2}2} & \ytableaushort{{1,\!2}3} & \ytableaushort{1{1,\!2},2} & \ytableaushort{1{1,\!3},2} & \ytableaushort{{1,\!2}{2,\!3}} & \ytableaushort{1X,2}
\\[15pt]
1,3 & \ytableaushort{11,3} & \ytableaushort{12,3} & \ytableaushort{{1,\!3}3} & \ytableaushort{1{1,\!2},3} & \ytableaushort{1{1,\!3},3} & \ytableaushort{1{2,\!3},3} & \ytableaushort{1X,3}
\\[15pt]
2,3 & \ytableaushort{1,{2,\!3}} & \ytableaushort{22,3} & \ytableaushort{{2,\!3}3} & \ytableaushort{12,{2,\!3}} & \ytableaushort{13,{2,\!3}} & \ytableaushort{2{2,\!3},3} & \ytableaushort{1{2,\!3},{2,\!3}}
\\[15pt]
1,2,3 & \ytableaushort{11,{2,\!3}} & \ytableaushort{{1,\!2}2,3} & \ytableaushort{X3} & \ytableaushort{1{1,\!2},{2,\!3}} & \ytableaushort{1{1,\!3},{2,\!3}} & \ytableaushort{{1,\!2}{2,\!3},3} & \ytableaushort{1X,{2,\!3}}
\end{array}
\]
\caption{The K-jdt on two boxes $T \ast S$, with top row $S$ and left column $T$ for $n = 3$, where $X = \{1,2,3\}$.}
\label{table:K_jdt_ex}
\end{table}

\begin{conj}
Let $S,T \in \svssyt^n(\Lambda_1)$. The K-jdt rule given by Equation~\eqref{eq:K_jdt} to rectify $T \ast S$ agrees with the Buch insertion $S \xleftarrow{B} T$.
\end{conj}

It is not clear to us how one should choose rectification orders for tableaux of more general shapes.
The na\"ive generalization of our choice for the $\lambda = \mu = \Lambda_1$ case is not sufficient. 

\begin{ex}
If we apply Equation~\eqref{eq:K_jdt} directly for each empty box on
\[
\ytableausetup{boxsize=1.5em}
\ytableaushort{{ }{ }{1,\!2},{1,\!2}2} \longrightarrow \ytableaushort{{ }12,{1,\!2}2} \longrightarrow \ytableaushort{112,22}\ ,
\]
and note that the result does not appear in the expansion of $\G_2 \G_1$ (with $n = 2$). Instead, what we should obtain is
\[
\ytableausetup{boxsize=1.5em}
\ytableaushort{{ }{ }{1,\!2},{1,\!2}2} \longrightarrow \ytableaushort{1{1,\!2}2,2}\ ,
\]
which may instead be obtained using the rectification order $U = \raisebox{-4pt}{$\ytableaushort{\bon {\btw,\!\bth}}$}$\,:
\begin{align*}
\rect_U\left( \ytableaushort{{ }{ }{1,\!2},{1,\!2}2} \right)
& =
\ytableaushort{\bon{\btw,\!\bth}{1,\!2}, {1,\!2}2}
\imapsto{\bth}
\ytableaushort{\bon{\btw,\!1}{\bth,\!2}, {1,\!2}2}
\imapsto{\bth}
\ytableaushort{\bon{\btw,\!1}{2,\!\bth}, {1,\!2}2}
\\ & 
\imapsto{\btw}
\ytableaushort{\bon{1,\!\btw}{2,\!\bth}, {1,\!2}2}
\imapsto{\btw}
\ytableaushort{\bon{1,\!2}{2,\!\bth}, {1,\!2}\btw}
\imapsto{\bon}
\ytableaushort{1{1,\!2}{2,\!\bth}, {\bon,\!2}\btw}
\\ & 
\imapsto{\bon}
\ytableaushort{1{1,\!2}{2,\!\bth}, {2,\!\bon}\btw}
\longmapsto
\ytableaushort{1{1,\!2}2,2}\ ,
\end{align*}
\end{ex}

\begin{prob}
\label{prob:K_jdt_svt}
Determine a K-jdt rule on semistandard set-valued tableaux such that the K-rectification of $T \ast S$ equals the Buch insertion $S \xleftarrow{B} T$.
\end{prob}

\section{K-theory crystal arrows}
\label{sec:K_theory}

The primary goal of this project was to develop K-theory analogs of crystals. In this section, we describe some progress towards this goal and discuss some remaining difficulties that must be overcome. Thus, our main problem is the following.

\begin{prob}
\label{prob:Krystal}
Construct a K-theory analog of crystals for $\svssyt^n(\lambda)$.
\end{prob}

We will solve special cases of Open Problem~\ref{prob:Krystal} for certain classes of $\lambda$.

Based on the description of the Demazure--Lascoux operators and our $U_q(\fsl_n)$-crystal structure, our approach is to construct an additional set of operators, which we call \defn{K-crystal operators}, $e_i^K, f_i^K \colon \svssyt^n(\lambda) \to \svssyt^n(\lambda) \sqcup \{0\}$ with $f_i^K T$ formed by adding an $i+1$ to some cell of $T$ and $e_i^K T' = T$ if and only if $T' = f_i^K T$. Such K-crystal operators should satisfy the following properties.
\begin{enumerate}
\item[(K.1)] \label{Krystal:connected} The set $\svssyt^n(\lambda)$ is connected with the \defn{minimal highest weight element} $u_{\lambda}$, the highest weight element of $B(\lambda)$ (or has minimal excess), being the only highest weight element such that $e_i^K u_{\lambda} = 0$.
\item[(K.2)] \label{Krystal:demazure} The \defn{K-Demazure crystal}
  \[
  \svssyt^n_w(\lambda) := \left\{ b \in \svssyt^n(\lambda) \mid (e_{i_{\ell}}^K)^{\max} e_{i_{\ell}}^{\max} \cdots (e_{i_1}^K)^{\max} e_{i_1}^{\max} b = u_{\lambda} \right\}
  \]
  does not depend on the choice of reduced expression $w = s_{i_1} \cdots s_{i_{\ell}}$.
\item[(K.3)] \label{Krystal:character} The character of $\svssyt^n_w(\lambda)$ is the Lascoux polynomial $L_{w\lambda}(\xx)$.
\end{enumerate}

We can also recursively construct $\svssyt^n_w(\lambda)$ by applying $f_{i_j}$ as many times as possible to every element in $S_{j-1} \cup f_{i_j}^K S_{j-1}$ for all $1 \leq j \leq \ell$ starting with $S_0 = \{T_{\lambda}\}$.
Note that in a K-crystal, we have $\svssyt^n_{w_0}(\lambda) = \svssyt^n(\lambda)$.

\begin{dfn}\label{def:Krystal}
If the above properties (K.1)--(K.3) hold for some K-crystal operators, we say the give $\svssyt^n(\lambda)$ the structure of a \defn{K-crystal}.
\end{dfn}

We anticipate that there is a unique K-crystal structure on $\svssyt^n(\lambda)$ given the $U_q(\fsl_n)$-crystal structure.
Moreover, we expect our K-crystal operators to further satisfy:
\begin{enumerate}[\quad(i)]
\item[(H.1)] for all $T \in \svssyt^n(\lambda)$, we have $e_i^K e_i^K T = 0$ and $f_i^K f_i^K T = 0$;
\item[(H.2)] if $e_i T \neq 0$ or $f_i T = 0$, then $f_i^K T = 0$.
\end{enumerate}
In particular, we expect~(H.1) from the fact that $\varpi_i^2 = \varpi_i$ and the definition of $\varpi_i$. We want $f_i T = 0$ to imply $f_i^K T = 0$ by $\varpi_i f = \pi_i f + \beta \pi_i (x_{i+1} \cdot f)$ and $\pi_i(x_{i+1} \cdot f) = 0$ for any polynomial $f$ that does not contain the variables $x_i$ and $x_{i+1}$; hence, we expect (H.2) to hold. In fact, we will use (H.1) and (H.2) as heuristics for defining the K-crystal operators.

However, we will show in Section~\ref{sec:gen_case} that all of these properties cannot hold in the general case (see Proposition~\ref{prop:Krystal_counterex}), and we have to weaken the definition of a K-crystal.

\subsection{Single row K-crystals}

We begin by constructing a K-crystal structure on $\svssyt^n( k \Lambda_1)$.

\begin{dfn}
\label{def:K_crystal_ops}
The K-crystal operator $f_i^K$ acts on $T \in \svssyt^n(k \Lambda_1)$ as follows: If $i \notin T$ or $i+1 \in T$, then $f_i^K T = 0$; otherwise $f_i^K T$ is given by adding $i+1$ to the rightmost box containing $i$ in $T$. The K-crystal operator $e_i^K$ acts on $T \in \svssyt^n(k\Lambda_1)$ as follows: If there is no box in $T$ containing both $i$ and $i+1$, then $e_i^K T = 0$; otherwise $e_i^K T$ is given by deleting $i+1$ from that (necessarily unique) box.
\end{dfn}

Clearly, $e_i^K T' = T$ if and only if $T' = f_i^K T$. For examples of these operators, see the K-crystal structure on $\svssyt^3(2 \Lambda_1)$ depicted in Figure~\ref{fig:Krystal_ex_onerow}. 

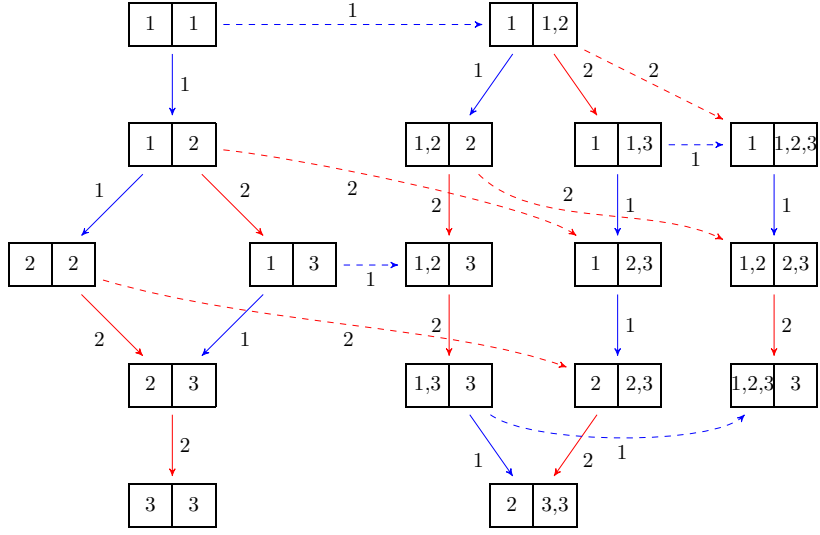
\begin{figure}[h]
\[
\ytableausetup{boxsize=2em}
\scalebox{0.8}{
\begin{tikzpicture}[xscale=2,yscale=2]
\node (hw) at (0,0) {\ytableaushort{11}};
\node (f1) at (0,-1) {\ytableaushort{12}};
\node (f11) at (-1,-2) {\ytableaushort{22}};
\node (f21) at (1,-2) {\ytableaushort{13}};
\node (f211) at (0,-3) {\ytableaushort{23}};
\node (f2211) at (0,-4) {\ytableaushort{33}};
\node (k1) at (3,0) {\ytableaushort{1{1,\!2}}};
\node (f1k1) at (2.3,-1) {\ytableaushort{{1,\!2}2}};
\node (f21k1) at (2.3,-2) {\ytableaushort{{1,\!2}3}};
\node (f221k1) at (2.3,-3) {\ytableaushort{{1,\!3}3}};
\node (f2k1) at (3.7,-1) {\ytableaushort{1{1,\!3}}};
\node (f12k1) at (3.7,-2) {\ytableaushort{1{2,\!3}}};
\node (f112k1) at (3.7,-3) {\ytableaushort{2{2,\!3}}};
\node (f1221k1) at (3,-4) {\ytableaushort{2{3,\!3}}};
\node (k21) at (5,-1) {\ytableaushort{1{1,\!2,\!3}}};
\node (f1k21) at (5,-2) {\ytableaushort{{1,\!2}{2,\!3}}};
\node (f21k21) at (5,-3) {\ytableaushort{{1,\!2,\!3}3}};
\draw[->,blue] (hw) -- node[midway,right,black] {$1$} (f1);
\draw[->,blue] (f1) -- node[midway,above left,black] {$1$} (f11);
\draw[->,red] (f1) -- node[midway,above right,black] {$2$} (f21);
\draw[->,blue] (f21) -- node[midway,below right,black] {$1$} (f211);
\draw[->,red] (f11) -- node[midway,below left,black] {$2$} (f211);
\draw[->,red] (f211) -- node[midway,right,black] {$2$} (f2211);
\draw[->,dashed,blue] (hw) -- node[midway,above,black] {$1$} (k1);
\draw[->,blue] (k1) -- node[midway,above left,black] {$1$} (f1k1);
\draw[->,blue] (f2k1) -- node[midway,right,black] {$1$} (f12k1);
\draw[->,blue] (f12k1) -- node[midway,right,black] {$1$} (f112k1);
\draw[->,red] (k1) -- node[midway,above right,black] {$2$} (f2k1);
\draw[->,red] (f1k1) -- node[midway,left,black] {$2$} (f21k1);
\draw[->,red] (f21k1) -- node[midway,left,black] {$2$} (f221k1);
\draw[->,blue] (f221k1) -- node[midway,below left,black] {$1$} (f1221k1);
\draw[->,red] (f112k1) -- node[midway,below right,black] {$2$} (f1221k1);
\draw[->,dashed,red] (k1) -- node[midway,above,black] {$2$} (k21);
\draw[->,blue] (k21) -- node[midway,right,black] {$1$} (f1k21);
\draw[->,red] (f1k21) -- node[midway,right,black] {$2$} (f21k21);
\draw[->,dashed,red] (f1) .. controls (1,-1.1) and (3,-1.5) .. node[pos=0.4,below left,black] {$2$} (f12k1);
\draw[->,dashed,red] (f1k1) .. controls (3,-1.7) and (4,-1.5) .. node[pos=0.35,above right,black] {$2$} (f1k21);
\draw[->,dashed,red] (f11) .. controls (0.7,-2.5) and (2.3,-2.5) .. node[midway,below,black] {$2$} (f112k1);
\draw[->,dashed,blue] (f21) -- node[midway,below,black] {$1$} (f21k1);
\draw[->,dashed,blue] (f2k1) -- node[midway,below,black] {$1$} (k21);
\draw[->,dashed,blue] (f221k1) .. controls (3,-3.5) and (4.5,-3.5) .. node[midway,below,black] {$1$} (f21k21);
\end{tikzpicture}
}
\]
\caption{The K-crystal $\svssyt^3(2\Lambda_1)$, where we depict K-crystal operators by dashed lines and $U_q(\fsl_n)$-crystal operators by solid lines.}\label{fig:Krystal_ex_onerow}
\end{figure}

Recall for a Coxeter group $W$ with Coxeter generators $\{s_i\}$, the \defn{length} of an element $w \in W$ is the minimal number of simple reflections $s_i$ needed to express $w$. Moreover, for a parabolic subgroup $W_J$, every coset in $W / W_J$ has a unique representative in $W$ that is of minimal length, and for a given $w \in W$, we denote by $\lfloor w \rfloor$ the corresponding minimal length coset representative.
The \defn{support} of an element $w \in W$ is the set $\Supp(w)$ such that the simple reflection $s_i$ for all $i \in \Supp(w)$ appears in some (and hence any) reduced expression for $w$. For more on Coxeter groups, we refer the reader to, \textit{e.g.},~\cite{BB05,Davis08,Humphreys90,Kane01}.

\begin{lemma}
\label{lemma:coset_reprs}
Let $W = \sym_n$. Let $J = I \setminus \{1\}$ and let $W_J = \langle s_i \mid i \in J \rangle$ be the corresponding parabolic subgroup. Each $\lfloor w \rfloor$ is of the form $s_m \cdots s_1$ for some $m$ (we consider $m = 0$ if $w$ is the identity). We have
\begin{equation}\label{eq:Lascouxcrystal}
\svssyt^n_w(k \Lambda_1) = \svssyt^n_{\lfloor w \rfloor}(k \Lambda_1) =  \{ T \in \svssyt^n(k \Lambda_1) \mid \max(T) \leq m + 1 \}.
\end{equation}
Conversely, if $m + 1 = \max \{ \max(T) \mid T \in \svssyt^n_w(k \Lambda_1) \}$, then 
\[
\lfloor w \rfloor = s_{m} \dotsm s_1.
\]
\end{lemma}

\begin{proof}
For $w \in W$, it is easy to see from the Coxeter relations on $S_n$ that $\lfloor w \rfloor$ may be written uniquely as $s_m \cdots s_1$. 

We now prove Equation~\eqref{eq:Lascouxcrystal} by induction on $m$. The base case of $m = 0$ is trivial. Assume that the claim holds for $m$ and consider $w$ with $\lfloor w \rfloor = s_{m+1} \dotsm s_1$.  Let $w = w' s_{m+1} w''$ such that $m+1 \notin \Supp(w'')$.
By the inductive hypothesis, we have then
\[
\svssyt^n_w(k \Lambda_1) = \svssyt^n_{w' s_{m+1} w''}(k \Lambda_1)  =  \svssyt^n_{w' s_{m+1} s_j \cdots s_1}(k \Lambda_1)
\]
 for some $j \leq m$. If $j < m$, then
\[
\svssyt^n_{w' s_{m+1} s_j \cdots s_1}(k \Lambda_1) = \svssyt^n_{w' s_j \cdots s_1 s_{m+1}}(k \Lambda_1) = \svssyt^n_{w' s_j \cdots s_1}(k \Lambda_1),
\]
since $\max\{ \max(T) \mid T \in \svssyt^n_{s_j \cdots s_1}(k\Lambda_1) \} < m+1$, and we can proceed by downward induction on length of $w$.
Thus, we may assume $j=m$, so that we have 
\[
\svssyt^n_w(k \Lambda_1) =  \svssyt^n_{w' s_{m+1} s_m \cdots s_1}(k \Lambda_1).
\]
Note that $\lfloor w \rfloor = \lfloor w' s_{m+1} s_m \cdots s_1 \rfloor$.
From the definition of the crystal operators and K-Demazure crystal, it is sufficient to show Equation~\eqref{eq:Lascouxcrystal} for $s_{m+1} s_m \cdots s_1$.
Note that, for any $T \in \svssyt^n(k\Lambda_1)$, we can repeatedly apply $e_{m+1}$ until we remove all $(m+2)$'s except possibly in the rightmost box, in which case that box also contains an $m+1$, and so the final $m+2$ is removed by the application of $e_{m+1}^K$. Thus, by the inductive hypothesis, we have
\[
\svssyt^n_{s_{m+1} s_m \cdots s_1}(k \Lambda_1) = \{ T \in \svssyt^n(k \Lambda_1) \mid \max(T) \leq m + 2 \}.
\]

For the converse, the definition of the (K-)crystal operators implies $w$ must have support $\{1, \dotsc, m\}$. From the construction of the minimal length coset representatives in $W / W_J$, the claim follows.
\end{proof}

\begin{thm}
\label{thm:Krystal_rows}
The $U_q(\fsl_n)$-crystal $\svssyt^n(k \Lambda_1)$ is a K-crystal.
\end{thm}

\begin{proof}
(K.1) is immediately from Lemma~\ref{lemma:coset_reprs}.
(K.2) follows from Lemma~\ref{lemma:coset_reprs}, as the lemma implies $\svssyt^n_w(k \Lambda_1)$ only depends on the minimal length coset representative, which is independent of reduced expression of $w$.
From the properties of the Demazure--Lascoux operators, we have $\varpi_w \xx^{\lambda} = \varpi_{\lfloor w \rfloor} \xx^{\lambda}$, where $W_J$ is the parabolic subgroup corresponding to the stabilizer of $\lambda$.
Hence, (K.3) follows from Lemma~\ref{lemma:coset_reprs}.
\end{proof}

\begin{remark}
\label{rem:type_2_operators}
From Remark~\ref{rem:crystal_excitations}, we see that the K-crystal operators from Definition~\ref{def:K_crystal_ops} correspond to Type~$2$ elementary excitations/emissions on the corresponding excited Young diagram. Additionally, as for the usual crystal operators, the action of $f_i^K$ only corresponds to a subset of all possible Type~$2$ elementary excitations. Moreover, we see that~\cite[Lemma~4.17]{GK15} does not hold when restricted to the (K-)crystal operators by considering the element
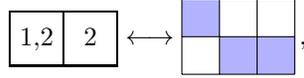
\[
\ytableausetup{boxsize=2em}
\raisebox{-8pt}{$
\ytableaushort{{1,\!2}2}
$}
\longleftrightarrow
\begin{tikzpicture}[xscale=0.5,yscale=-0.5,baseline=-17]
\fill[blue!30] (0,0) rectangle (1,1);
\fill[blue!30] (1,1) rectangle (3,2);
\draw[step=1] (0,0) grid (3,2);
\end{tikzpicture}\,,
\]
which either requires a Type~$1'$ or~$2$ elementary emission to recover the ground state.
\end{remark}

\subsection{Single column K-crystals}

Next, we show that $\svssyt^n(\Lambda_k)$ is a K-crystal, when equipped with the K-crystal operators of Definition~\ref{def:K_crystal_ops}.

It is clear that these K-crystal operators on $\svssyt^n(\Lambda_k)$ always yield either valid semistandard set-valued tableaux or $0$.
In order to show that these operators form a K-crystal structure, we first note that straightforwardly one has $e_i^K T' = T$ if and only if $T' = f_i^K T$. Also note that (as ordinary crystals)
\begin{equation}
\label{eq:column_isomorphism}
\svssyt^n(\Lambda_k) \iso \bigoplus_{m=k}^{n-1} B(\Lambda_m)^{\oplus \binom{m}{k}}
\end{equation}
since the sizes of the sets within a column cannot change under applying the ordinary $U_q(\fsl_n)$-crystal operators. As an example, the K-crystal structure on $\svssyt^4(\Lambda_2)$ is shown in Figure~\ref{fig:Krystal_ex_La2}.

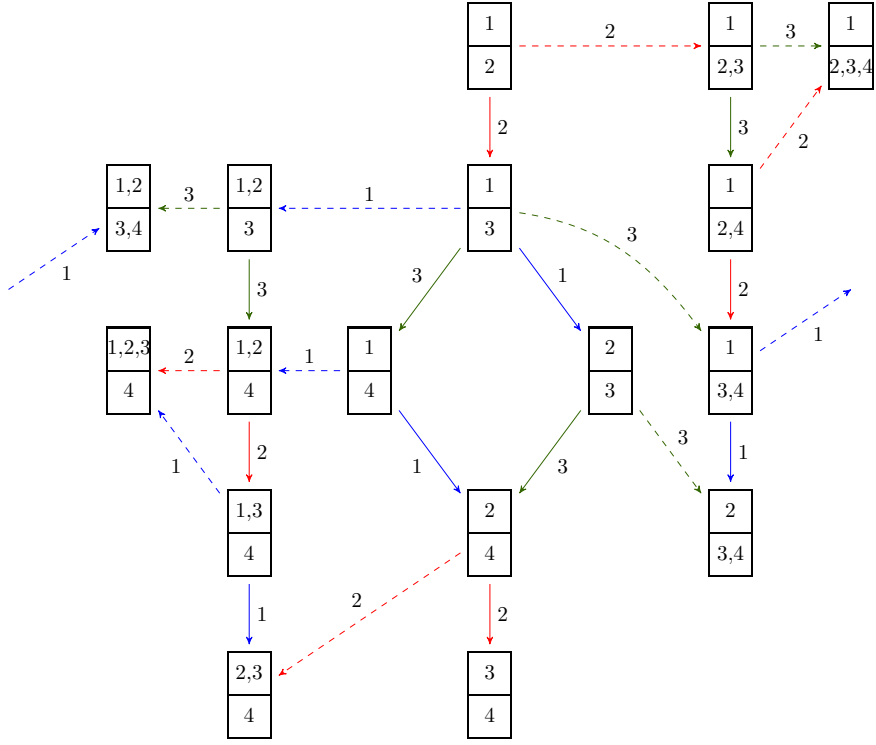
\begin{figure}[h]
\[
\ytableausetup{boxsize=2em}
\scalebox{0.8}{
\begin{tikzpicture}[xscale=2,yscale=2.7]
\node (hw) at (0,0) {\ytableaushort{1,2}};
\node (f2) at (0,-1) {\ytableaushort{1,3}};
\node (f32) at (-1,-2) {\ytableaushort{1,4}};
\node (f12) at (1,-2) {\ytableaushort{2,3}};
\node (f132) at (0,-3) {\ytableaushort{2,4}};
\node (f2132) at (0,-4) {\ytableaushort{3,4}};
\node (k2) at (2,0) {\ytableaushort{1,{2,\!3}}};
\node (f3k2) at (2,-1) {\ytableaushort{1,{2,\!4}}};
\node (f23k2) at (2,-2) {\ytableaushort{1,{3,\!4}}};
\node (f123k2) at (2,-3) {\ytableaushort{2,{3,\!4}}};
\node (k1f2) at (-2,-1) {\ytableaushort{{1,\!2},3}};
\node (f3k1f2) at (-2,-2) {\ytableaushort{{1,\!2},4}};
\node (f23k1f2) at (-2,-3) {\ytableaushort{{1,\!3},4}};
\node (f123k1f2) at (-2,-4) {\ytableaushort{{2,\!3},4}};
\node (k31f2) at (-3,-1) {\ytableaushort{{1,\!2},{3,\!4}}};
\node (k32) at (3,0) {\ytableaushort{1,{2,\!3,\!4}}};
\node (k21f32) at (-3,-2) {\ytableaushort{{1,\!2,\!3},4}};
\draw[->,red] (hw) -- node[midway,right,black] {$2$} (f2);
\draw[->,dgreencolor] (f2) -- node[midway,above left,black] {$3$} (f32);
\draw[->,blue] (f2) -- node[midway,above right,black] {$1$} (f12);
\draw[->,blue] (f32) -- node[midway,below left,black] {$1$} (f132);
\draw[->,dgreencolor] (f12) -- node[midway,below right,black] {$3$} (f132);
\draw[->,red] (f132) -- node[midway,right,black] {$2$} (f2132);
\draw[->,dashed,red] (hw) -- node[midway,above,black] {$2$} (k2);
\draw[->,dgreencolor] (k2) -- node[midway,right,black] {$3$} (f3k2);
\draw[->,red] (f3k2) -- node[midway,right,black] {$2$} (f23k2);
\draw[->,blue] (f23k2) -- node[midway,right,black] {$1$} (f123k2);
\draw[->,dashed,blue] (f2) -- node[midway,above,black] {$1$} (k1f2);
\draw[->,dgreencolor] (k1f2) -- node[midway,right,black] {$3$} (f3k1f2);
\draw[->,red] (f3k1f2) -- node[midway,right,black] {$2$} (f23k1f2);
\draw[->,blue] (f23k1f2) -- node[midway,right,black] {$1$} (f123k1f2);
\draw[->,dashed,dgreencolor] (f2) .. controls (0.9,-1.1) and (1.3,-1.3) .. node[midway,above right,black] {$3$} (f23k2);
\draw[->,dashed,dgreencolor] (k2) -- node[midway,above,black] {$3$} (k32);
\draw[->,dashed,blue] (f32) -- node[midway,above,black] {$1$} (f3k1f2);
\draw[->,dashed,red] (f3k1f2) -- node[midway,above,black] {$2$} (k21f32);
\draw[->,dashed,dgreencolor] (f12) -- node[midway,above right,black] {$3$} (f123k2);
\draw[->,dashed,red] (f132) -- node[midway,above left,black] {$2$} (f123k1f2);
\draw[->,dashed,blue] (f23k1f2) -- node[midway,below left,black] {$1$} (k21f32);
\draw[->,dashed,red] (f3k2) -- node[midway,below right,black] {$2$} (k32);
\draw[->,dashed,dgreencolor] (k1f2) -- node[midway,above,black] {$3$} (k31f2);
\draw[->,dashed,blue] (f23k2) -- node[midway,below right,black] {$1$} (3,-1.5);
\draw[->,dashed,blue] (-4,-1.5) -- node[midway,below right,black] {$1$} (k31f2);
\end{tikzpicture}
}
\]
\caption{The K-crystal for $\svssyt^4(\Lambda_2)$, depicted with the same conventions as in Figure~\ref{fig:Krystal_ex_onerow}.}
\label{fig:Krystal_ex_La2}
\end{figure}

\begin{lemma}
\label{lemma:demazure_redexp}
The K-Demazure crystal $\svssyt_w^n(\Lambda_k)$ does not depend on the choice of reduced expression for $w$.
\end{lemma}

\begin{proof}
By Matsumoto's theorem~\cite{Matsumoto64}, it is sufficient to show the set $\svssyt_w^n(\Lambda_k)$ does not change when applying a braid relation $s_i s_{i+1} s_i = s_{i+1} s_i s_{i+1}$ or $s_i s_j = s_j s_i$ for $\lvert i - j \rvert > 1$.

It is clear that $f_i^K f_j = f_j f_i^K$ and $f_i^K f_j^K = f_j^K f_i^K$ for all $\lvert i - j \rvert > 1$. So the claim follows when the reduced expressions differ by $s_i s_j = s_j s_i$.

Next we consider the braid relation $s_i s_{i+1} s_i = s_{i+1} s_i s_{i+1}$.
Note that, by definition of $f_i^K$, we have $f_i^K T = 0$ if and only if $f_i T = 0$. Furthermore, we have $f_i f_i^K T = f_i^K f_i T = 0$ for all $T \in \svssyt^n(\lambda)$.
Therefore, it is straightforward to see that we have $f_{i+1}^K f_i = f_i f_{i+1} f_i^K$.
A direct computation in all cases shows that application of all other compositions of operators in $\{ f_i, f_{i+1}, f_i^K, f_{i+1}^K \}$ results in $0$. Hence, the lemma follows.
\end{proof}

Next, we need the following K-analog of~\cite[Prop.~3.3.4]{K93}. Define an \defn{$i$-K-string} to be the subcrystal of the form
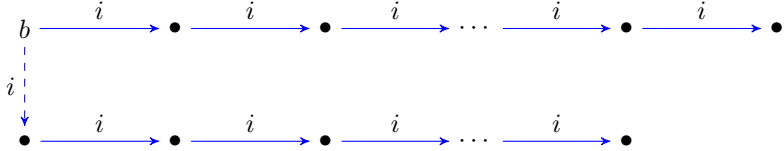
\[
\begin{tikzpicture}[xscale=2,yscale=-1.5]
\node (top) at (0,0)  {$b$};
\node (f) at (1,0)  {$\bullet$};
\node (ff) at (2,0)  {$\bullet$};
\node (dots) at (3,0) {$\cdots$};
\node (fend) at (4,0)  {$\bullet$};
\node (ffend) at (5,0)  {$\bullet$};
\node (K) at (0,1)  {$\bullet$};
\node (fK) at (1,1)  {$\bullet$};
\node (ffK) at (2,1)  {$\bullet$};
\node (dotsK) at (3,1) {$\cdots$};
\node (fendK) at (4,1)  {$\bullet$};
\draw[->,blue] (top) -- node[midway,above,black] {$i$} (f);
\draw[->,blue] (f) -- node[midway,above,black] {$i$} (ff);
\draw[->,blue] (ff) -- node[midway,above,black] {$i$} (dots);
\draw[->,blue] (dots) -- node[midway,above,black] {$i$} (fend);
\draw[->,blue] (fend) -- node[midway,above,black] {$i$} (ffend);
\draw[->,blue] (K) -- node[midway,above,black] {$i$} (fK);
\draw[->,blue] (fK) -- node[midway,above,black] {$i$} (ffK);
\draw[->,blue] (ffK) -- node[midway,above,black] {$i$} (dotsK);
\draw[->,blue] (dotsK) -- node[midway,above,black] {$i$} (fendK);
\draw[->,blue,dashed] (top) -- node[midway,left,black] {$i$} (K);
\end{tikzpicture}
\]
where the $i$-strings are as long as possible and the dashed arrow represents the $f_i^K$ action. We say the $i$-K-string has \defn{length} $\ell$ if $\varphi_i(b) = \ell$. (This is equivalent to saying that the $i$-string starting at $b$ has length $\ell$.) Note that $\varphi_i(f_i^K b) = \ell - 1$. This is our proposed K-theory analog of an $i$-string or the restriction to $\fsl_2$ (equivalently, $\svssyt^2(\ell \Lambda_1)$) corresponding to $\{i\}$ based on the description of the Demazure--Lascoux operator as $\varpi_i f = \pi_i f + \beta \pi_i (x_{i+1} \cdot f)$.

It is clear that $\svssyt^n(\lambda)$ (for $\lambda = \Lambda_k, k\Lambda_1$) decomposes into a direct sum of $i$-K-strings. Note that this is branching down to type $\fsl_2$, where this clearly yields a K-crystal structure for any $i$.

\begin{prop}
\label{prop:demazure_strings}
Let $S$ be an $i$-K-string of $\svssyt^n(\Lambda_k)$ and let $b$ be the highest weight element of $S$. Then, the set $\svssyt^n_w(\Lambda_k) \cap S$ is either empty, $S$, or $\{b\}$.
\end{prop}

\begin{proof}
We proceed by induction on the Coxeter length of $w$. The base case with $w$ being the identity is trivial. Hence, we inductively assume that the claim holds for all $w$ up through length $\ell$, and let $j$ be such that $\ell(s_j w) > \ell(w)$. We must show that the claim holds for $s_j w$. Note that all $i$-strings (and hence $i$-K-strings) have length at most $1$.

From the definition of the (K-)crystal operators and the K-Demazure crystal, if $i = j$, then the claim clearly holds for $s_j w$.
Thus, we assume $i \neq j$, and it is sufficient to show the following.
Fix some $i$-K-string $S$ with highest weight element $b$, and let $T := f_i b$ and $\widetilde{T} := f_i^K b$ be such that $T, \widetilde{T} \notin \svssyt^n_w(\Lambda_k)$.
If one of the following equations hold
\[
T = f_j T',
\qquad
T = f_j^K T',
\qquad
\widetilde{T} = f_j T',
\qquad
\widetilde{T} = f_j^K T',
\]
for some $T' \in \svssyt^n_w(\Lambda_k)$, then $b, T, \widetilde{T} \in \svssyt^n_{s_j w}(\Lambda_k)$.

Suppose $\lvert i - j \rvert > 1$. Consider the case where $T = f_j T'$. Since $f_i f_j = f_j f_i$, we have
\[
b' := e_i T' \neq 0,
\qquad
\widetilde{T}' := f_i^K b' \neq 0,
\qquad
f_j b' = b,
\qquad
f_j \widetilde{T}' = \widetilde{T}.
\]
The local K-crystal structure is illustrated in Figure~\ref{fig:local_structure_commute}.
By the induction assumption, we have $b', \widetilde{T}' \in \svssyt^n_w(\Lambda_k)$. Therefore, we have $b, \widetilde{T} \in \svssyt^n_{s_j w}(\Lambda_k)$ as claimed. The other three cases are similar.

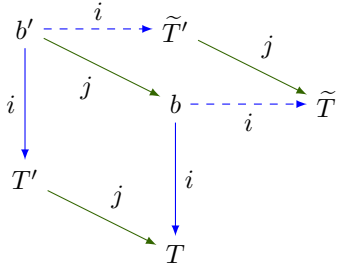
\begin{figure}
\[
\begin{tikzpicture}[>=latex,scale=2]
\node (b) at (0,0) {$b$};
\node (T) at (0,-1) {$T$};
\node (tT) at (1,0) {$\widetilde{T}$};
\node (bp) at (-1,0.5) {$b'$};
\node (Tp) at (-1,-0.5) {$T'$};
\node (tTp) at (0,0.5) {$\widetilde{T}'$};
\draw[->,blue] (b) -- node[midway,right,black] {$i$} (T);
\draw[->,dashed,blue] (b) -- node[midway,below,black] {$i$} (tT);
\draw[->,blue] (bp) -- node[midway,left,black] {$i$} (Tp);
\draw[->,dashed,blue] (bp) -- node[midway,above,black] {$i$} (tTp);
\draw[->,dgreencolor] (bp) -- node[midway,below left,black] {$j$} (b);
\draw[->,dgreencolor] (Tp) -- node[midway,above right,black] {$j$} (T);
\draw[->,dgreencolor] (tTp) -- node[midway,above right,black] {$j$} (tT);
\end{tikzpicture}
\]
\caption{The local K-crystal structure near $T$ in the the proof of Proposition~\ref{prop:demazure_strings} in one of the cases when $\lvert i - j \rvert > 1$.}
\label{fig:local_structure_commute}
\end{figure}

Now, suppose $|i - j| = 1$.
Note that there does not exist a $T \in \svssyt^n(\Lambda_k)$ such that $\varepsilon_i(T) = \varepsilon_{i \pm 1}(T) = 1$. Hence, the case $T = f_j T'$ is impossible when $\lvert i - j \rvert = 1$.

Hence, consider when $\widetilde{T} = f_j^K T'$, then we have $f_j T' = b$.
If $j = i - 1$, then $\widetilde{T}$ contains a box with $i-1, i, i + 1$. Thus we have $b = e_i T' \in \svssyt^n_{s_j w}(\Lambda_k)$ by the induction assumption; so the claim holds in this case.
If $j = i + 1$, then $\widetilde{T}$ has a box with $i, i+1, i+2$. A straightforward computation yields the following local picture around $\widetilde{T}$:
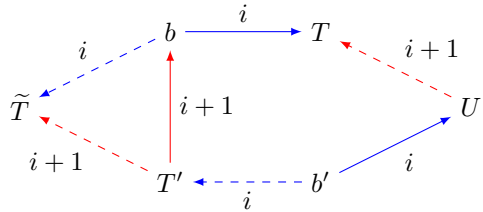
\[
\begin{tikzpicture}[>=latex,scale=2]
\node (T) at (1,0.5) {$T$};
\node (b) at (0,0.5) {$b$};
\node (tT) at (-1,0) {$\widetilde{T}$};
\node (tTp) at (0,-0.5) {$T'$};
\node (bp) at (1,-0.5) {$b'$};
\node (Tp) at (2,0) {$U$};
\draw[->,blue] (b) -- node[midway,above,black] {$i$} (T);
\draw[->,dashed,blue] (b) -- node[midway,above left,black] {$i$} (tT);
\draw[->,dashed,red] (tTp) -- node[midway,below left,black] {$i+1$} (tT);
\draw[->,blue] (bp) -- node[midway,below right,black] {$i$} (Tp);
\draw[->,dashed,red] (Tp) -- node[midway,above right,black] {$i+1$} (T);
\draw[->,red] (tTp) -- node[midway,right,black] {$i+1$} (b);
\draw[->,dashed,blue] (bp) -- node[midway,below,black] {$i$} (tTp);
\end{tikzpicture}
\]
where the dashed lines denote K-crystal operators. Note that $b',U \in \svssyt^n_w(\Lambda_k)$ by the induction assumption, and so $b, T \in \svssyt^n_{s_j w}(\Lambda_k)$

The proof that $T = f_j^K T'$ implies $b, \widetilde{T} \in \svssyt^n_{s_j w}(\Lambda_k)$ is similar to the previous case. 

Last, the case of $\widetilde{T} = f_j T'$ is impossible for $j = i+1$ as then $i, i+1 \in \widetilde{T}$, which implies $e_j \widetilde{T} = 0$. For $j = i - 1$, observe $\widetilde{T}$ contains a box with $i, i+1$ but $i-1 \notin \widetilde{T}$. Therefore the local crystal structure around $\widetilde{T}$ is
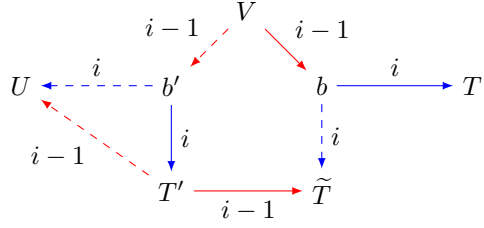
\[
\begin{tikzpicture}[>=latex,scale=2]
\node (T) at (1,0.5) {$T$};
\node (b) at (0,0.5) {$b$};
\node (tT) at (0,-0.2) {$\widetilde{T}$};
\node (Tp) at (-1,-0.2) {$T'$};
\node (bp) at (-1,0.5) {$b'$};
\node (U) at (-2,0.5) {$U$};
\node (V) at (-0.5,1) {$V$};
\draw[->,blue] (b) -- node[midway,above,black] {$i$} (T);
\draw[->,dashed,blue] (b) -- node[midway,right,black] {$i$} (tT);
\draw[->,red] (Tp) -- node[midway,below,black] {$i-1$} (tT);
\draw[->,blue] (bp) -- node[midway,right,black] {$i$} (Tp);
\draw[->,dashed,red] (V) -- node[midway,above left,black] {$i-1$} (bp);
\draw[->,red] (V) -- node[midway,above right,black] {$i-1$} (b);
\draw[->,dashed,blue] (bp) -- node[midway,above,black] {$i$} (U);
\draw[->,dashed,red] (Tp) -- node[midway,below left,black] {$i-1$} (U);
\end{tikzpicture}
\]
Note that $b', U, V \in \svssyt^n_w(\Lambda_k)$ by the induction assumption. However, $V \in \svssyt^n_w(\Lambda_k)$ implies that $\widetilde{T} \in \svssyt^n_w(\Lambda_k)$, which is impossible.
\end{proof}

\begin{thm}
\label{thm:Krystal_columns}
The $U_q(\fsl_n)$-crystal $\svssyt^n(\Lambda_k)$ is a K-crystal.
\end{thm}

\begin{proof}
Checking~(K.1) is a straightforward computation.
Lemma~\ref{lemma:demazure_redexp} establishes Property~(K.2).
Property~(K.3) follows from Proposition~\ref{prop:demazure_strings} and from the definition of the Demazure--Lascoux operators. 
\end{proof}

\subsection{Towards the general case}
\label{sec:gen_case}

When trying to construct a K-crystal structure on $\svssyt^n(\lambda)$ for general $\lambda$, we note the K-crystal operators given in Definition~\ref{def:K_crystal_ops}, where (inspired by (H.2)) we instead consider $f_i^K(T) = 0$ whenever $e_i(T) \neq 0$ or $f_i(T) = 0$, do not seem to extend to give a K-crystal structure on $\svssyt^n(\lambda)$.

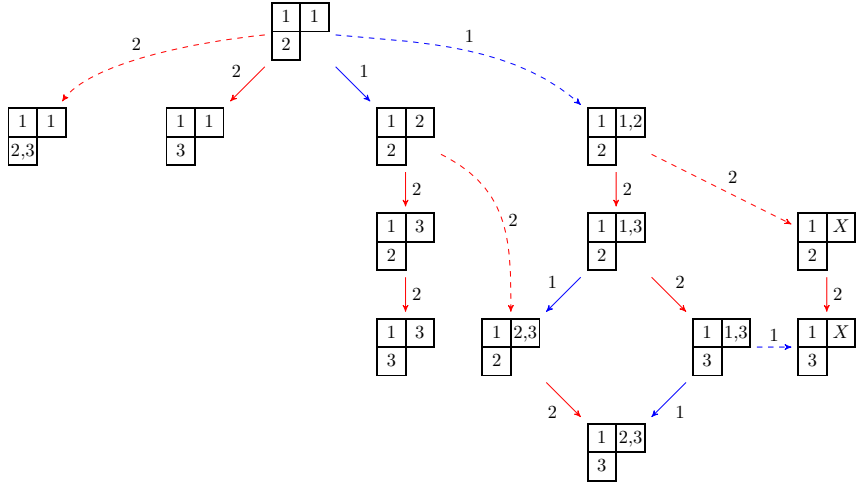
\begin{figure}
\[
\ytableausetup{boxsize=1.5em}
\scalebox{0.7}{
\begin{tikzpicture}[scale=2]
\node (hw) at (0,0) {\ytableaushort{11,2}};
\node (f1) at (1,-1) {\ytableaushort{12,2}};
\node (f21) at (1,-2) {\ytableaushort{13,2}};
\node (f221) at (1,-3) {\ytableaushort{13,3}};
\node (f2) at (-1,-1) {\ytableaushort{11,3}};
\node (k1) at (3,-1) {\ytableaushort{1{1,\!2},2}};
\node (f2k1) at (3,-2) {\ytableaushort{1{1,\!3},2}};
\node (f22k1) at (4,-3) {\ytableaushort{1{1,\!3},3}};
\node (f12k1) at (2,-3) {\ytableaushort{1{2,\!3},2}};
\node (f122k1) at (3,-4) {\ytableaushort{1{2,\!3},3}};
\node (k2) at (-2.5,-1) {\ytableaushort{11,{2,\!3}}};
\node (k21) at (5,-2) {\ytableaushort{1X,2}};
\node (f2k21) at (5,-3) {\ytableaushort{1X,3}};
\draw[->,blue] (hw) -- node[midway,above right,black] {$1$} (f1);
\draw[->,blue,dashed] (hw) .. controls (1,-0.1) and (2,-0.1) .. node[midway,above right,black] {$1$} (k1);
\draw[->,red] (hw) -- node[midway,above left,black] {$2$} (f2);
\draw[->,red] (f1) -- node[midway,right,black] {$2$} (f21);
\draw[->,red] (f21) -- node[midway,right,black] {$2$} (f221);
\draw[->,red,dashed] (hw) .. controls (-1,-0.1) and (-2,-0.3) .. node[midway,above left,black] {$2$} (k2);
\draw[->,red,dashed] (f1) .. controls (2,-1.5) and (2,-2) .. node[midway,right,black] {$2$} (f12k1);
\draw[->,red] (k1) -- node[midway,right,black] {$2$} (f2k1);
\draw[->,red] (f2k1) -- node[midway,above right,black] {$2$} (f22k1);
\draw[->,blue] (f2k1) -- node[midway,above left,black] {$1$} (f12k1);
\draw[->,blue] (f22k1) -- node[midway,below right,black] {$1$} (f122k1);
\draw[->,red] (f12k1) -- node[midway,below left,black] {$2$} (f122k1);
\draw[->,red,dashed] (k1) -- node[midway,above right,black] {$2$} (k21);
\draw[->,red] (k21) -- node[midway,right,black] {$2$} (f2k21);
\draw[->,blue,dashed] (f22k1) -- node[midway,above,black] {$1$} (f2k21);
\end{tikzpicture}
}
\]
\caption{The K-Demazure crystal $\svssyt^3_{s_2s_1}(\Lambda_2 + \Lambda_1)$ using the na\"ive generalization of the K-crystal operators given by Definition~\ref{def:K_crystal_ops}, where $X = \{1,2,3\}$.}
\label{fig:bad_demazure}
\end{figure}

\begin{ex}
We have
\begin{align*}
\varpi_{s_1}(x_1^2 x_2) & = \beta x_1^2 x_2^2 + x_1^2 x_2 + x_1 x_2^2,
\\ \varpi_{s_2 s_1}(x_1^2 x_2) & = 
\beta^{2} ( x_{1}^{2} x_{2}^{2} x_{3} + x_{1}^{2} x_{2} x_{3}^{2})
\\ & \hspace{10pt} + \beta (x_{1}^{2} x_{2}^{2} + 2 x_{1}^{2} x_{2} x_{3} + x_{1} x_{2}^{2} x_{3} + x_{1}^{2} x_{3}^{2} + x_{1} x_{2} x_{3}^{2})
\\ & \hspace{10pt} + x_{1}^{2} x_{2} + x_{1} x_{2}^{2} + x_{1}^{2} x_{3} + x_{1} x_{2} x_{3} + x_{1} x_{3}^{2},
\end{align*}
If we attempt to naturally extend the K-crystal operators, for $\svssyt^3_{s_2 s_1}(\Lambda_2 + \Lambda_1)$ we obtain the K-Demazure crystal given by Figure~\ref{fig:bad_demazure}.
Note that for the $1$-K-string $S$ starting at the tableau $T = \scalebox{0.7}{$\ytableaushort{1{1,\!3},3}$}$, we have
\[
S \cap \svssyt^3_{s_2 s_1}(\Lambda_2 + \Lambda_1) = \{ T, f_1 T, f_1^K T \} \neq \emptyset, S, \{ T \}.
\]
Hence, the na\"ive generalization of Proposition~\ref{prop:demazure_strings} is not true in this case. Moreover, if we attempt to extend this structure to $\svssyt^3_{s_1 s_2 s_1}(\Lambda_2 + \Lambda_1)$, then we fail to obtain one of the following two tableaux:
\[
\ytableaushort{{1,\!2}3,3}\,,
\qquad\qquad
\ytableaushort{13,{2,\!3}}\,.
\]
\end{ex}

In fact, there does not exist any K-crystal structure for general shapes that satisfies our heuristic~(H.2).

\begin{prop}
\label{prop:Krystal_counterex}
There does not exist a K-crystal structure for $\svssyt^3(\Lambda_2+\Lambda_1)$ that also satisfies~(H.2).
\end{prop}

\begin{proof}
For this proof, for a tableau $T$, we equate the weight $\wt(T)$ with the corresponding monomial $\xx^{\wt(T)}$.
Let $\lambda = \Lambda_2 + \Lambda_1$, and let $u_{\lambda}$ be the unique (highest) weight element of weight $x_1^2 x_2$.
We have
\begin{align*}
\G_{21}(x_1, x_2, x_3) & = 
x_1^2 x_2^2 x_3^2 \beta^3 + (2 x_1^2 x_2^2 x_3 + 2 x_1^2 x_2 x_3^2 + 2 x_1 x_2^2 x_3^2) \beta^2
\\ & \hspace{20pt} + (x_1^2 x_2^2 + 3 x_1^2 x_2 x_3 + 3 x_1 x_2^2 x_3 + x_1^2 x_3^2 + 3 x_1 x_2 x_3^2 + x_2^2 x_3^2) \beta
\\ & \hspace{20pt} + x_1^2 x_2 + x_1 x_2^2 + x_1^2 x_3 + 2 x_1 x_2 x_3 + x_2^2 x_3 + x_1 x_3^2 + x_2 x_3^2.
\end{align*}
Consider the three tableaux of weight $\beta x_1 x_2 x_3^2$:
\[
T_1 := \ytableaushort{{1,\!2}3,3}\,,
\qquad\qquad
T_2 := \ytableaushort{13,{2,\!3}}\,,
\qquad\qquad
T_3 := \ytableaushort{1{2,\!3},3}\,,
\]
Note that $e_1 e_2 T_a \neq 0$ for all $a = 1,2,3$, and so we have $e_2^K T_a = 0$. There exists a unique $a$ such that $e_1^K T_a \neq 0$ since there is a unique element of weight $x_1 x_3^2$, and we have $e_1 e_2^2 e_1^K T_a = u_{\lambda}$. Similarly, there exists a unique $a'$ such that $e_2^K e_ 2 T_{a'} \neq 0$ since there is a unique element of weight $x_1 x_2^2$, and we have $e_1 e_2^K e_2 T_{a'} = u_{\lambda}$. Therefore, $T_{a'}$ satisfies~(K.2).

Now by~(K.2) for
\[
e_1 e_2 T_1 = \ytableaushort{1{1,\!2},3}\,,
\qquad\qquad
e_1 e_2 T_2 = \ytableaushort{11,{2,\!3}}\,
\]
and weight considerations, we must have either
\[
e_2 e_1^K e_1 e_2 T_1 = u_{\lambda} = e_2^K e_1 e_2 T_2
\quad \text{or} \quad
e_2^K e_1 e_2 T_1 = u_{\lambda} = e_2 e_1^K e_1 e_2 T_2.
\]
In either case, both $T_1$ and $T_2$ satisfy~(K.2) for $w = s_2 s_1 s_2$.
Without loss of generality, suppose $e_2 e_1^K e_1 e_2 T_1 = u_{\lambda}$. In order to satisfy~(K.2) for $w = s_1 s_2 s_1$, we must have either $e_1^K T_1 = f_1 f_2 f_2 u_{\lambda}$ or $e_2^K e_2 T_1 = f_1 u_{\lambda}$.
Suppose $e_1^K T_1 = f_1 f_2 f_2 u_{\lambda}$, then for $T_2$ to satisfy~(K.2), we require $e_2^K e_2 T_2 = f_1 u_{\lambda}$. Similarly, if $e_1^K T_1 = f_1 f_2 f_2 u_{\lambda}$, then we must have $e_2^K e_2 T_2 = f_1 u_{\lambda}$.
In any of these cases, note that $e_2^K T_3 = 0$, $e_2^K e_2 T_3 = 0$, $e_2^K e_1 e_2 T_3 = 0$ by weight considerations and the only paths to $u_{\lambda}$ are
\[
e_1^K e_2 e_1 e_2 T_3 = e_1^K e_2 e_2 e_1 T_3 = u_{\lambda}.
\]
Hence, $T_3$ cannot satisfy~(K.2) for $w_0 = s_2 s_1 s_2$.
\end{proof}

Note that if we remove~(H.2) as being a requirement, then if we also have
\[
f_2^k \ytableaushort{13,2} = \ytableaushort{13,{2,\!3}},
\]
to Figure~\ref{fig:K_crystal_21_212}, then we would obtain K-crystal structure.
However, this would be highly unnatural given the Demazure--Lascoux operators.
Despite this, perhaps there is a natural notion of a K-crystal on rectangles given Theorem~\ref{thm:Krystal_rows} and Theorem~\ref{thm:Krystal_columns}.

\begin{conj}
\label{conj:Krystal_rectangles}
There exists a K-crystal structure also satisfying~(H.2) for $\svssyt^n(k\Lambda_i)$ (\textit{i.e.}, for rectangle shapes).
\end{conj}

We believe a more natural condition would be to enforce Proposition~\ref{prop:demazure_strings}.
In this case, we require a weakening of the K-crystal structure, where the K-crystal operators depend on the choice of reduced expression for $w_0$ and every subword. For $\lambda = \Lambda_2 + \Lambda_1$ and $n = 3$, we give an example of this weak K-crystal structure for $w_0 = s_1 s_2 s_1$ by Figure~\ref{fig:K_crystal_21_212} and for $w_0 = s_1 s_2 s_1$ by Figure~\ref{fig:K_crystal_21_121}.

Let of focus on the weak K-crystal given by Figure~\ref{fig:K_crystal_21_121}. There are a few K-crystal operators that require more care than in the single row and single column cases; in particular
\begin{subequations}
\begin{align}
\ytableausetup{boxsize=1.5em}
\label{eq:bad_tableaux1}
f_1^K \ytableaushort{1{1,\!3},3} & = \ytableaushort{1{1,\!3},{2,\!3}}\,,
&
f_1^K \ytableaushort{1X,3} & = \ytableaushort{1X,{2,\!3}}\,,
\\ \label{eq:bad_tableaux2}
f_2^K \ytableaushort{1{1,\!2},2} & = \ytableaushort{1X,2}\,,
&
f_2^K \ytableaushort{12,2} & = \ytableaushort{12,{2,\!3}}\,.
\end{align}
\end{subequations}
In particular, note that in~\eqref{eq:bad_tableaux2}, the extra $1$ affects the result of $f_2^K$.

\begin{sidewaysfigure}[ht]
\[
\ytableausetup{boxsize=1.5em}
\scalebox{0.8}{
\begin{tikzpicture}[xscale=2.3,yscale=2,baseline=470]
\node (hw) at (0,0) {\ytableaushort{11,2}};
\node (f1) at (-1,-1) {\ytableaushort{12,2}};
\node (f21) at (-1,-2) {\ytableaushort{13,2}};
\node (f221) at (-1,-3) {\ytableaushort{13,3}};
\node (f2) at (1,-1) {\ytableaushort{11,3}};
\node (f12) at (1,-2) {\ytableaushort{12,3}};
\node (f112) at (1,-3) {\ytableaushort{22,3}};
\node (lw) at (0,-4) {\ytableaushort{12,3}};
\node (k1) at (5,-1) {\ytableaushort{1{1,\!2},2}};
\node (f2k1) at (5,-2) {\ytableaushort{1{1,\!3},2}};
\node (f22k1) at (6,-3) {\ytableaushort{1{1,\!3},3}};
\node (f12k1) at (4,-3) {\ytableaushort{1{2,\!3},2}};
\node (f122k1) at (5,-4) {\ytableaushort{1{2,\!3},3}};
\node (f1122k1) at (5,-5) {\ytableaushort{2{2,\!3},3}};
\node (k2) at (-2.5,-1) {\ytableaushort{11,{2,\!3}}};
\node (f1k2) at (-2.5,-2) {\ytableaushort{12,{2,\!3}}};
\node (f21k2) at (-2.5,-3) {\ytableaushort{13,{2,\!3}}};
\node (k1f2) at (2,-1) {\ytableaushort{1{1,\!2},3}};
\node (f1k1f2) at (2,-2) {\ytableaushort{{1,\!2}2,3}};
\node (f21k1f2) at (2,-3) {\ytableaushort{{1,\!2}3,3}};
\node (k21) at (3,-2) {\ytableaushort{1X,2}};
\node (f2k21) at (3,-3) {\ytableaushort{1X,3}};
\node (f12k21) at (3,-4) {\ytableaushort{{1,\!2}{2,\!3},3}};
\node (k12) at (-4,-3) {\ytableaushort{1{1,\!2},{2,\!3}}};
\node (f2k12) at (-4,-4) {\ytableaushort{1{1,\!3},{2,\!3}}};
\node (f12k12) at (-4,-5) {\ytableaushort{1{2,\!3},{2,\!3}}};
\node (k212) at (0,-5) {\ytableaushort{1X,{2,\!3}}};
\draw[->,blue] (hw) -- node[midway,above left,black] {$1$} (f1);
\draw[->,red] (hw) -- node[midway,above right,black] {$2$} (f2);
\draw[->,red] (f1) -- node[midway,left,black] {$2$} (f21);
\draw[->,red] (f21) -- node[midway,left,black] {$2$} (f221);
\draw[->,blue] (f2) -- node[midway,right,black] {$1$} (f12);
\draw[->,blue] (f12) -- node[midway,right,black] {$1$} (f112);
\draw[->,red] (f112) -- node[midway,below right,black] {$2$} (lw);
\draw[->,blue] (f221) -- node[midway,below left,black] {$1$} (lw);
\draw[->,blue,dashed] (k2) -- node[midway,above left,black] {$1$} (k12);
\draw[->,red] (k12) -- node[midway,right,black] {$2$} (f2k12);
\draw[->,blue] (f2k12) -- node[midway,right,black] {$1$} (f12k12);
\draw[->,blue,dashed] (f22k1) .. controls (6.5,-6) and (2,-7) .. node[midway,below,black] {$1$} (f2k12);
\draw[->,red,dashed] (f1k2) -- node[midway,above left,black] {$2$} (f12k12);
\draw[->,blue,dashed] (f2) -- node[midway,above,black] {$1$} (k1f2);
\draw[->,blue] (k1f2) -- node[midway,right,black] {$1$} (f1k1f2);
\draw[->,red] (f1k1f2) -- node[midway,right,black] {$2$} (f21k1f2);
\draw[->,blue,dashed] (f221) .. controls (0.7,-3.6) and (1.5,-3.5) .. node[pos=0.2,above,black] {$1$} (f21k1f2);
\draw[->,red,dashed] (hw) .. controls (-1,-0.1) and (-2,-0.3) .. node[midway,above left,black] {$2$} (k2);
\draw[->,blue] (k2) -- node[midway,right,black] {$1$} (f1k2);
\draw[->,red] (f1k2) -- node[midway,right,black] {$2$} (f21k2);
\draw[->,red,dashed] (f1) .. controls (-0.3,-3) and (3.2,-1.9) .. node[midway,above,black] {$2$} (f12k1);
\draw[->,blue,dashed] (hw) .. controls (1,-0.1) and (2,-0.1) .. node[midway,above right,black] {$1$} (k1);
\draw[->,red] (k1) -- node[midway,right,black] {$2$} (f2k1);
\draw[->,red] (f2k1) -- node[midway,above right,black] {$2$} (f22k1);
\draw[->,blue] (f2k1) -- node[midway,above left,black] {$1$} (f12k1);
\draw[->,blue] (f22k1) -- node[midway,below right,black] {$1$} (f122k1);
\draw[->,red] (f12k1) -- node[midway,below left,black] {$2$} (f122k1);
\draw[->,blue] (f122k1) -- node[midway,right,black] {$1$} (f1122k1);
\draw[->,red,dashed] (f112) .. controls (2,-5) and (4,-5) .. node[midway,below,black] {$2$} (f1122k1);
\draw[->,red,dashed] (k1) -- node[midway,above,black] {$2$} (k21);
\draw[->,red] (k21) -- node[midway,right,black] {$2$} (f2k21);
\draw[->,blue] (f2k21) -- node[midway,right,black] {$1$} (f12k21);
\draw[->,blue,dashed] (f2k1) -- node[midway,above,black] {$1$} (k21);
\draw[->,red,dashed] (f1k1f2) .. controls (2.7,-2) and (2.25,-4) .. node[midway,above right,black] {$2$} (f12k21);
\draw[->,red,dashed] (k12) -- node[midway,below left,black] {$2$} (k212);
\draw[->,blue,dashed] (f2k21) -- node[pos=0.7,below right,black] {$1$} (k212);
\end{tikzpicture}
}
\]
\vspace{-40pt}
\caption{A weak K-crystal structure on $\svssyt^3(\Lambda_2  + \Lambda_1)$, where $X = \{1,2,3\}$, for $w_0 = s_2 s_1 s_2$.}
\label{fig:K_crystal_21_212}
\end{sidewaysfigure}

\begin{sidewaysfigure}[ht]
\[
\ytableausetup{boxsize=1.5em}
\scalebox{0.8}{
\begin{tikzpicture}[xscale=2.3,yscale=2,baseline=470]
\node (hw) at (0,0) {\ytableaushort{11,2}};
\node (f1) at (-1,-1) {\ytableaushort{12,2}};
\node (f21) at (-1,-2) {\ytableaushort{13,2}};
\node (f221) at (-1,-3) {\ytableaushort{13,3}};
\node (f2) at (1,-1) {\ytableaushort{11,3}};
\node (f12) at (1,-2) {\ytableaushort{12,3}};
\node (f112) at (1,-3) {\ytableaushort{22,3}};
\node (lw) at (0,-4) {\ytableaushort{12,3}};
\node (k1) at (5,-1) {\ytableaushort{1{1,\!2},2}};
\node (f2k1) at (5,-2) {\ytableaushort{1{1,\!3},2}};
\node (f22k1) at (6,-3) {\ytableaushort{1{1,\!3},3}};
\node (f12k1) at (4,-3) {\ytableaushort{1{2,\!3},2}};
\node (f122k1) at (5,-4) {\ytableaushort{1{2,\!3},3}};
\node (f1122k1) at (5,-5) {\ytableaushort{2{2,\!3},3}};
\node (k2) at (-2.5,-1) {\ytableaushort{11,{2,\!3}}};
\node (f1k2) at (-2.5,-2) {\ytableaushort{12,{2,\!3}}};
\node (f21k2) at (-2.5,-3) {\ytableaushort{13,{2,\!3}}};
\node (k1f2) at (2,-1) {\ytableaushort{1{1,\!2},3}};
\node (f1k1f2) at (2,-2) {\ytableaushort{{1,\!2}2,3}};
\node (f21k1f2) at (2,-3) {\ytableaushort{{1,\!2}3,3}};
\node (k21) at (3,-2) {\ytableaushort{1X,2}};
\node (f2k21) at (3,-3) {\ytableaushort{1X,3}};
\node (f12k21) at (3,-4) {\ytableaushort{{1,\!2}{2,\!3},3}};
\node (k12) at (-4,-3) {\ytableaushort{1{1,\!2},{2,\!3}}};
\node (f2k12) at (-4,-4) {\ytableaushort{1{1,\!3},{2,\!3}}};
\node (f12k12) at (-4,-5) {\ytableaushort{1{2,\!3},{2,\!3}}};
\node (k212) at (0,-5) {\ytableaushort{1X,{2,\!3}}};
\draw[->,blue] (hw) -- node[midway,above left,black] {$1$} (f1);
\draw[->,red] (hw) -- node[midway,above right,black] {$2$} (f2);
\draw[->,red] (f1) -- node[midway,left,black] {$2$} (f21);
\draw[->,red] (f21) -- node[midway,left,black] {$2$} (f221);
\draw[->,blue] (f2) -- node[midway,right,black] {$1$} (f12);
\draw[->,blue] (f12) -- node[midway,right,black] {$1$} (f112);
\draw[->,red] (f112) -- node[midway,below right,black] {$2$} (lw);
\draw[->,blue] (f221) -- node[midway,below left,black] {$1$} (lw);
\draw[->,blue,dashed] (k2) -- node[midway,above left,black] {$1$} (k12);
\draw[->,red] (k12) -- node[midway,right,black] {$2$} (f2k12);
\draw[->,blue] (f2k12) -- node[midway,right,black] {$1$} (f12k12);
\draw[->,blue,dashed] (f22k1) .. controls (6.5,-6) and (2,-7) .. node[midway,below,black] {$1$} (f2k12);
\draw[->,red,dashed] (f12k1) .. controls (4,-6) and (1,-6) .. node[midway,below,black] {$2$} (f12k12);
\draw[->,blue,dashed] (f2) -- node[midway,above,black] {$1$} (k1f2);
\draw[->,blue] (k1f2) -- node[midway,right,black] {$1$} (f1k1f2);
\draw[->,red] (f1k1f2) -- node[midway,right,black] {$2$} (f21k1f2);
\draw[->,blue,dashed] (f221) .. controls (0.7,-3.6) and (1.5,-3.5) .. node[pos=0.2,above,black] {$1$} (f21k1f2);
\draw[->,red,dashed] (hw) .. controls (-1,-0.1) and (-2,-0.3) .. node[midway,above left,black] {$2$} (k2);
\draw[->,blue] (k2) -- node[midway,right,black] {$1$} (f1k2);
\draw[->,red] (f1k2) -- node[midway,right,black] {$2$} (f21k2);
\draw[->,red,dashed] (f1) -- node[midway,above,black] {$2$} (f1k2);
\draw[->,blue,dashed] (hw) .. controls (1,-0.1) and (2,-0.1) .. node[midway,above right,black] {$1$} (k1);
\draw[->,red] (k1) -- node[midway,right,black] {$2$} (f2k1);
\draw[->,red] (f2k1) -- node[midway,above right,black] {$2$} (f22k1);
\draw[->,blue] (f2k1) -- node[midway,above left,black] {$1$} (f12k1);
\draw[->,blue] (f22k1) -- node[midway,below right,black] {$1$} (f122k1);
\draw[->,red] (f12k1) -- node[midway,below left,black] {$2$} (f122k1);
\draw[->,blue] (f122k1) -- node[midway,right,black] {$1$} (f1122k1);
\draw[->,red,dashed] (f112) .. controls (2,-5) and (4,-5) .. node[midway,below,black] {$2$} (f1122k1);
\draw[->,red,dashed] (k1) -- node[midway,above,black] {$2$} (k21);
\draw[->,red] (k21) -- node[midway,right,black] {$2$} (f2k21);
\draw[->,blue] (f2k21) -- node[midway,right,black] {$1$} (f12k21);
\draw[->,blue,dashed] (f2k1) -- node[midway,above,black] {$1$} (k21);
\draw[->,red,dashed] (f1k1f2) .. controls (2.7,-2) and (2.25,-4) .. node[midway,above right,black] {$2$} (f12k21);
\draw[->,red,dashed] (k12) -- node[midway,below left,black] {$2$} (k212);
\draw[->,blue,dashed] (f2k21) -- node[pos=0.7,below right,black] {$1$} (k212);
\end{tikzpicture}
}
\]
\vspace{-40pt}
\caption{A weak K-crystal structure on $\svssyt^3(\Lambda_2  + \Lambda_1)$, where $X = \{1,2,3\}$, for $w_0 = s_1 s_2 s_1$.}
\label{fig:K_crystal_21_121}
\end{sidewaysfigure}

A possible construction of a K-crystal on $\svssyt^n(\lambda)$ is to define a notion of tensor products of K-crystals. Such a tensor product rule should have connected components whose characters are Grothendieck polynomials.
Then one could take a connected component of $\svssyt^n(\Lambda_1)^{\otimes \lvert \lambda \rvert}$ containing a minimal highest weight element of weight $\lambda$.

One approach to showing such the tensor product rule would be use Equation~\eqref{eq:grothendieck_pieri_rule} and construct a bijection
\begin{equation}
\label{eq:pieri_Krystal_rule}
\svssyt^n(\Lambda_1) \otimes \svssyt^n(\lambda) \iso \bigoplus_{\nu} \svssyt^n(\nu),
\end{equation}
where $\nu / \lambda$ is a single box and $\ell(\nu) \leq n$, that we consider as a K-crystal isomorphism.
By using~\cite[Thm.~5.4]{Buch02}, the minimal highest weight elements should have a reading word that is highest weight except we read a box in \emph{increasing} order, as opposed to decreasing order in how we construct the crystal operators. Therefore, the minimal highest weight elements should be of the form
\[
\boxed{i_1 < \cdots < i_k} \otimes T_{\lambda}
\]
where $i_1, \dotsc, i_k$ are rows with addable corners in $\lambda$. The bijection in~\eqref{eq:pieri_Krystal_rule} would be given by the insertion $S \xleftarrow{B} T$ defined in~\cite{Buch02}; more explicitly, the bijection would be given by $S \otimes T \mapsto (S \xleftarrow{B} T)$.

\begin{prob}
\label{prob:tensor_products}
Construct a tensor product rule for (weak) K-crystals.
\end{prob}

Recall that the insertion given in~\cite{Buch02} cannot be considered as an associative product (unlike RSK); see~\cite[Ex.~4.3]{Buch02}. Therefore, we do not expect a tensor product rule to necessarily be associative. Note that a solution to Open Problem~\ref{prob:K_jdt_svt} would likely help as we will want K-jdt to be a K-crystal isomorphism. Furthermore, the $U_q(\fsl_n)$-crystal structure on $\svssyt^n(\lambda)$ is not local in the sense that it only changes one box. In contrast, the tensor product rule should only change one factor. It would be very interesting (and represent progress towards Problem~\ref{prob:tensor_products}) to determine an K-theoretic analog of the quantum group with representations whose characters are Grothendieck polynomials.

Another possible construction would be to use another combinatorial model such as excited Young diagrams. In Remark~\ref{rem:type_2_operators}, we saw that the K-crystal operators given in Definition~\ref{def:K_crystal_ops} correspond to an Type~$2$ excitation. Recall that we had to introduce a Type~$1'$ elementary excitation in order to construct the $U_q(\fsl_n)$-crystal structure. Moreover, from Figure~\ref{fig:K_crystal_21_121}, there is at least one other elementary excitation necessary to obtain the (weak) K-crystal structure:
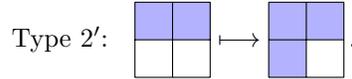
\[
\text{Type $2'$:} \quad
\begin{tikzpicture}[xscale=0.5,yscale=-0.5,baseline=-17]
\fill[blue!30] (0,0) rectangle (2,1);
\draw[step=1] (0,0) grid (2,2);
\end{tikzpicture}
\longmapsto
\begin{tikzpicture}[xscale=0.5,yscale=-0.5,baseline=-17]
\fill[blue!30] (0,0) rectangle (2,1);
\fill[blue!30] (0,1) rectangle (1,2);
\draw[step=1] (0,0) grid (2,2);
\end{tikzpicture}\,.
\]
However, it is not clear that this is the only additional excitation required to obtain a (weak) K-crystal on $\svssyt^n(\lambda)$.

We note that the K-crystal operator $f_i^K$ is likely to act on cyclically decreasing factorizations given in Section~\ref{sec:hecke} by using the relation $pp \equiv p$ and bringing the $p$ into the $i+1$ factor to the left. As seen in~\eqref{eq:bad_tableaux1}, the na\"ive choice does not always work. Yet, on those examples, the $f_i^K$ are well-behaved under Hecke insertion:
\begin{align*}
\left(\ytableaushort{13,2}\,, \ytableaushort{1{1,\!3},3} \right)
& \xleftarrow[\hspace{20pt}]{H}
\begin{bmatrix}
1 & 1 & 3 & 3 \\
3 & 1 & 3 & 2
\end{bmatrix}
\longleftrightarrow (3 \, 1) () (3 \, 2),
\allowdisplaybreaks \\
\left(\ytableaushort{13,2}\,, \ytableaushort{1{1,\!3},{2,\!3}} \right)
& \xleftarrow[\hspace{20pt}]{H}
\begin{bmatrix}
1 & 1 & 2 & 3 & 3 \\
3 & 1 & 3 & 3 & 2
\end{bmatrix}
\longleftrightarrow (1 \, 3) (3) (3 \, 2),
\allowdisplaybreaks \\
\left(\ytableaushort{13,2}\,, \ytableaushort{1X,3} \right)
& \xleftarrow[\hspace{20pt}]{H}
\begin{bmatrix}
1 & 1 & 2 & 3 & 3 \\
3 & 1 & 1 & 3 & 2
\end{bmatrix}
\longleftrightarrow (3 \, 1) (1) (3 \, 2),
\allowdisplaybreaks \\
\left(\ytableaushort{13,2}\,, \ytableaushort{1X,{2,\!3}} \right)
& \xleftarrow[\hspace{20pt}]{H}
\begin{bmatrix}
1 & 1 & 2 & 2 & 3 & 3 \\
3 & 1 & 3 & 1 & 3 & 2
\end{bmatrix}
\longleftrightarrow (3 \, 1) (3 \, 1) (3 \, 2),
\end{align*}
where $X = \{1, 2, 3\}$.

We remark that our definition of a (weak) K-crystal is ad-hoc based on natural properties that we expect to see in a K-theory analog of crystals. We note that from our expected properties (H.1)--(H.2), the K-crystal operators should satisfy the following axioms in addition to the usual Stembridge crystal axioms:
\begin{enumerate}
\item $f_i^K b = b'$ if and only if $b = e_i^K b'$ for all $b,b' \in B$;
\item $f_i^K f_i^K b = 0$ for all $b \in B$;
\item for all $b \in B$ such that $f_i^K b \neq 0$:
\begin{enumerate}
  \item $\wt(f_i^K b) = \wt(b) +\Lambda_{i+1} - \Lambda_i$;
  \item $\varepsilon_i(f_i^K b) = \varepsilon_i(b) = 0$;
  \item $\varphi_i(f_i^K b) = \varphi_i(b) - 1$;
\end{enumerate}
\end{enumerate}
Moreover, from the examples, there appears to be some additional local structure that the (K-)crystal operators satisfy (see also the proof of Proposition~\ref{prop:demazure_strings}). It would be interesting to determine if there is a K-theory analog of the Stembridge axioms~\cite{Stembridge03}.

One other potential approach (of independent interest) would be to use dual (symmetric) Grothendieck polynomials, which form the basis that is dual to the Grothendieck polynomials under the Hall inner product. One could develop a \emph{dual} K-crystal structure on reverse plane partitions, which are used to describe the dual Grothendieck polynomials~\cite{LamPyl07}. Indeed, there is a $U_q(\fsl_n)$-crystal structure on reverse plane partitions given by P.~Galashin~\cite{Galashin17}; this crystal has the same flavor as the $U_q(\fsl_n)$-crystal structure on set-valued tableaux. From there, the structure coefficients $C_{\lambda\mu}^{\nu}$ appear as the decomposition of a skew dual Grothendieck polynomial into dual Grothendieck polynomials~\cite{LMS16} (note that the notion of Yamanouchi set-valued tableaux in~\cite{LMS16} differs from our notion). The results in~\cite{KM17} may also provide some additional insight into this problem.

\appendix
\section{Calculations using \textsc{SageMath}}
\label{sec:sage}

We give some code using \textsc{SageMath}~\cite{sage} that we used to compute examples.
Our full \textsc{SageMath} code can be found at~\cite{code}.

We first give our helper functions. We denote the value $\beta$ by $q$.

\begin{lstlisting}
def construct_coordinates(n, var='x'):
    S = PolynomialRing(ZZ, 'q').fraction_field()
    R = PolynomialRing(S, [var+str(i) for i in range(1,n+2)])
    return R
\end{lstlisting}
\begin{lstlisting}
def partial(f, i):
    if isinstance(i, (list, tuple)):
        ret = f
        for k in reversed(i):
            ret = partial(ret, k, beta)
        return ret
    R = f.parent()
    names = R.variable_names()
    g = R.gens()
    fs = f.subs(**{names[i-1]: g[i], names[i]: g[i-1]})
    return (f - fs) // (g[i-1] - g[i])
\end{lstlisting}
\begin{lstlisting}
def pi(f, i):   # The Demazure operator
    if isinstance(i, (list, tuple)):
        ret = f
        for k in reversed(i):
            ret = pi(ret, k, beta)
        return ret
    return partial(R.gen(i-1) * f, i)
\end{lstlisting}
\begin{lstlisting}
def DL(f, i, beta):   # The Demazure-Lascoux operator
    if isinstance(i, (list, tuple)):
        ret = f
        for k in reversed(i):
            ret = DL(ret, k, beta)
        return ret
    return pi((1 + beta * R.gen(i)) * f, i)
\end{lstlisting}

We perform a sample computation of $\varpi_i(x_1^3 x_2)$ for $i = 1,2,3$:

\begin{lstlisting}
sage: R = construct_coordinates(4)
sage: R.inject_variables()
Defining x1, x2, x3, x4, x5
sage: q = R.base_ring().gen()
sage: DL(x1^3*x2, 1, q)
q*x1^3*x2^2 + q*x1^2*x2^3 + x1^3*x2 + x1^2*x2^2 + x1*x2^3
sage: DL(x1^3*x2, 2, q)
q*x1^3*x2*x3 + x1^3*x2 + x1^3*x3
sage: DL(x1^3*x2, 3, q)
x1^3*x2
\end{lstlisting}

Next, we construct $\varpi_w(x_1^2 x_2)$ for $w = s_1 s_2, s_2 s_1$:

\begin{lstlisting}
sage: DL(x1^2*x2, [2,1], q)
q^2*x1^2*x2^2*x3 + q^2*x1^2*x2*x3^2 + q*x1^2*x2^2 + 2*q*x1^2*x2*x3
 + q*x1*x2^2*x3 + q*x1^2*x3^2 + q*x1*x2*x3^2 + x1^2*x2 + x1*x2^2
 + x1^2*x3 + x1*x2*x3 + x1*x3^2
sage: DL(x1^2*x2, [1,2], q)
q^2*x1^2*x2^2*x3 + q*x1^2*x2^2 + 2*q*x1^2*x2*x3 + 2*q*x1*x2^2*x3
 + x1^2*x2 + x1*x2^2 + x1^2*x3 + x1*x2*x3 + x2^2*x3
\end{lstlisting}

After loading our code into \textsc{SageMath}, we can compute the Buch insertion to verify Table~\ref{table:K_jdt_ex} agrees with insertion in~\cite{Buch02}.

\begin{lstlisting}
sage: L = [[{1}],[{2}],[{3}],[{1,2}],[{1,3}],[{2,3}],[{1,2,3}]]
sage: ascii_art([[buch_insertion(b[0], [a]) for a in L] for b in L])
\end{lstlisting}

\bibliographystyle{alpha}
\bibliography{crystals}{}
\end{document}